\documentclass[a4paper,10pt]{article} 
\usepackage[english]{babel}
\usepackage{a4wide}
\usepackage[utf8]{inputenc}
\usepackage{amssymb,amsmath,amscd,amsfonts,amsthm,bbm,mathrsfs,enumerate}
\usepackage[colorlinks=true, linkcolor=blue, citecolor=blue]{hyperref}
\usepackage{graphicx} 
\usepackage{color} 
\usepackage{accents} 
\usepackage{subcaption}

\DeclareUnicodeCharacter{00A0}{~}

\theoremstyle{definition}
\newtheorem{theorem}{Theorem}
\newtheorem*{theorem*}{Statement}
\newtheorem{lemma}[theorem]{Lemma}
\newtheorem{corollary}[theorem]{Corollary}
\newtheorem{proposition}[theorem]{Proposition}

\newtheorem{remark}{Remark}

\newtheorem*{condition*}{Condition}
\newtheorem{assumption}{Assumption}

\DeclareMathOperator{\var}{\mathbb Var}

\DeclareMathOperator{\rank}{rank}
\DeclareMathOperator{\tr}{tr}

\DeclareMathOperator{\dist}{dist}
\DeclareMathOperator{\colspan}{span}

\DeclareMathOperator{\sign}{sign}
\DeclareMathOperator{\jac}{Jac}

\newcommand{\1}{\mathbbm 1}
\newcommand{\T}{{\mathsf T}} 

\newcommand{\ZZ}{\mathbb Z}
\newcommand{\CC}{\mathbb{C}}
\newcommand{\PP}{{{\mathbb P}}} 
\newcommand{\EE}{{{\mathbb E}}} 
\newcommand{\NN}{{{\mathbb N}}} 
\newcommand{\RR}{{{\mathbb R}}} 
\newcommand{\SSS}{{{\mathbb S}}}

\newcommand{\cD}{{\mathcal D}} 
\newcommand{\cE}{{\mathcal E}} 
\newcommand{\cI}{{\mathcal I}} 
\newcommand{\cJ}{{\mathcal J}} 
\newcommand{\cK}{{\mathcal K}} 
\newcommand{\cL}{{\mathcal L}} 
\newcommand{\cN}{{\mathcal N}} 
\newcommand{\cO}{{\mathcal O}} 
 
\newcommand{\cX}{{\mathcal X}} 
\newcommand{\cU}{{\mathcal U}} 

\newcommand{\Eop}{{\mathcal E}_{\text{op}}} 
\newcommand{\HS}{{\text{HS}}} 
\newcommand{\cpl}{{\text{c}}} 
\newcommand{\ssup}{{\bf s}_{\text{sup}}}
\newcommand{\sinf}{{\bf s}_{\text{inf}}}

\newcommand{\mom}{{\boldsymbol m}_4} 

\newcommand{\comp}{\text{comp}} 
\newcommand{\incomp}{\text{incomp}}

\newcommand{\ps}[1]{\langle #1 \rangle}

\newcommand{\bs}{\boldsymbol}
\DeclareMathOperator*{\diag}{diag}
\newcommand{\eqdef}{\triangleq} 
\newcommand{\toaslong}{\xrightarrow[n\to\infty]{\text{a.s.}}}

\begin{document}

\title{Smallest singular value and limit eigenvalue distribution \\ of a 
class of non-Hermitian random matrices \\
with statistical application}

% \subtitle{Do you have a subtitle?\\ If so, write it here}

% if too long for running head
% \titlerunning{The Big One} 

\author{Arup Bose\thanks{Statistics and Mathematics Unit, Indian Statistical 
Institute, Kolkata. Email: \texttt{bosearu@gmail.com}.}  
\and 
Walid Hachem\thanks{
CNRS / LIGM (UMR 8049), Universit\'e Paris-Est Marne-la-Vallée, France.
Email: \texttt{walid.hachem@u-pem.fr}.}} 

\date{September 25, 2019} 
%
% The correct dates will be entered by the editor

\maketitle

\begin{abstract} 
Suppose $X$ is an $N \times n$ complex matrix whose entries are centered, independent, and identically distributed random variables
with variance $1/n$ and whose fourth moment is of order ${\mathcal O}(n^{-2})$. In the first part of the paper, we consider the non-Hermitian matrix $X A X^* - z$, where $A$ is a deterministic matrix whose smallest and largest singular values are bounded below and above respectively, and $z\neq 0$ is a complex number. Asymptotic probability bounds for the smallest singular value of this matrix are obtained in the large dimensional regime where $N$ and $n$ diverge to infinity at the same rate. 

In the second part of the paper, we consider the special case where $A = J = [\1_{i-j = 1\mod n} ]$ is a circulant matrix. Using the result of the first part, it is shown that the limit spectral distribution of $X J X^*$ exists in the large dimensional regime, and we determine this limit explicitly. A statistical application of this result devoted towards testing the presence of correlations within a multivariate time series is considered. Assuming that $X$ represents a $\CC^N$-valued time series which is observed over a time window of length $n$, the matrix $X J X^*$ represents the one-step sample autocovariance matrix of this time series. Guided by the result on the limit spectral distribution  of this matrix, a whiteness test against an MA correlation model for the time series is introduced. Numerical simulations show the excellent performance of this test.  
\end{abstract}


\section{Introduction and the main results}
\label{sec-xjx} 

Let $(N^{(n)})_{n\geq 1}$ be a sequence of positive integers, which diverges to $\infty$ as $n \to \infty$.  Suppose 
$(X^{(n)} = [x_{ij}^{(n)} ]_{i,j=0}^{N^{(n)}-1,n-1})_{n\geq 1}$ is a sequence of complex random matrices whose entries satisfy the following assumptions:

\begin{assumption}
\label{ass-model}
For each $n\geq 1$, the complex random variables 
$\{x_{ij}^{(n)} \}_{i,j=0}^{N^{(n)}-1,n-1}$ are i.i.d.~with $\EE x_{00}^{(n)} = 0$, $\EE|x_{00}^{(n)} |^2 = 1/n$, and 
$\sup_{n}n^2 \EE |x_{00}^{(n)} |^4 \leq \mom <\infty$.
\end{assumption} 

Let $(A^{(n)})$ be a sequence of deterministic matrices such that $A^{(n)} \in \CC^{n\times n}$, and such that 
\begin{assumption}
\label{assA} 
\[
0 < \inf_n s_{n-1}(A^{(n)}) \leq \sup_n s_0(A^{(n)}) < \infty \, ,  
\]
where $s_0(M) \geq \cdots \geq s_{n-1}(M)$ will refer hereinafter to the singular values of the matrix $M \in \CC^{n\times n}$. 
\end{assumption} 

Suppose that $N^{(n)} / n \to \gamma$, $0 < \gamma  < \infty$ as $n\to\infty$. We shall first be interested in the behavior of the smallest singular value of the non-Hermitian matrix $X^{(n)} A^{(n)} {X^{(n)}}^* - z I_N$, where $z$ is an arbitrary non-zero complex number. We shall then use this result to obtain the limiting spectral behavior of the matrix $X^{(n)} J^{(n)} {X^{(n)}}^*$ where $J^{(n)}$ is given by Equation~\eqref{eq:Jn} below. Finally, we shall discuss a statistical application of this last result. \\ 

The behavior of the smallest singular value of large random matrices has recently aroused an intense research effort in the field of random matrix theory~\cite{tao-topics}. One of the main motivations for this interest is its close connections with the theory of the spectral behavior of large square non-Hermitian random matrices. It is indeed well-known that the probabilistic control of the smallest singular value of the matrix $Y - z$ is a key step towards understanding the behavior of the spectral measure of the matrix $Y$ \cite{bor-cha-12, tao-topics}.  Starting with the fundamental model where $Y$ has i.i.d.~elements, most of the contributions dealing with the smallest singular value assume the independence between the entries of $Y$, as seen in \cite{lit-paj-rud-tom-05,rud-ver-advmath08,got-tik-ap10,tao-vu-aop10,coo-18}
among many others. More structured models, such as the one dealt with in this paper, have  received comparatively much less attention. \\ 

Our results will be established under the following additional assumption on the elements of $X^{(n)}$. 

\begin{assumption}
\label{Xcomp}
The random variables $x_{00}^{(n)}$ satisfy 
$\sup_n | n \EE (x_{00}^{(n)})^2 | < 1$. 
\end{assumption} 

To understand the implication of Assumption~\ref{Xcomp}, suppose it does not
hold. Drop the superscript $^{(n)}$, and write 
$x_{00} = \Re x_{00} + \imath \Im x_{00}$. In that case, 
$1/n=|\EE x_{00}^2 | = \EE|x_{00}|^2$.
Expanding the expectations, this implies that $( \EE \Re x_{00} \Im x_{00})^2 =
\EE (\Re x_{00})^2 \EE (\Im x_{00})^2$.  Suppose for the moment, $\Re x_{00}
\not\equiv 0$. Then clearly $\Im x_{00} = \alpha \Re x_{00}$ w.p.1~for some
constant $\alpha$. Thus, 
$x_{00} \stackrel{{\mathcal L}}{=} \exp(\imath\theta) Z$, where $Z$ is a real
random variable and $\theta$ is a constant.  This amounts to $x_{00}$ being
real since the factor $\exp(\imath\theta)$ has no influence on $XJX^*$. Thus,
Assumption~\ref{Xcomp} essentially says that the $x_{ij}$ are not real. \\ 

We can now state our first result. We denote as $\| \cdot \|$ the spectral norm of a matrix. Events are expressed in the forms $[\ldots]$ or 
$\{\ldots\}$.  

\begin{theorem} 
\label{snA}
Let Assumptions~\ref{ass-model}, \ref{assA}, and~\ref{Xcomp} hold true. 
Then, there exist $\alpha, \beta > 0$ such that for each $C > 0$, 
$t >0$, and $z \in \CC \setminus \{ 0 \}$, 
\[
\PP\left[ s_{N-1}(X^{(n)}A^{(n)} {X^{(n)}}^* - z) \leq t, \ \| X \| \leq C \right] 
 \leq c \left( n^\alpha t^{1/2} + n^{-\beta} \right) , 
\] 
where the constant $c > 0$ depends on $C$, $z$, and $\mom$ only. 
\end{theorem}

To prove this theorem, the first step is to linearize the model 
$X^{(n)}A^{(n)} {X^{(n)}}^* - z$ by considering the matrix 
\[
H^{(n)} = \begin{bmatrix} {A^{(n)}}^{-1} & {X^{(n)}}^* \\ 
  X^{(n)} & z \end{bmatrix} 
\in \CC^{(N+n)\times(N+n)} . 
\]
By using, \emph{e.g.}, the inversion formula for partitioned matrices, 
it is easy to see that 
\[
\| (X^{(n)}A^{(n)} {X^{(n)}}^* - z)^{-1} \| \leq 
 \| {H^{(n)}}^{-1} \| 
\]
(versions of this ``linearization trick'' have been used in many different 
contexts, see, \emph{e.g.}, \cite{haa-tho-05}). 
Thus, the problem is reduced to controlling the smallest singular value of
$H^{(n)}$.  A similar problem was tackled in~\cite{ver-14} and \cite{ngu-12}.
In this paper, we follow closely the approach of ~\cite{ver-14}.  However,
there instead of $H^{(n)}$, the author had a real symmetric matrix with
i.i.d.~elements above the diagonal. Our matrix $H^{(n)}$ is more structured, 
and this necessitates a suitable modification in the arguments. 
Theorem~\ref{snA} will be proven in Section~\ref{sec-smallest}. \\ 

Theorem~\ref{snA} can be used to study the eigenvalue distribution of the matrix $X^{(n)} A^{(n)} {X^{(n)}}^*$ in the large dimensional regime
(see~\cite{bor-cha-12} or Section~\ref{sec:prf-main} below for more explanations on this connection). Motivated by the statistical application described in Section~\ref{stats}, we shall restrict our study in this paper to the specific case where $A^{(n)}$ equals the circulant matrix 
\begin{equation}
\label{eq:Jn} 
J^{(n)} = \begin{bmatrix}  
0 &        &        &  1 \\
1 & \ddots &             \\
  & \ddots & \ddots      \\
  &        & 1      &  0  
\end{bmatrix} \in \RR^{n\times n} , 
\end{equation} 
This matrix satisfies Assumption~\ref{assA}, since it is orthogonal.  
Let $\{ \lambda_0^{(n)},\ldots, \lambda_{N^{(n)}-1}^{(n)} \}$ be the 
eigenvalues of the matrix $X^{(n)}J^{(n)} {X^{(n)}}^*$, which are in 
general complex-valued. The spectral distribution or measure of this matrix
is defined as the random probability measure: 
\[
\mu_n = \frac{1}{N^{(n)}} \sum_{i=0}^{N^{(n)}-1} 
  \delta_{\lambda_i^{(n)}}. 
\]
Given a sequence of random probability measures $(\zeta_n)$ on the space 
$\cX = \RR$ or $\CC$ and a deterministic probability measure $\bs\zeta$ on 
$\cX$,
we recall that $\zeta_n$ is said to converge weakly in the almost sure sense
(resp.~in probability) if for each continuous and bounded real function 
$\varphi$ on $\cX$, 
\[
\int \varphi \, d\zeta_n \xrightarrow[n\to\infty]{} \int \varphi \, d{\bs\zeta}  
\quad \text{almost surely (resp. in probability).}  
\]
This weak convergence will be denoted as $\zeta_n \Rightarrow \bs\zeta$ 
a.s.~(resp. in probability). 

In the asymptotic regime where $N^{(n)} / n \to \gamma$, $0 < \gamma  <
\infty$, we shall identify a deterministic probability measure $\bs \mu$ such
that $\mu_n \Rightarrow \bs\mu$ in probability.  This limit $\bs\mu$ is called
the \textit{limiting spectral distribution or measure} (LSD) of the sequence of
matrices.  
To state our result regarding this LSD, we need 
the following function. For any $ 0 < \gamma < \infty$, let  
\begin{equation}
\label{functiong}
% \label{def-g} 
g(y) = \frac{y}{y+1} (1 - \gamma + 2y)^2, \  (0 \vee (\gamma-1)) \leq y \leq \gamma.  
\end{equation}
Then $g^{-1}$ exists on the interval 
$[ 0 \vee ((\gamma-1)^3 / \gamma), \gamma(\gamma+1) ]$ and maps it to 
$[ 0 \vee (\gamma-\gamma^{-1}), \gamma]$. It is an analytic increasing function on 
the interior of the interval.

\begin{theorem}
\label{th:main}
Suppose Assumptions~\ref{ass-model} and \ref{Xcomp} hold. Then, there exists
a deterministic probability measure $\bs\mu$ such that 
$\mu_n \Rightarrow \bs\mu$ in probability. The limit  measure $\bs\mu$ is rotationally invariant on $\CC$. 
Let 
$F(r) = \bs\mu(\{ z \in \CC \, : \, |z| \leq r \}), \ 0\leq r < \infty$ 
be the distribution function of the radial component. 

\noindent If $\gamma \leq 1$, then 
\[
F(r) = \left\{\begin{array}{cl} 
 \gamma^{-1} g^{-1}(r^2) & \text{if} \  0 \leq r \leq \sqrt{\gamma(\gamma+1)}, \\ \\
 1 &\text{if} \ r > \sqrt{\gamma(\gamma+1)}. 
 \end{array}\right.
\]
If $\gamma > 1$, then 
\[
F(r) = \left\{ 
\begin{array}{cl} 
 1 - \gamma^{-1} & \text{if} \ 0 \leq r \leq (\gamma-1)^{3/2} / \sqrt{\gamma}, \\
 \\ \gamma^{-1} g^{-1}(r^2) 
   & \text{if} \ (\gamma-1)^{3/2} / \sqrt{\gamma} < r \leq  
   \sqrt{\gamma(\gamma+1)}, \\  \\
 1 & \text{if} \ r > \sqrt{\gamma(\gamma+1)} . 
 \end{array} \right. 
\]
\end{theorem} 

The  theorem implies that the support of $\bs\mu$ is the disc $\{z: |z|\leq
\sqrt{\gamma(\gamma+1)}\}$ when $\gamma\leq 1$,  and when $ \gamma > 1$, it is
the ring $\{z: (\gamma-1)^{3/2} / \sqrt{\gamma}\leq |z|\leq
\sqrt{\gamma(\gamma+1)}\}$  together with the point $\{0\}$ where there is a mass $1-\gamma^{-1}$.

Moreover, $F(r)$ has a positive and analytical density on the open interval 
$(0 \vee \sign(\gamma-1) |\gamma-1|^{3/2}/ \sqrt{\gamma}, 
\sqrt{\gamma(\gamma+1)})$. A closer inspection of $g$ shows that this density
is bounded if $\gamma \neq 1$. If 
$\gamma = 1$, then the density is bounded everywhere except when 
$r \downarrow 0$. 
A cumbersome closed form expression for $g^{-1}$ (and hence for $F(\cdot)$) can
be obtained by calculating the root of a third degree polynomial. For the
special case $\gamma=1$, $g^{-1}$ is given by 
\[
g^{-1}(t)=\frac{t^{1/3}}{2}
\left(\left[ 1+
\sqrt{1-\frac{t}{27}} \right]^{1/3} 
 + \left[ 1- \sqrt{1-\frac{t}{27}}\right]^{1/3} \right) , 
\quad  0 \leq t \leq 2. 
\]

As an illustration of these results, eigenvalue realizations corresponding to
the cases where $\gamma = 0.5$ and $\gamma = 2$ are shown in
Figure~\ref{fig-scatter}.  Plots of the functions $F(r)$ given in the statement
of Theorem~\ref{th:main} are shown on Figure~\ref{fig-DF}, along with their
empirical counterparts.

\begin{figure}[ht]
\centering
\begin{subfigure}{.48\textwidth}
\centering 
\includegraphics[width=1.2\linewidth]{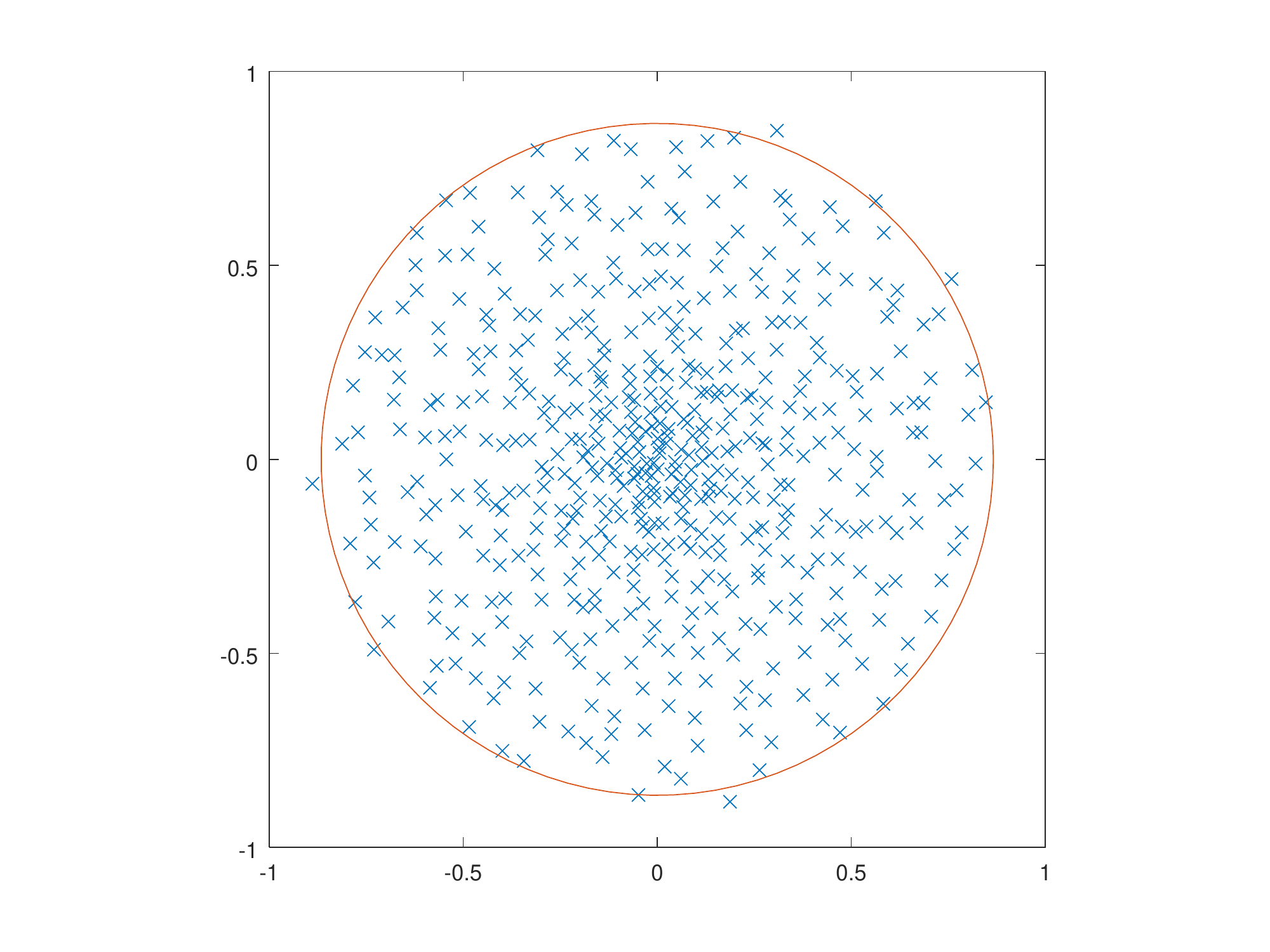}
\caption{$(N,n)=(500,1000)$} 
\end{subfigure} 
\begin{subfigure}{.48\textwidth}
\centering 
\includegraphics[width=1.2\linewidth]{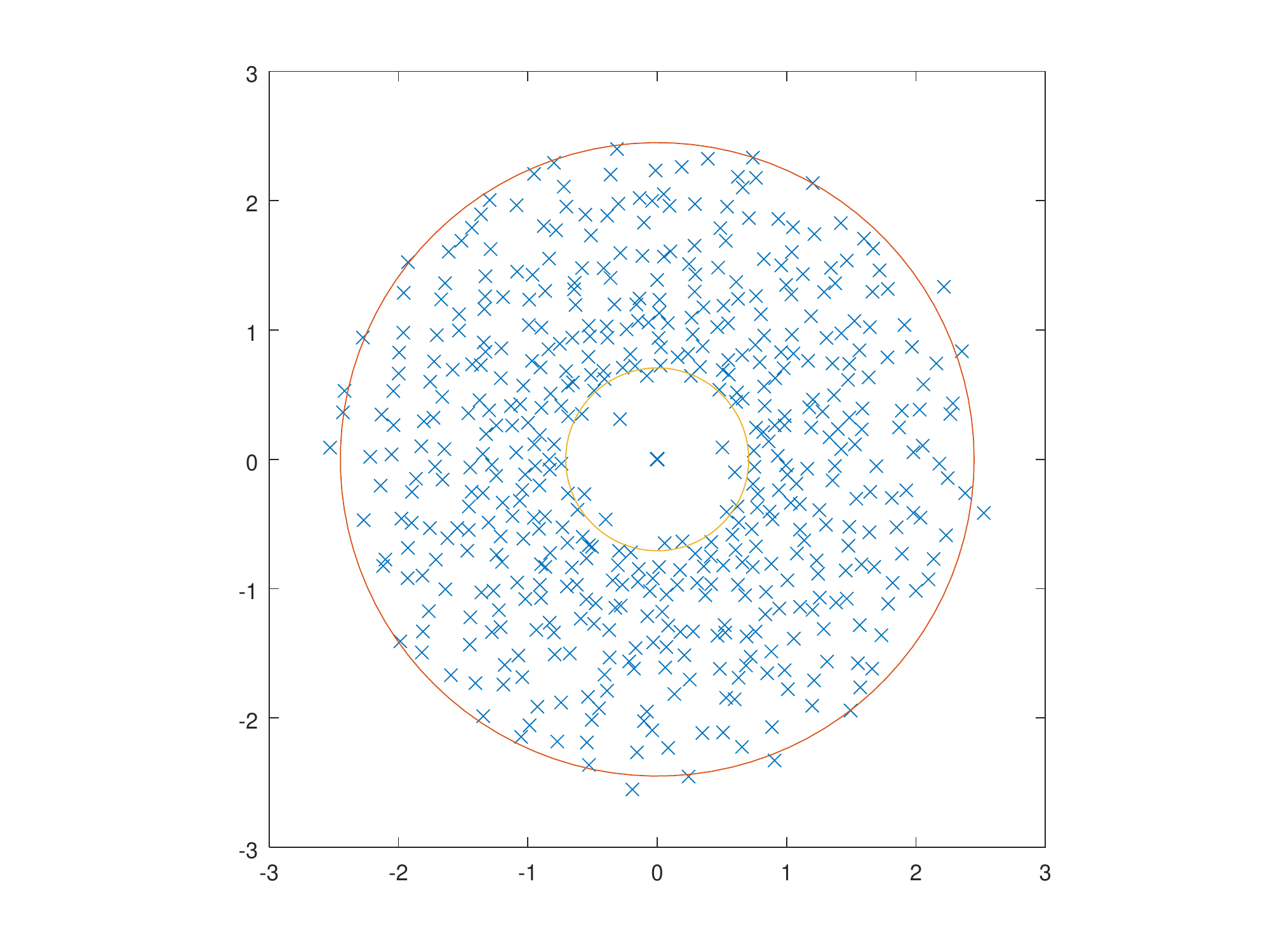}
\caption{$(N,n)=(1000,500)$} 
% \caption
\end{subfigure} 
\caption{Eigenvalue realizations and LSD support.}  
\label{fig-scatter} 
\end{figure}

\begin{figure}[ht]
\centering
\begin{subfigure}{.48\textwidth}
\centering 
\includegraphics[width=1.1\linewidth]{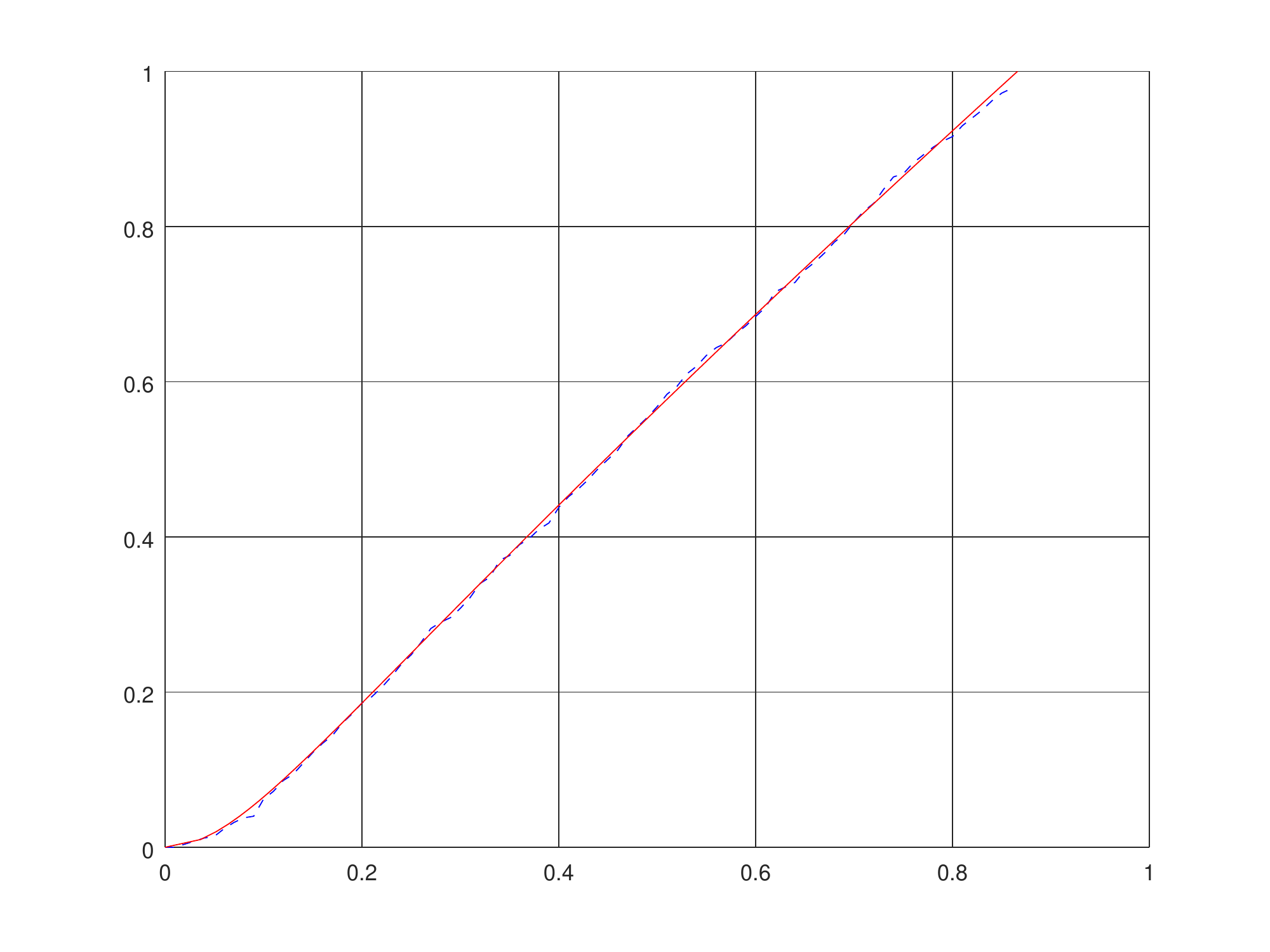}
\caption{$(N,n)=(500,1000)$} 
\end{subfigure} 
\begin{subfigure}{.48\textwidth}
\centering 
\includegraphics[width=1.1\linewidth]{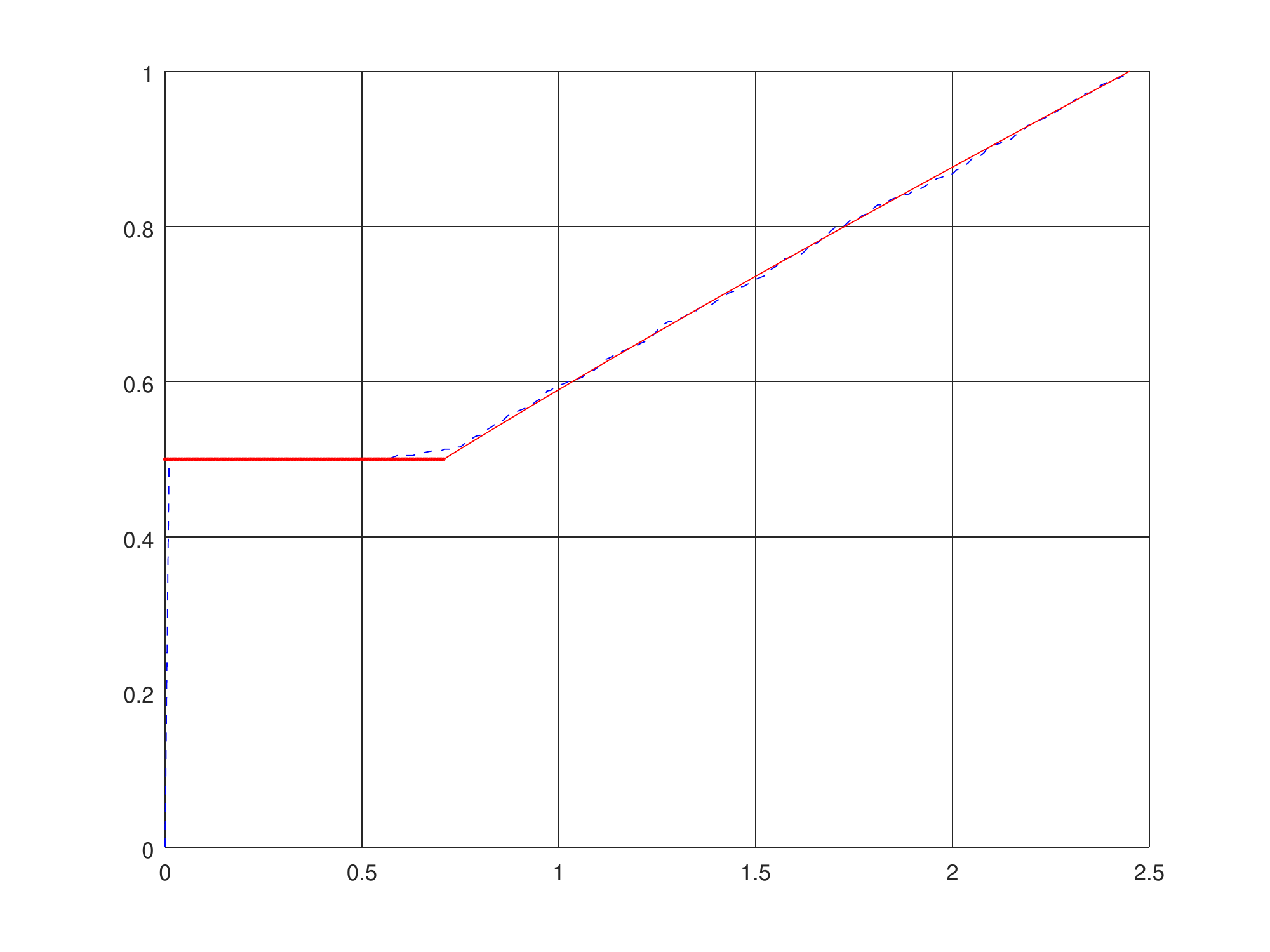}
\caption{$(N,n)=(1000,500)$} 
% \caption
\end{subfigure} 
\caption{Plots of $F(r)$ (plain curves) and their empirical realizations
 (dashed curves).}  
\label{fig-DF} 
\end{figure}

Theorem~\ref{th:main} is proven in 
Sections~\ref{sec:prf-main} -- \ref{sec:prop-mu}. Before turning to the proofs
of Theorems~\ref{snA} and~\ref{th:main}, we consider in the next section
a statistical application of these results. 

\section{Application to statistical hypothesis testing}
\label{stats} 

Consider the high dimensional linear moving average time series model
\begin{equation}
\label{timeseries} 
y_{t}^{(n)} =\sum_{i=0}^p B_i^{(n)} w_{t-i}^{(n)}, 
\end{equation} 
where $\{B_i^{(n)}\}_{i=0}^p$ are $\CC^{N\times N}$ deterministic parameter 
matrices, and $\{w_{i}^{(n)}\}_i$ are random vectors such that the random 
matrix $W^{(n)} = [ w_{0}^{(n)} \ \cdots \ w_{n-1}^{(n)} ]$ is equal in 
distribution to $n^{1/2} X^{(n)}$. Such models have found increasing attention 
in, \emph{e.g.}, the fields of signal processing,  wireless communications, 
Radar, Sonar, and wideband antenna array processing 
\cite{haykin1992adaptive,van-(livre)02}. 
The sample autocovariance matrices
$\{n^{-1}  \sum_{t=k+1}^n y_{t}^{(n)} (y_{t-k}^{(n)})^*\}, k \geq 0$, ($k$ is called the lag or the step) carry 
useful information about the model \eqref{timeseries}, specially through their 
spectral distributions. Some of the works that deal with limit spectral 
distributions, mostly for high-dimensional real-valued time series, and
their use in statistical inference are, \cite{BB2013a, BB2014free,
BB2015freey0, BB2018,  LAP2013, WAP2015, li-lam-yao-yao-(arxiv)18, li-li-yao-18, BB2019}. 

The $k$-step sample autocovariance matrices, except for the order $k=0$, are non-Hermitian. LSD results
are so far known only for certain symmetrized versions of these matrices. All
the references cited above rely on this idea of symmetrization.  To the best of
our knowledge, no LSD results are known for the non-Hermitian sample
autocovariance matrices.  The result of Theorem~\ref{th:main} above is a
beginning towards deriving the LSD of the sample autocovariance matrices in the
general model \eqref{timeseries} by considering the simplest case where
$B_0^{(n)}=I_N$ and $p=0$. This will be called the \textit{white noise model}.  

Consider the problem of testing the white noise model against an MA correlated
model. To this end, we explore the idea of designing a test which is based on
the eigenvalue distribution of the one-step sample autocovariance matrix, in
contrast to more classical tests that are based on its singular value
distribution. A non-rigorous justification of this idea is that
when performing an eigenvalue-based test, we take advantage of the higher
sensitivity of the eigenvalues of a matrix with respect to perturbations as
compared to its singular values. 

Assuming for simplicity that $p=1$, our purpose is to test the null (white
noise) hypothesis {\bf H0}: $B_0^{(n)} = I, B_1^{(n)} = 0$ against the
alternative {\bf H1}: $B_0^{(n)} = I, B_1^{(n)} \neq 0$. Consider the one-step 
sample autocovariance matrix 
\[
\widehat R_1^{(n)} = \frac 1n \sum_{t=0}^{n-1} y_t^{(n)} {y_{t-1}^{(n)}}^*,  
\]
where the sum is taken modulo $n$, and observe that under {\bf H0}, this
matrix coincides with $X^{(n)} J^{(n)} {X^{(n)}}^*$. We shall consider the 
asymptotic regime where $n\to\infty$ and $N/n \to\gamma > 0$. 
By Theorem~\ref{th:main}, the spectral measure of $\widehat R_1^{(n)}$ 
converges weakly in probability to the measure $\bs\mu$. This suggests the 
use of a white noise test based on a distance between the spectral measure of
$\widehat R_1^{(n)}$ and $\bs\mu$. We consider herein a test based on the 
$2$-Wasserstein distance between these two distributions. 
For the sake of comparison, we also considered the more classical singular
value based test which consists in comparing  
$N^{-1} \tr \widehat R_1^{(n)} (\widehat R_1^{(n)})^*$ to a threshold. We 
denote these two tests as T1 and T2 respectively. 

To get a more complete picture of the problem, we also considered a third
test which is based on the eigenvalue distribution of the Hermitian sample 
covariance matrix 
\[
\widehat R_{0,1}^{(n)} = \frac 1n \sum_{t=0}^{n-1} 
\begin{bmatrix} y_t^{(n)} \\ y_{t-1}^{(n)} \end{bmatrix} 
\begin{bmatrix} {y_t^{(n)}}^* & {y_{t-1}^{(n)}}^* \end{bmatrix}. 
\]
Its spectral distribution is known to converge weakly almost surely under {\bf H0} to the
Marchenko-Pastur distribution $\text{MP}_{2\gamma}$ with parameter $2\gamma$
(see \cite{lou-jtp16}, which deals with the Gaussian case). This suggests the
use of the $2$-Wasserstein distance between the spectral measure of 
$\widehat R_{0,1}$ and $\text{MP}_{2\gamma}$. We denote the resulting test as T3. 

Figures~\ref{roc-eye} and \ref{roc-toeplitz} represent the ROC curves obtained for
these three tests.  The tests T1 and T3 were implemented by sampling $\bs\mu$
and $\text{MP}_{2\gamma}$ from the spectra of two large random matrices and by
using the \texttt{transport} library of the \texttt{R} software. 
For Figure~\ref{roc-eye}, $B_1^{(n)} = \alpha I_N$, 
while for Figure~\ref{roc-toeplitz}, the elements $b_{ij}$ of $B_1^{(n)}$ are
chosen as $b_{ij}  = \alpha' \exp(-8 |i-j| / N)$, where $\alpha$ and $\alpha'$
are non-zero real numbers. 

\begin{figure}[h]
\centering
\includegraphics[width=0.7\linewidth]{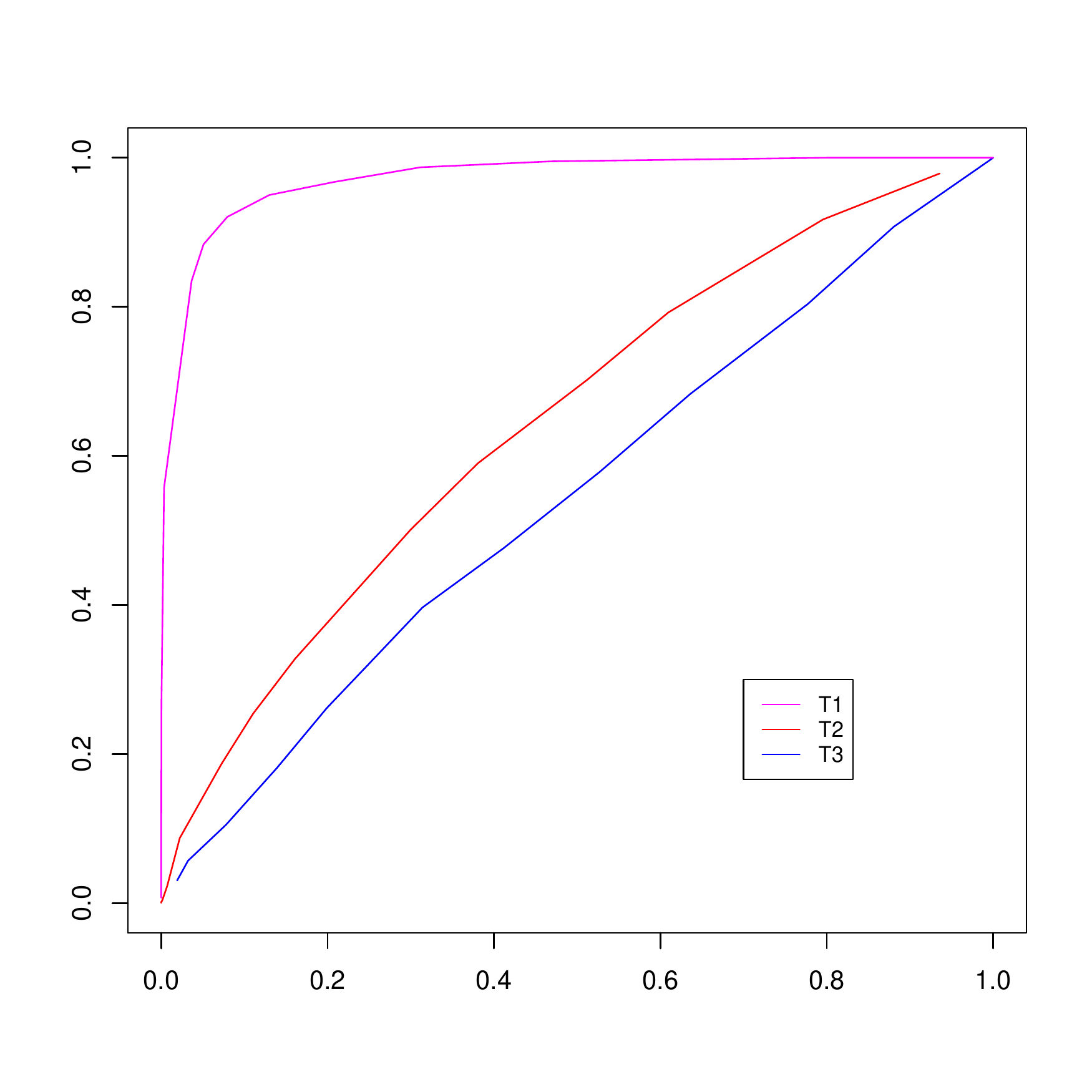}
\caption{ROC curves. Setting: $B_1^{(n)}  = \alpha I_N$ with 
 $\alpha^2 = 10^{-2.5}$, $(N,n) = (50,100)$.}  
\label{roc-eye} 
\end{figure}

\begin{figure}[h]
\centering
\includegraphics[width=0.7\linewidth]{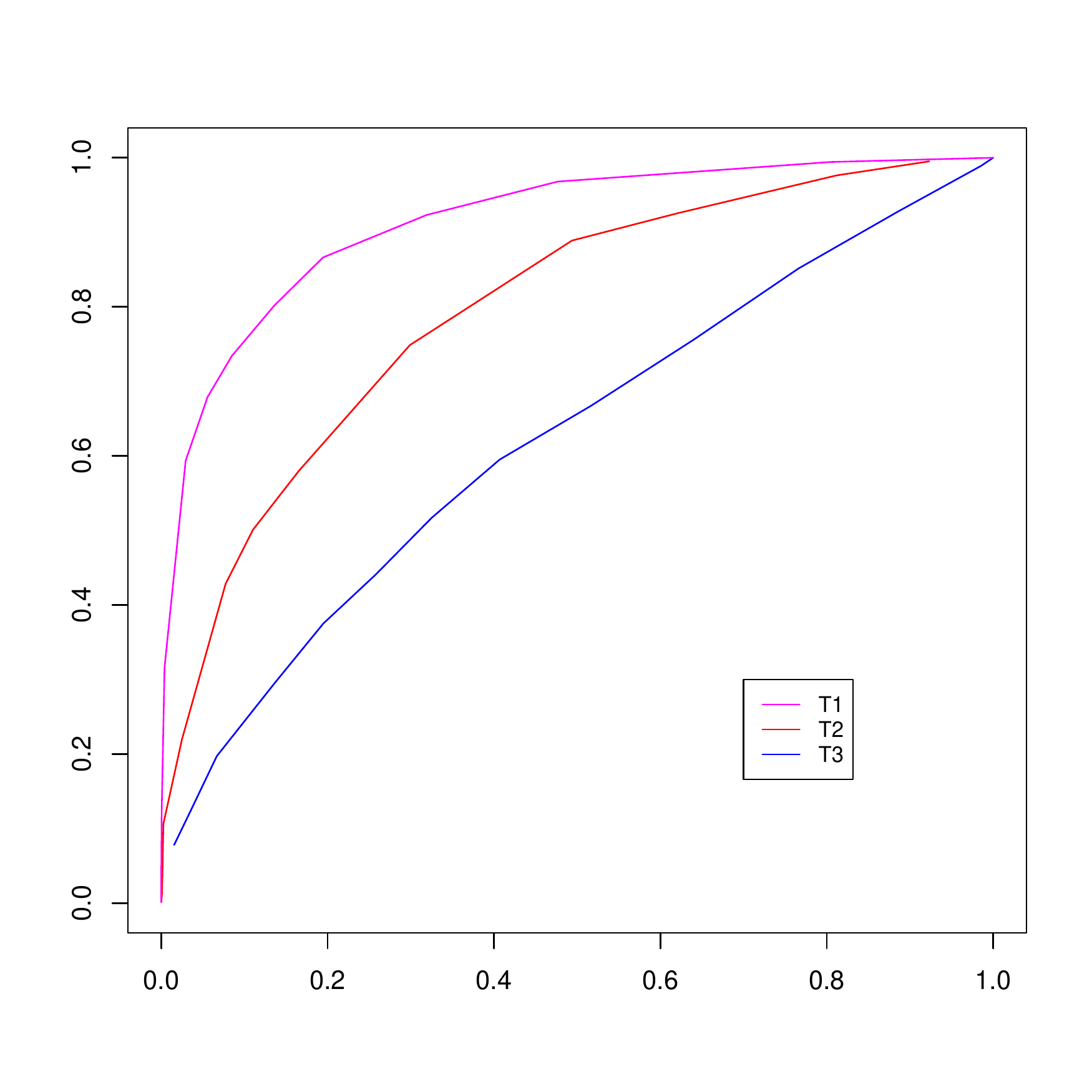}
\caption{ROC curves. Setting: 
  $B_1^{(n)}$ is a Toeplitz matrix with 
 $\tr B_1^{(n)} (B_1^{(n)})^* / N = 10^{-2}$, $(N,n) = (50,100)$.}  
\label{roc-toeplitz} 
\end{figure}

These figures clearly show that T1 outperforms T2 and T3. This tends to
corroborate the intuition that the eigenvalue sensitivity alluded to earlier, can be beneficial
when it comes to designing white noise tests. 

To better understand the behavior of
the eigenvalue-based tests, the next step would be to study the large dimensional
behavior of the spectral distribution of $\widehat R_1^{(n)}$ under {\bf H1}.
This appears to be quite non-trivial and is left for future research. 

\subsubsection*{Notations} 

The notations $\dim(V)$ and $V^\perp$ will refer to the dimension of the vector
subspace $V$, and the subspace orthogonal to $V$ respectively. The column span
of a matrix $M$ will be denoted as $\colspan(M)$. Similarly, $\colspan(V, d)$
is the span of the vector space $V$ and the vector $d$. 

The indices of the elements of a vector or a matrix start from zero. 
Given a positive integer $m$, we write $[m] = \{0,\ldots, m-1\}$.  For
$i\in[m]$, we denote as $e_{m,i}$ the $i^{\text{th}}$ canonical vector of
$\CC^m$, with $1$ at the $m$th place and $0$ elsewhere.  Given a matrix $M \in \CC^{m\times n}$ and two sets $\cI \subset [m]$
and $\cJ \subset [n]$, we denote as $M_{\cI,\cJ}$ the $|\cI| \times |\cJ|$
submatrix of $M$ that is obtained by retaining the rows of $M$ whose indices
belong to $\cI$ and the columns whose indices belong to $\cJ$. We also
write $M_{\cdot, k} = M_{[m], \{ k \}}$ and $M_{k, \cdot} = M_{\{ k \}, [n]}$.
We define as $\Pi_\cI :\CC^m \to \CC^m$ the projection operator such that 
$\Pi_\cI u$ is the vector obtained by setting to zero the elements of $u$ 
whose indices are in $\cI^\cpl$. We also denote as $u_\cI$ the vector of 
$\CC^{|\cI|}$ obtained by removing the elements of $u$ whose indices are in 
$\cI^\cpl$. When $M$ is a matrix, $\Pi_M$ refers to the orthogonal projector
on $\colspan(M)$. 

As mentioned above, $\|\cdot\|$ denotes the spectral norm. It will also 
denote the Euclidean norm of a vector. The Hilbert-Schmidt norm of a matrix 
will be denoted as $\| \cdot \|_\HS$. The unit-sphere of $\CC^n$ will be denoted 
as $\SSS^{n-1}$. 

The notations $\PP_x$ and $\EE_x$ will refer respectively to the probability
and the expectation with respect to the law of the vector $x$. 

\section{Proof of Theorem~\ref{snA}: smallest singular value}  
\label{sec-smallest}

To simplify the notations, from now on, we omit the superscript $^{(n)}$.
We shall mostly work on the matrix $A^{-1}$ instead of working on $A$. Writing
$\sinf = \inf_n s_{n-1}(A^{-1})$ and $\ssup = \sup_n s_0(A^{-1})$,
Assumption~\ref{assA} is rewritten as $0 < \sinf \leq \ssup < \infty$.
We also assume that $z\neq 0$ without further mention.

\subsection{General context and outline of proof} 
\label{sn-outline} 

We first observe that if we establish Theorem~\ref{snA} under the assumption that the entries have
densities, then it continues to hold in the general case. This is because we
can replace the matrix $X$ with, say, the independent sum  $(1 -
n^{-20})^{-1/2} (X + n^{-10} X')$ where $X'$ is a properly chosen matrix whose
elements have densities, and use a standard perturbation argument.  Hence,
\textit{we assume throughout this section that the elements of $X$ have
densities}. It may be noted that instead of 20, any other positive number could be used and
that would sharpen some of the bounds obtained later. However, it was not our goal to achieve sharp bounds. 

Suppose $E \in \CC^{N\times n}$ is such that $\| E \|^2 \| A \| < |z|$. Then
$\det ( z - E A E^*) \neq 0$. This implies that the multivariate polynomial
$\det ( z - X A X^*)$ in the variables $(\Re x_{ij}, \Im x_{ij} )_{i,j}$ is not
identically zero.  Since $X$ has a density, we conclude that $z - X A X^*$ is
invertible w.p.~1. 

Define the matrix 
\[
H = \begin{bmatrix} A^{-1} & X^* \\ X & z \end{bmatrix} 
\in \CC^{(N+n)\times(N+n)} . 
\]
By the well-known inversion formula for partitioned matrices
\cite[\S 0.7.3]{HorJoh90}, we have  
\[
H^{-1} = 
 \begin{bmatrix} 
A + A X^* (z - XAX^*)^{-1} X A & - A X^* (z - XAX^*)^{-1} \\ 
- (z - XAX^*)^{-1} X A & (z - XAX^*)^{-1}  \end{bmatrix} , 
\]
which shows that 
\[
\| (XAX^* - z)^{-1} \| \leq \| H^{-1} \| . 
\]
Therefore, to obtain Theorem~\ref{snA}, it is enough to prove that 
\begin{equation} 
\label{snH} 
\PP\left[ s_{N+n-1}(H) \leq t, \ \| X \| \leq C \right] 
 \leq c \left( n^\alpha t^{1/2} + n^{-\beta} \right) , 
\end{equation} 
where $c > 0$ depends on $C$, $z$, and $\mom$ only. 

As we mentioned in the introduction, a similar problem was considered
in~\cite{ver-14} and \cite{ngu-12}.  We shall follow here the argument
of~\cite{ver-14}. However, since our matrix $H$ is more structured than the one
considered in this reference, a substantial adaptation of the proof is
required.  Here is a description of the general approach.

First recall that 
\[
s_{N+n-1}(H) = \min_{u\in \SSS^{N+n-1}} \| H u \| . 
\]
Invoking an idea that has been frequently used in the literature since 
\cite{lit-paj-rud-tom-05,rud-ver-advmath08},  
we partition $\SSS^{N+n-1}$ into two sets of
\textit{compressible} and \textit{incompressible} 
vectors as follows. 

Let $\theta, \rho \in (0,1)$ be fixed. A vector in $\SSS^{N+n-1}$ is
said to be  \textit{$\theta$-sparse} if it does not have more than 
$\lfloor \theta (N+n) \rfloor$ non-zero elements.  Let $\SSS^{N+n-1}_\cI$ be 
the set of vectors of $\SSS^{N+n-1}$ that are supported by the (index) set 
$\cI \subset [N+n]$. Given $S \subset \CC^{N+n}$, let $\cN_\delta(S)$ denote 
the $\delta$-neighborhood of $S$ in $\CC^{N+n}$ in the Euclidean metric. 

Given $\theta, \rho \in (0,1)$, we define the set of 
$(\theta,\rho)$-compressible vectors as 
\[
\comp(\theta,\rho) = \SSS^{N+n-1} \cap 
\bigcup_{\substack{\cI \subset [N+n] \\ 
|\cI| = \lfloor \theta (N+n) \rfloor}} 
  \cN_\delta( \SSS_{\cI}^{N+n-1} ). 
\]
Note that this is the set of all unit vectors at a distance less or equal to 
$\rho$ from the set of the $\theta$-sparse unit vectors. The set 
$\incomp(\theta,\rho)$ of $(\theta,\rho)$-incompressible vectors 
is the complementary set $\SSS^{N+n-1} \setminus \comp(\theta,\rho)$. 

With these notations, we write 
\begin{equation} 
\label{comp-incomp} 
s_{N+n-1}(H) = 
 \inf_{u\in \comp(\theta, \rho)} \| H u \| 
 \ \wedge \ 
 \inf_{u\in \incomp(\theta, \rho)} \| H u \| 
\end{equation} 
for judiciously chosen $\theta,\rho \in (0,1)$.

The infimum over $\comp(\theta,\rho)$ is relatively easier to handle. Given a fixed vector $u \in \SSS^{N+n-1}$,
we first show that $\PP\left[ \| Hu \| \leq c \right]$ for some $c > 0$ is
exponentially small in $n$. Recall that an \textit{$\varepsilon$-net} is a set of 
points that are separated from each other by a distance of at most 
$\varepsilon$. Now, since the vectors of $\comp(\theta, \rho)$ are close to 
being sparse, it has an $\varepsilon$-net of controlled cardinality for a 
well-chosen $\varepsilon > 0$.  
Using this, along with a simple union bound, we will be able to infer  the 
smallness of the probability that  $\inf_{u\in \comp(\theta, \rho)} \| H u \|$ is small. 

The infimum over the set of incompressible vectors poses  a much bigger challenge since the 
$\varepsilon$-net argument fails. In this case 
the argument is more geometric. Observe that when 
$u$ is incompressible, $Hu$ is close to a sum of ${\mathcal O}(n)$ columns of 
$H$ with comparable weights. This helps to reduce the problem of 
controlling $\inf_{u\in \incomp(\theta, \rho)} \| H u \|$ to the problem of controlling the distance between an arbitrary column of $H$ and
the subspace generated by the other columns. 

Let $h_0$ be the first column of
$H$, and let $H_{-0} \in \CC^{(N+n)\times (N+n-1)}$ be the submatrix left after
extracting this column. Partition  $H$ accordingly as  
\[
% \begin{equation}
%\label{H-partition} 
H = \begin{bmatrix} b & g_{01} \\ g_{10} & G \end{bmatrix}, 
% \end{equation}
\]
with $b \in \CC$ and $G \in \CC^{(N+n-1) \times (N+n-1)}$. 
Then, the distance $\dist(h_0, H_{-0})$ between $h_0$ and
the column span of $H_{-0}$ equals ($G^{-1}$ will be shown to exist)
\begin{equation}\label{eq:column-distance}
\dist(h_0, H_{-0}) = \frac{\left| b - g_{01} G^{-1} g_{10} \right|} 
  {\sqrt{1 + \| g_{01} G^{-1} \|^2}}. 
\end{equation}
Our purpose is to 
bound the probability that this distance is small. If we write 
\[ 
A^{-1} = \begin{bmatrix} b & b_{01} \\ b_{10} & B \end{bmatrix},  
\quad 
X = \begin{bmatrix} x & W \end{bmatrix} , 
\]
where $B \in \CC^{(n-1)\times(n-1)}$, and $x\in\CC^N$ is the first column of 
$X$, then 
\begin{equation}\label{gdefine}
g_{01} = \begin{bmatrix} b_{01} & x^* \end{bmatrix} ,  \ 
g_{10} = \begin{bmatrix} b_{10} \\ x \end{bmatrix} ,  \ 
 \text{and} \ 
G = \begin{bmatrix} B & W^* \\ W & z I_N \end{bmatrix}. 
\end{equation}
Assuming inverse exists, partition  $G^{-1}$ as  
\begin{equation}\label{eq:ginverse}
G^{-1} = \begin{bmatrix} E & F \\ P & R \end{bmatrix},  
\quad E \in \CC^{(n-1)\times (n-1)}, \ R \in \CC^{N\times N}. 
\end{equation}
Then using Equation \eqref{eq:column-distance},  we have 
\begin{equation}
\label{eq:d-outline} 
\dist(h_0, H_{-0}) = \frac{\text{Num}}{\text{Den}}, 
\end{equation} 
where 
\begin{align} 
\label{eq:num-den} 
\begin{split} 
\text{Num} &= \left| b - b_{01} E b_{10} - x^* P b_{10} - b_{01} F x 
  - x^* R x \right| , \quad \text{and}  \\
\text{Den} &= \left(1 + \| b_{01} E + x^* P \|^2 + \| b_{01} F + x^* R \|^2 
               \right)^{1/2}. 
\end{split} 
\end{align} 

To control the behavior of Num, 
we need an \textit{anti-concentration} result. 
Loosely speaking, we 
show that conditionally on the matrix $W$ and for most of these matrices, 
the probability that a properly normalized version of the random variable
$x^* P b_{10} + b_{01} F x + x^* R x$ lives in an arbitrary ball of 
$\CC$ of small radius is itself small. 

Small ball probabilities are captured by the so-called \textit{Lévy's 
concentration function}. Given a constant vector $a \in \CC^n$ and a random 
vector $Z \in \CC^n$, Lévy's concentration function of the inner product 
$\ps{a,Z}$ at $\varepsilon > 0$ is  
\[
\cL_Z(\ps{a,Z}, \varepsilon) =
  \sup_{w \in \CC} \PP_Z\left[ \left| \ps{a,Z} - w \right| \leq \varepsilon 
   \right]. 
\]
When the elements of $Z$ are i.i.d.~random variables with finite third moment, 
the behavior of $\cL_Z$ can be controlled by the 
\textit{Berry-Esséen theorem}, whose use in random matrix theory dates
back to~\cite{lit-paj-rud-tom-05}. Berry-Esséen theorem is a refinement of
the Central Limit Theorem and implies that when $a\in\SSS^{n-1}$ has 
$\cO(n)$ elements with magnitudes of order $1/\sqrt{n}$, it holds that 
$\cL_Z(\ps{a,Z}, \varepsilon) \lesssim \varepsilon +  1 / \sqrt{n}$. 

Our plan now is to apply this theorem after replacing $Z$ with the random
vector $x$. Unfortunately, this theorem cannot be used as is on the random
variable $x^* P b_{10} + b_{01} F x + x^* R x$ because of the presence of the
quadratic form $x^* R x$. To circumvent this problem, we use a decoupling
argument that replaces $x^* P b_{10} + b_{01} F x + x^* R x$ with an inner
product that can be processed by the Berry-Esséen theorem.  This decoupling
idea that dates back to~\cite{got-79} has also been used in~\cite{ver-14}.

\subsection{Technical results} 
\label{subsec-technical}

The following proposition is a variation of \cite[Prop.~5.1]{tao-vu-aop10},
see also \cite[Lem.~A2]{bor-cha-12} and \cite{got-tikh-arxiv10}. 
This variation is needed because we
want the constants $c$ and $c'$ to depend on the law of the $Z_i$'s via 
$\kappa$ and $C_\kappa$. For completeness, we provide the modified proof in Appendix~\ref{prf-proj1}.   

\begin{proposition}[Distance of a random vector to a constant subspace]  
\label{proj1}
Let $Z = (Z_0,\ldots, Z_{n-1}) \in \CC^n$ be a vector of i.i.d.~centered 
unit-variance random variables such that for some $\kappa > 0$, $\EE | Z_0 |^{2+\kappa} \leq C_\kappa < \infty$. Then, there exist 
$c, c' > 0$ and $\delta \in (0,1)$ that depend only on $\kappa$ and $C_\kappa$ 
and that satisfy the following property.
For all $n \gg 1$, and for any deterministic subspace $V$ of $\CC^n$ such that 
$0\leq \dim (V) \leq \delta n$, 
\[
\PP [ \dist( Z, V) \leq c \sqrt{n} ] \leq \exp(-c' n ) .
\]
\end{proposition} 

We shall also make use of: 
\begin{lemma}[Rosenthal's inequality \cite{ros-70}] 
Let $Z_0, \ldots, Z_{n-1}$ be independent random variables such that 
$\EE Z_i = 0$ and $\EE |Z_i|^p < \infty$ for $p > 2$. Then there exists a
universal constant $C_p$ such that 
\[
\EE \Bigl| \sum_{i=0}^{n-1} Z_i \Bigr|^p \leq C_p \Bigl( 
\sum_{i=0}^{n-1} \EE |Z_i|^p \ \vee \ 
   \Bigl( \sum_{i=0}^{n-1} \EE |Z_i|^2 \Bigr)^{p/2} \Bigr) . 
\]
\end{lemma} 

These results easily lead to the following lemma: 
\begin{lemma}
\label{Xu} 
Let the matrix $X$ satisfy Assumption~\ref{ass-model}. Then, there exist 
constants $c, c' > 0$ and a constant $\delta \in (0,1)$ that depend on 
$\mom$ only and that satisfy the following property. For each 
deterministic vector $u \in \SSS^{n-1}$ and each deterministic subspace 
$V \subset \CC^N$ with $0\leq \dim(V) \leq \delta N$, 
\begin{equation}\label{distxuv}
\PP \left[ \dist(X u, V) \leq c \right] \leq \exp(- c' n).  
\end{equation}
In particular, for each deterministic vector $a \in \CC^N$, it holds that 
$\PP \left[ \| X u - a \| \leq c \right] \leq \exp(- c' n)$. Similar 
conclusions hold if $X$ is replaced with $X^*$. 
\end{lemma}

\begin{proof}
Let $\tilde x_0, \ldots, \tilde x_{N-1} \in \CC^{1\times n}$ be the rows of 
$X$, and define the random variables $Z_k = \sqrt{n} \tilde x_k u$ for 
$k \in [n]$. These random variables are i.i.d., centered, and have
unit-variance. Furthermore, writing 
$u = [u_0,\ldots, u_{n-1}]^\T$, we get by Rosenthal's inequality that for some universal constant $C$, 
\[
\EE | Z_1 |^4 \leq C \left( (n^2 \EE|X_{11}|^4 \sum |u_i|^4) \ \vee \ 1
    \right) \leq C \mom. 
\]
 Writing $Z = [Z_0,\ldots, Z_{N-1}]^\T$,
we note that $\dist(Xu, V) = \dist(Z,V) / \sqrt{n}$. Applying 
Proposition~\ref{proj1} with $\kappa = 2$, we obtain (\ref{distxuv}). The
rest of the claims follow immediately. 
\end{proof}  

The $\varepsilon$-net argument alluded to above will use the following lemma. 
\begin{lemma}[Metric entropy of a complex sphere, Lemma~2.2 of
\cite{coo-18}] 
\label{entrop} 
Let $V \in \CC^{n}$ be a $k$-dimensional subspace, and let $S\subset V$. Given 
$\varepsilon > 0$, the set $\SSS^{n-1} \cap S$ has an $\varepsilon$-net of 
cardinality bounded by $(3/\varepsilon)^{2k}$. 
\end{lemma}

The two following results regarding Lévy's concentration functions will 
be needed. 

\begin{lemma}[Restriction of the concentration function, Lemma~2.1 
 of~\cite{rud-ver-advmath08}] 
\label{lm:restrict}
Let $Z \in \CC^n$ be a vector of independent random variables. Then, for 
each non-empty $\cI \subset [n]$, we have 
$\cL_Z(\ps{a,Z}, \varepsilon) \leq  
  \cL_{Z_\cI}(\ps{a_{\cI},Z_\cI}, \varepsilon)$. 
\end{lemma} 

\begin{proposition}[Anti-concentration via the Berry-Ess\'een theorem] 
\label{prop:BE} 
There exists a constant $c > 0$ such that for any vector 
$Z = [Z_0, \ldots, Z_{n-1}]$ of complex centered independent random variables 
with finite third moments, 
\[
\cL_Z\Bigl( \sum Z_i, t\Bigr) \leq 
\frac{ct}{\sqrt{\sum \EE | Z_i |^2}} + 
 \frac{c \sum \EE | Z_i|^3}{(\sum \EE | Z_i |^2)^{3/2}} . 
\]
\end{proposition} 
For a proof, see \cite[Chap.~2]{tao-topics} or \cite[Lem.~A6]{bor-cha-12}. 
In particular, if there exist two positive constants $c_2$ and $c_3$ 
such that $\EE | Z_i |^2 \geq c_2$ and $\EE | Z_i |^3 \leq c_3$ for each
$i\in [n]$, then 
\begin{equation}
\label{eq:be-sqrt(n)} 
\cL\Bigl( \sum Z_i, t \sqrt{n}\Bigr) \leq 
 c' t + \frac{c''}{\sqrt{n}}, 
\end{equation} 
where $c' = c / \sqrt{c_2}$ and $c'' = c c_3 / c_2^{3/2}$.

We now enter the proof of Theorem~\ref{snA} via proving 
Inequality~\eqref{snH}. Recall that we have written  
$X = \begin{bmatrix} x & W \end{bmatrix}$ where 
$x$ is the first column of $X$. Given $C > 0$, we denote as $\Eop(C)$ the 
event  
\[
\Eop(C) = \left[ \| W \| \leq C \right] . 
\]
In the remainder of this section, the constants that do not depend on $n$ will
be referred to by the letter $c$, possibly with primes or numerical indices. 
In all statements of the type 
\[
\PP\left[ \left[ \cdots \leq c \right] \cap \cE \right] \leq 
  \exp(-c'n) + c_1 n^{-\gamma}  ,  
\]
where $\cE = [ \| X \| \leq C]$ or $\Eop(C)$, the constants such as $c$, $c'$, 
or $c_1$ depend on $C$, $z$, and $\mom$ at most. 

\subsection{Compressible vectors} 
\label{subsec-comp}

Recalling~\eqref{comp-incomp}, we start with the compressible vectors. 
The probability bound for these vectors is provided by the 
following proposition: 

\begin{proposition} 
\label{prop-comp} 
Let Assumption~\ref{ass-model} hold true. Then, there exists 
$\theta_{\ref{prop-comp}} \in (0,1)$, $\rho_{\ref{prop-comp}} > 0$, $c > 0$ 
and $c' > 0$ such that 
\[
\PP\left[ 
\Bigl[ \inf_{u\in \comp(\theta_{\ref{prop-comp}},\rho_{\ref{prop-comp}})} 
   \| H u \| \leq c \Bigr] \cap [\|X\| \leq C] \right] 
  \leq \exp(-c' n) \ \ \text{for large enough} \ \ n.
\]
 
\end{proposition} 
\begin{proof}
We first show that there exist $c_0, c_1 > 0$ such that for each deterministic 
vector $u \in \SSS^{N+n-1}$, 
\begin{equation}\label{hupointwise}
\PP\left[ \| H u \| \leq c_0 \right] \leq \exp(-c_1 n). 
\end{equation}
Let us partition $u$ as $u = [v^\T,w^\T]^\T$, where $v \in \CC^n$ and 
$w \in \CC^N$. Since $\| u\| =1$, either $\| v \| \geq 1/\sqrt{2}$ or 
$\| w \| \geq 1/\sqrt{2}$. Assume that $\| w \| \geq 1/\sqrt{2}$, and note that
$[\| H u \| \leq c_0 ] \subset [ \| A^{-1}v + X^* w \| \leq c_0 ]$. 
Writing $\tilde w = w / \| w \|$, we have 
\begin{align*} 
\PP\left[ \| X^* w + A^{-1}v \| \leq c_0 \right] &= 
\PP\left[ \| X^* \tilde w + A^{-1}v / \|w\| \| \leq c_0/ \| w\| \right] \\
 &\leq 
 \PP\left[ \| X^* \tilde w + A^{-1}v / \|w\| \| \leq c_0 \sqrt{2} \right] \\ 
 &\leq \exp(-c_1 n) 
\end{align*} 
by applying Lemma~\ref{Xu} and choosing  $c_0$ and $c_1$ judiciously. 
When $\| v \| \geq 1/\sqrt{2}$, we can use a similar argument 
(with possibly different $c_0$ and $c_1$) after observing 
that $[\| H u \| \leq c_0] \subset [ \| X v + zw \| \leq c_0 ]$. This establishes (\ref{hupointwise}). 

Now, on the event $[\| X \|\leq C]$, we have 
\[
\| H \| \leq \left\| \begin{bmatrix} & X^* \\ X \end{bmatrix} \right\| 
   + \left\| \begin{bmatrix} A^{-1} \\ & z \end{bmatrix} \right\| 
  \leq C_H \eqdef C + |z| \vee \ssup. 
\]
On this event, assume that there exists $y \in \cN_{c_0 / (2C_H)} ( \{ u \})$ 
such that $\| H y \| \leq c_0 / 2$. Then 
$\| H u \| \leq \| H(u-y)\| + \| Hy \| \leq c_0$. In other words, 
\[
\left[ \exists y \in \cN_{c_0 / (2C_H)} (\{ u \}) \, : \, 
  \| H y \| \leq c_0 / 2 \right] \cap [\| X \|\leq C] \subset 
  \left[ \| H u \| \leq c_0 \right].
\]
Now, let $\theta_{\ref{prop-comp}} \in (0,1)$ to be fixed in a moment, and 
choose $\cI \subset [N+n]$ in such a way that 
$|\cI | = \lfloor \theta_{\ref{prop-comp}} (n+N) \rfloor$. 
By Lemma~\ref{entrop}, the unit-sphere $\SSS_\cI^{N+n-1}$ of the subspace 
of the vectors of $\CC^{N+n}$ that are supported by $\cI$ has a 
$(c_0 / (2 C_H))$-net of cardinality bounded by $(6 C_H / c_0)^{2 |\cI|}$. 
Applying the previous results and making use of the union bound, we get that 
\[
\PP \left[ 
  \left[ \exists y \in \cN_{c_0 / (2C_H)} (\SSS^{N+n-1}_\cI) \, : \, 
  \| H y \| \leq c_0 / 2 \right] \cap [\| X \|\leq C] \right] 
 \leq (6 C_H / c_0)^{2 \theta_{\ref{prop-comp}} (N+n)} \exp(-c_1 n) .
\]
Finally, considering all the sets $\cI \subset [N+n]$ such that 
$|\cI| = \lfloor \theta_{\ref{prop-comp}}(N+n) \rfloor$, recalling the 
elementary bound on the binomial coefficients ${{m}\choose{k}} \leq (em/k)^k$, 
and using the union bound, we get that 
\[
\PP \left[ 
  \left[ \exists y \in \comp(\theta_{\ref{prop-comp}}, c_0 / (2C_H)) \, : \, 
  \| H y \| \leq c_0 / 2 \right] \cap [\| X \|\leq C] \right] 
 \leq \left(
 \frac{36 e C_H^2}{\theta_{\ref{prop-comp}} c_0^2}
    \right)^{\theta_{\ref{prop-comp}}(N+n)}  
   \exp(-c_1 n) .
\]
A small calculation shows that the right hand side is of the form $\exp(-c'n)$
for large enough $n$ when $\theta_{\ref{prop-comp}}$ is chosen small enough. By 
taking $\rho_{\ref{prop-comp}} = c_0 / (2 C_H)$, the proposition is proven.  
\end{proof} 

\subsection{Incompressible vectors} 
\label{subsec-incomp}

\subsubsection{Tools} 

One main feature of incompressible vectors of $\CC^n$ is that they contain 
$\cO(n)$ elements of absolute values of order 
$\cO(n^{-1/2})$, as shown in \cite[Lem.~3.4]{rud-ver-advmath08}. A slightly 
stronger version of this lemma will be needed in this paper: 
\begin{lemma}
\label{spread} 
Let $u  = [u_0, \ldots, u_{n-1}]^\T \in \incomp(\theta,\rho)$, and let 
$\tilde u = [\tilde u_0, \ldots, \tilde u_{n-1}]^\T \in \SSS^{n-1}$. Then the 
set 
\[
J = \{ i \in [n] \, : \, 
 \frac{\rho}{\sqrt{n}} \leq |u_i| \leq \frac{2}{\sqrt{\theta n}} 
  \ \text{and} \ |\tilde u_i| \leq \frac{2}{\sqrt{\theta n}}  \} 
\]
satisfies $|J|\geq \theta n / 2$. 
\end{lemma} 
\begin{proof}
Let 
\[
J_1 = \{ i \in [n] \, : \, |u_i| \leq \frac{2}{\sqrt{\theta n}} \} , \ 
J_2 = \{ i \in [n] \, : \, |\tilde u_i| \leq \frac{2}{\sqrt{\theta n}} \} , 
        \ \text{and} \ 
J_3 = \{ i \in [n] \, : \, |u_i| \geq \frac{\rho}{\sqrt{n}} \} . 
\] 
Since $\| u \| = \| \tilde u \| = 1$, we get by Tchebychev's inequality that
$| J_1^\cpl |, | J_2^\cpl | \leq \theta n / 4$. Moreover, 
$\| u - \Pi_{J_3} u \| = \| \Pi_{J_3^\cpl} u \| < \rho$ by the definition
of $J_3$. Recalling the definition of incompressibility, we get that 
$| J_3 | > \theta n$. Thus, 
$|J| = J_1 \cap J_2 \cap J_3 \geq n - |J_1^\cpl| - |J_2^\cpl| - |J_3^\cpl| 
 \geq \theta n / 2$. 
\end{proof} 

One consequence of \cite[Lem.~3.4]{rud-ver-advmath08} is the following lemma,
which implies that the infimum of $\| H u \|$ over a set of incompressible 
vectors can be handled by controlling the distance between an arbitrary column 
of $H$ and the subspace generated by the other columns: 

\begin{lemma}[Invertibility via mean distance, Lemma~3.5 
  of~\cite{rud-ver-advmath08}] 
\label{lm-distM} 
Let $M \in \CC^n$ be a random matrix. Let $m_k$ be the $k$th 
column of $M$ and let $M_{-k} \in \CC^{n\times (n-1)}$ be the submatrix left after
removing this column. Then, 
\[
\PP\left[ \inf_{u\in\incomp(\theta,\rho)} \| M u \| 
   \leq \frac{\rho t}{\sqrt{n}} \right] \leq 
  \frac{2}{\theta n} \sum_{k=0}^{n-1} 
     \PP\left[ \dist(m_k, M_{-k}) \leq t \right] .
\]
\end{lemma} 

An expression for these distances is provided next. 

\begin{lemma}
\label{lm-dist} 
Let $M\in\CC^{n\times n}$, and partition this matrix as 
\[
M = \begin{bmatrix} m_0 & M_{-0} \end{bmatrix} = 
 \begin{bmatrix} m_{00} & m_{01} \\ m_{10} & M_{11} \end{bmatrix} , 
\]
where $m_0$ and $M_{-0}$ are as in Lemma \ref{lm-distM}, $m_{00}$ is the first
element of the vector $m_0$, and $M_{11}$ is the bottom $(n-1) \times (n-1)$ 
submatrix of $M_{-0}$. Assume that $M_{11}$ is invertible. Then, 
\[
\dist(m_0, M_{-0}) = \frac{|m_{00} - m_{01} M^{-1}_{11} m_{10}|} 
     {\sqrt{1 + \| m_{01} M^{-1}_{11} \|^2}} . 
\]
\end{lemma}
\begin{proof}
We develop the expression 
$\dist(m_0, M_{-0})^2 = m_0^* \Pi_{M_{-0}}^\perp m_0$, where 
\[
 \Pi_{M_{-0}}^\perp = I - M_{-0} (M_{-0}^* M_{-0})^{-1} M_{-0}^*  . 
\]
is the orthogonal projector on $\colspan(M_{-0})^\perp$. 
Using the Sherman-Morrison-Woodbury formula,   
\[
(M_{-0}^* M_{-0})^{-1} = (m_{01}^* m_{01} + M_{11}^* M_{11})^{-1} = 
 M_{11}^{-1} M_{11}^{-*} - 
\frac{1}{1 + \| a \|^2} M_{11}^{-1} a^* a M_{11}^{-*}, 
\]
where $a = m_{01} M_{11}^{-1}$. We thus obtain after a small calculation that  
\[
\Pi_{M_{-0}}^\perp = \frac{1}{1+ \| a \|^2} 
   \begin{bmatrix} 1 & -a \\ -a^* & a^* a \end{bmatrix} = 
  \frac{1}{1+ \| a \|^2} 
   \begin{bmatrix} 1 \\ -a^* \end{bmatrix}  
   \begin{bmatrix} 1 & -a \end{bmatrix} .  
\]
Since $m_0 = \begin{bmatrix} m_{00} \\ m_{10} \end{bmatrix}$, we then 
get that $\dist(m_0, M_{-0})^2 = | m_{00} - a m_{10} |^2 / (1 + \| a \|^2)$, 
which is the required result. 
\end{proof}

\subsubsection{Distance control} 

Using Lemma~\ref{lm-distM}, we need to control the distance between a column of
$H$ and the subspace generated by the other columns. 

Denote as $x_k$ the $k^{\text{th}}$ column of $X$ (thus, $x_0 = x$). Let 
$b_k$ and $\tilde x_\ell$ denote the $k^{\text{th}}$ column of $A^{-1}$ and 
the $\ell^{\text{th}}$ row of $X$ respectively. Then the columns of $H$ are 
one of the two types: 
$\begin{bmatrix} b_k \\ x_k \end{bmatrix}$, or 
$\begin{bmatrix} \tilde x_\ell^* \\ z e_{N,\ell} \end{bmatrix}$. 
Due to the fact that $A$ is not necessarily a diagonal matrix, it
will be more difficult to control the distances involving columns of the first
type. 

Partitioning $H$ as $H = \begin{bmatrix} h_0 & H_{-0} \end{bmatrix}$, 
where $h_0$ is the first column of $H$, we have 
\begin{proposition}
\label{dh1} 
Let Assumptions~\ref{ass-model}, \ref{assA}, and \ref{Xcomp} hold true. 
Then 
\[
\PP[ [ \dist(h_0, H_{-0}) \leq t ] \cap [\| X \| \leq C] ]
\leq c_1 (n^{59/88} t^{1/2} + n^{-1 / 22} ) + \exp(-c_2 n) . 
\]
\end{proposition}
Since $[\| X \| \leq C]$ is obviously included in $\Eop(C)$, it will be 
enough to establish the inequality 
\[ 
\PP[ [ \dist(h_0, H_{-0}) \leq t ] \cap \Eop(C)]
\leq c_1 (n^{59/88} t^{1/2} + n^{-1 / 22} ) + \exp(-c_2 n)  
\] 
to obtain Proposition~\ref{dh1}. Replacing $[\| X \| \leq C]$ with $\Eop(C)$
will be more convenient due to the independence of $x$ and $\Eop(C)$. 
The remainder of this section is devoted towards proving this inequality. 
Recall the formula for $\dist(h_0, H_{-0})$ given 
in~\eqref{eq:column-distance}. To be able to use Lemma~\ref{lm-dist}, we need 
to check that $G$ defined in (\ref{gdefine}) is invertible. Recall that $X$ 
is assumed to have a density. 
\begin{lemma} 
\label{lm-inv}  
The matrix $G$ is invertible with probability one. 
\end{lemma}
\begin{proof} 
Since $z\neq 0$, the matrix $z I_N$ is invertible. Thus, to show that $G$ is
invertible with the probability one, we need to show that the Schur complement 
$\Delta = B - z^{-1} W^* W$ of $z I_N$ in $G$ is invertible with
probability one. 

Since $A^{-1} = \begin{bmatrix} b & b_{01} \\ b_{10} & B \end{bmatrix}$, 
it holds that $\rank(B) \geq n-2$. Thus, either $B$ is invertible or 
$\rank(B) = n-2$. 

Assume it is invertible. Then on the set 
$\{ W \in \CC^{N \times (n-1)} \, : \, \| z^{-1} W^* W \| \leq 
 s_{n-2}(B) / 2 \}$, it holds that 
$s_{n-2}(\Delta) \geq s_{n-2}(B) - \| z^{-1} W^* W \| \geq s_{n-2}(B) / 2 > 0$.
Thus, $\det(\Delta)$ is a non-zero multivariate polynomial in the real and 
imaginary parts of the elements of $W$. Since $W$ has a density, 
$\det(\Delta) \neq 0$ w.p.~1. 

Assume now that $\rank(B) = n-2$. Then we can write $B = U V^*$ where 
$U,V \in C^{(n-1)\times (n-2)}$ are full column-rank matrices. 
Writing $W^* = \begin{bmatrix} w & Y \end{bmatrix}$ where $w\in \CC^{n-1}$, we 
get that 
\[
B - z^{-1} W^* W = \begin{bmatrix} U & z^{-1} w \end{bmatrix} 
\begin{bmatrix} V & - w \end{bmatrix}^* - z^{-1} Y Y^* = 
 D - z^{-1} Y Y^* .  
\]
Given a vector $u \perp \colspan(U)$, the inner product 
$u^* w$ is a continuous random variable, thus $u^* w \neq 0$ w.p.~1. 
Consequently, $w \not\in \colspan(U)$ w.p.~1., which implies that 
$\begin{bmatrix} U & z^{-1} w \end{bmatrix}$ is invertible w.p.~1. The
same argument holds for $\begin{bmatrix} V & - w \end{bmatrix}$, and thus the
matrix $D$ is invertible w.p.~1. To obtain that $\Delta$ is invertible, it 
remains to apply the previous argument after replacing $B$ with 
$D$ and $W^*$ with $Y$, and making use of the independence of $w$ and $Y$ along
with the Fubini-Tonelli theorem. 
\end{proof} 

Using Lemmas \ref{lm-dist} and \ref{lm-inv}, we get that on a probability one
set, Equation~\eqref{eq:column-distance} holds.  
On the probability one set where $G$ is invertible, write 
$G^{-1}$ as in~\eqref{eq:ginverse}. Then, from~\eqref{eq:column-distance}, 
$\dist(h_0, H_{-0}) = \text{Num}/\text{Den}$ 
where  \text{Num} and \text{Den} are as given in \eqref{eq:num-den}.

To study the behavior of \text{Num} and \text{Den}, we first need to
show that the image of each deterministic vector by the matrix $R$ at the right
hand side of~\eqref{eq:ginverse} is incompressible with high probability. This 
will be stated in the corollary of Proposition~\ref{incomp1} below.

\begin{lemma}
\label{weyl}
$s_{n-3}(B) \geq \sinf$. 
\end{lemma}
\begin{proof}
The matrix $b_{10} b_{10}^* + BB^*$ is a principal submatrix of the Hermitian 
matrix $A^{-1} A^{-*}$. Using the variational representation of 
the eigenvalues of $A^{-1} A^{-*}$, we get that 
$s_{n-2} ( b_{10} b_{10}^* + BB^* ) \geq {\bf s}_{\inf}^2$. 
By Weyl's interlacing inequalities, 
$s_{n-3} ( BB^* ) \geq s_{n-2} ( b_{10} b_{10}^* + BB^* )$, hence the result. 
\end{proof} 

\begin{proposition}
\label{incomp1} 
There exist $\theta_{\ref{incomp1}} \in (0,1)$, $\rho_{\ref{incomp1}} > 0$, 
and $c_{\ref{incomp1}} >0$ such that for each $d \in \CC^{N}$,  
\[
\PP \Bigl[ 
 \Bigl[ \inf_{\substack{v \in \CC^{n-1}, \\ 
  w \in \comp(\theta_{\ref{incomp1}},\rho_{\ref{incomp1}})}}  
  \dist\left( G \begin{bmatrix} v \\ w \end{bmatrix},  
  \colspan\left( \begin{bmatrix} 0 \\ d \end{bmatrix} \right)  \right)
   \leq \rho_{\ref{incomp1}} \Bigr]
  \cap \Eop(C) \Bigr] \leq \exp(-c_{\ref{incomp1}} n) . 
\]
\end{proposition} 
\begin{proof} 
Let $\theta_{\ref{incomp1}} \in (0,1)$ and $t \in (0,1)$ to be fixed later. Let $\cI \in [N]$ such 
that $| \cI | = \lfloor \theta_{\ref{incomp1}} N \rfloor$. Fix an element $w$ of the 
unit-sphere $\SSS_\cI^{N-1}$. In this first part of the proof, we shall control 
the probability of the event 
\[
 \Bigl[ \inf_{v \in \CC^{n-1}} 
  \dist\left( G \begin{bmatrix} v \\ w \end{bmatrix},  
  \colspan\left( \begin{bmatrix} 0 \\ d \end{bmatrix} \right)  \right) 
   \leq t \Bigr]
  \cap \Eop(C) . 
\]
The event between $[ \ \ ]$ brackets is included in the event 
\begin{equation}\label{ewt}
\cE_w(t) = \left[ 
\exists v\in \CC^{n-1}, \exists \alpha\in \CC \, : \, 
\| B v + W^* w \| \leq t, \,  
\| W v + z w + \alpha d \| \leq t \right] . 
\end{equation}
Let 
\begin{equation}
\label{svdB} 
B = \begin{bmatrix} P & p \end{bmatrix} 
  \begin{bmatrix} \Sigma & 0 \\  0 & s_{n-2}(B) \end{bmatrix} 
  \begin{bmatrix} Q^* \\ q^* \end{bmatrix} 
\end{equation} 
be a singular value decomposition of $B$, where $p$ (resp.~$q$) is the last
column of the unitary matrix $\begin{bmatrix} P & p\end{bmatrix}$ 
(resp.~$\begin{bmatrix} Q & q\end{bmatrix}$). Given any vector 
$y \in \CC^{n-1}$, we shall use in the remainder of the proof the notations 
$y_Q = \Pi_Q y$ and $y_q = \Pi_q y$, making $y = y_Q + y_q$ an orthogonal 
sum. As is well-known (see \cite{rao-mit-(livre71)}), the vector 
$u = - B^\sharp W^* w$ where $B^\sharp$ is the 
Moore-Penrose pseudo-inverse of $B$,  minimizes $\| B y + W^* x \|$ with respect 
to $y$. Assume that there is a solution $v \in \CC^{n-1}$ of the inequality 
$\| B y + W^* w \| \leq t$ in $y$. Then, since $u$ is also a 
solution, we get that
\[
\| B (u_Q - v_Q) + B (u_q - v_q) + B v + W^* w \| \leq t, 
\]
and hence, 
\[
\| B (u_Q - v_Q) + B (u_q - v_q) \| \leq  \| B v + W^* w \| + t 
  \leq 2 t .  
\]
Noting that $B (u_Q - v_Q)$ and $B (u_q - v_q)$ are orthogonal, we get that 
$\| B (u_Q - v_Q) \| \leq 2t$. 
By Lemma~\ref{weyl}, the smallest singular value of the 
restriction of the operator $B$ to the subspace $\colspan(Q)$ is 
bounded below by $\sinf$. Hence we get that  
\[
\| v_Q - u_Q \| \leq \frac{2t}{\sinf} . 
\]
The vector $v$ also satisfies the inequality 
$\| W v + z w + \alpha d \| \leq t$ for some $\alpha \in\CC$. Thus, 
\[
\| W (v_Q - u_Q) + W v_q + W u_Q + z w + \alpha d \| \leq t . 
\]
which implies that on the event $\Eop(C)$, 
\[
\| W v_q + W u_Q + z w + \alpha d \| \leq \|  W (v_Q - u_Q) \| + t 
 \leq \left( 1 + \frac{2C}{\sinf} \right) t . 
\]
Observing that $v_q$ is collinear with $q$, we get at this stage of the proof
that 
\begin{equation}
\label{Ew} 
\cE_w(t) \cap \Eop(C) \subset 
\left[ \exists \alpha, \beta \in \CC, \, : \, 
\| \beta W q + W u_Q + z w + \alpha d \| 
 \leq \left( 1 + \frac{2C}{\sinf} \right) t \right] \cap 
 \Eop(C) .
\end{equation} 
To proceed, we need to control the Euclidean norm of $u_Q$. For $m, M > 0$,
consider the event
\[
\cE_{u_Q}(m, M) = \left[ m \leq \| u_Q \| \leq M \right] .
\]
Since $u_Q = - \Pi_Q B^\sharp W^* w$, we get from Lemma~\ref{weyl} that 
$\ssup^{-1} \| W^* w \| \leq \| u_Q \| \leq \sinf^{-1} \| W^* w \|$. By 
Lemma~\ref{Xu}, there exist $c_0 >0$ and $c_1 > 0$ such that 
$\PP[ \| W^* w \| \leq c_0 ] \leq \exp(-c_1 n)$. We thus obtain 
\begin{equation}\label{euqop}
\PP\left[ \cE_{u_Q}( \ssup^{-1} c_0, \sinf^{-1} C )^\cpl \cap 
 \Eop(C) \right] \leq \exp(-c_1 n) .
\end{equation}
To bound the probability of the event at the right hand side of the 
inclusion~\eqref{Ew}, we consider separately the situations where $|\beta|$ 
is large and where $|\beta|$ is bounded above. Consider the event
\[
\cE_{|\beta| >}(m, M) = \left[ \exists \alpha, \beta \in \CC \, : \, 
\| \beta W q + W u_Q + z w + \alpha d \| \leq m , \, | \beta | \geq M \right] . 
\]
On $\cE_{u_Q}( \ssup^{-1} c_0, \sinf^{-1} C ) \cap \Eop(C)$, it holds that
\[ 
\| \beta W q + W u_Q + z w + \alpha d \| \geq 
 \| \beta W q + z w + \alpha d \| - \sinf^{-1} C^2 
 \geq |\beta| \dist( Wq, \colspan [ w, d ]) -  \sinf^{-1} C^2 
\] 
From Lemma~\ref{Xu}, there exist $c_2, c_3 > 0$ such 
that $\PP[\dist( Wq, \colspan [ w, d ]) \leq c_2 ] \leq \exp(-c_3 n)$. 
Writing $s = (1 + 2C / \sinf) t$, we have 
\begin{align*} 
&\cE_{|\beta|>}(s, M) \cap \cE_{u_Q}( \ssup^{-1} c_0, \sinf^{-1} C ) 
  \cap \Eop(C) \\ 
 &\subset \left[ \exists \beta \in \CC \, : \, 
  |\beta| \dist( Wq, \colspan [ w, d ]) -  \sinf^{-1} C^2 \leq s, \, 
  | \beta | \geq M \right] \\
 &\subset \left[ 
  \dist( Wq, \colspan [ w, d ]) \leq \frac{s + \sinf^{-1} C^2}{M} 
   \right] . 
\end{align*} 
Thus, setting $C' = (s + \sinf^{-1} C^2) / c_2$, we get that 
\begin{equation}\label{ebetalarge}
\PP\left[ 
\cE_{|\beta|>}(s, C') \cap \cE_{u_Q}( \ssup^{-1} c_0, \sinf^{-1} C ) 
  \cap \Eop(C) \right] \leq \exp(-c_3 n) .
\end{equation}
Now consider the case $|\beta| < C'$. We discretize this ball as follows. Consider the event 
\[
\cE_{|\beta| <}(s, C') = \left[ \exists \alpha, \beta \in \CC \, : \, 
\| \beta W q + W u_Q + z w + \alpha d \| \leq s , \, | \beta | <  C' \right] . 
\]
Given $k,\ell \in \ZZ$, define the event 
 \[
 \cE_{q}(k, \ell, s, C)  
 = \left[ \exists \alpha \in \CC \, : \, 
  \left\| \frac{s}{C\sqrt{2}}(k + \imath\ell) W q 
    + W u_Q + z w + \alpha d \right\| \leq s \right] . 
 \]
For $\beta\in\CC$, let 
 $k_{\beta} = \lfloor C\sqrt{2} \Re\beta / s \rfloor$ and 
 $\ell_{\beta} = \lfloor C \sqrt{2} \Im\beta / s \rfloor$. Then 
 $\left| \beta - ( k_{\beta} + \imath \ell_{\beta}) s / (C\sqrt{2}) \right| 
 \leq s / C$. 
Therefore, 
\[
\cE_{|\beta| <}(s, C') \cap \Eop(C) 
  \subset \bigcup_{\substack{k,\ell\in \ZZ, \\ 
   |k+\imath\ell| \leq C C'\sqrt{2} / s}} \cE_{q}(k, \ell, 2s, C) . 
\]  
Let us bound the probability of the event 
$\cE_{q}(k, \ell, 2 s, C) \cap \cE_{u_Q}( \ssup^{-1} c_0, \sinf^{-1} C )$.  
Recalling that $u_Q = - \Pi_Q B^\sharp W^* w$ and that $w$ is supported by 
$\cI$, we observe that $u_Q$ and $W_{\cI^\cpl, \cdot}$ are independent. 
Writing $r = \frac{s(k + \imath\ell)}{C\sqrt{2}} q + u_Q$ and 
$\tilde r = r / \| r \|$, we have 
\begin{align*}
\cE_{q}(k, \ell, 2 s, C) \cap \cE_{u_Q}( \ssup^{-1} c_0, \sinf^{-1} C ) 
 &\subset \left[ \exists \alpha\in\CC, \, : \, \| Wr + zw + \alpha d\| \leq 2s 
 \right] \cap \left[ \| u_Q \| \geq \ssup^{-1} c_0 \right]  \\
 &\subset \left[ \exists \alpha\in\CC, \, : \, 
  \| W_{\cI^\cpl,\cdot} r + zw_{\cI^\cpl} + \alpha d_{\cI^\cpl} \| \leq 2s 
 \right] \cap \left[ \| u_Q \| \geq \ssup^{-1} c_0 \right]  \\
 &\subset \left[ \| r \| \dist\left( W_{\cI^\cpl,\cdot} \tilde r, 
  \colspan [w_{\cI^\cpl}, d_{\cI^\cpl} ] \right) \leq 2s \right] 
     \cap \left[ \| u_Q \| \geq \ssup^{-1} c_0 \right]  \\
 &\subset \left[ \dist\left( W_{\cI^\cpl,\cdot} \tilde r, 
  \colspan[w_{\cI^\cpl}, d_{\cI^\cpl} ] \right) \leq 
  2s \ssup / c_0 \right] . 
\end{align*} 
By Lemma~\ref{Xu} once again, $\PP[ \dist\left( W_{\cI^\cpl,\cdot} \tilde r, 
 \colspan[w_{\cI^\cpl}, d_{\cI^\cpl} ] \right) \leq c_2 ] 
 \leq \exp(-c_3 |\cI^\cpl|)$. Thus, if we choose $t$ small enough so that 
$\left(2 + 4 \frac{C}{\sinf}\right) \frac{\ssup}{c_0} t \leq c_2$, 
we get that 
\begin{equation}\label{discrete}
\PP\left[ \cE_{q}(k, \ell, 2 s, C) \cap 
  \cE_{u_Q}( \ssup^{-1} c_0, \sinf^{-1} C ) \right] \leq 
 \exp(-(1-\theta_{\ref{incomp1}}) c_3 n) .
\end{equation}
Putting things together, we get 
\begin{align*}
 & \PP\Bigl[ \Bigl[ \inf_{v \in \CC^{n-1}} 
  \dist\left( G \begin{bmatrix} v \\ w \end{bmatrix},  
  \colspan\left( \begin{bmatrix} 0 \\ d \end{bmatrix} \right)  \right) 
   \leq t \Bigr] \cap \Eop(C) \Bigr]  \\
&\leq \PP\left[ \cE_w(t) \cap \Eop(C) \right] \ \text{(using (\ref{ewt}))}\\ 
&\leq \PP\left[ \cE_w(t) \cap \cE_{u_Q}( \ssup^{-1} c_0, \sinf^{-1} C )
   \cap \Eop(C) \right] + 
\PP\left[ \cE_{u_Q}( \ssup^{-1} c_0, \sinf^{-1} C )^\cpl 
   \cap \Eop(C) \right] \\ 
&\leq 
\PP\left[ \cE_{|\beta|>}(s, C') \cap \cE_{u_Q}( \ssup^{-1} c_0, \sinf^{-1} C ) 
  \cap \Eop(C) \right] \\
&\phantom{=} + 
\PP\left[ \cE_{|\beta|<}(s, C') \cap \cE_{u_Q}( \ssup^{-1} c_0, \sinf^{-1} C ) 
  \cap \Eop(C) \right] + \exp(-c_1 n) \ \text{(using (\ref{euqop}))}\\ 
&\leq \exp(-c_3 n) + 
  \sum_{|k+\imath\ell| \leq C C'\sqrt{2} / s} 
  \hspace{-0.2in}\PP\left[ \cE_{q}(k, \ell, 2 s, C) 
      \cap \cE_{u_Q}( \ssup^{-1} c_0, \sinf^{-1} C ) \right] + \exp(-c_1 n) \ \text{(using (\ref{ebetalarge}))}\\ 
&\leq  \exp(-c_1 n) + C'' \exp(-(1-\theta_{\ref{incomp1}}) c_3 n) \ \text{(using (\ref{discrete}))}, 
\end{align*} 
where $C'' = C''(\mom, C) > 0$. 

Now, let $\Sigma_t$ be a $t$-net of $(\SSS^{N-1}_{\cI})$.
Given an element $y$ of $\cN_t(\SSS^{N-1}_{\cI}) \cap \SSS^{N-1}$, there exists 
$y' \in \SSS_{\cI}^{N-1}$ such that $\| y - y' \| \leq t$, and there 
exists $w \in \Sigma_t$ such that $\| w - y' \| \leq t$. 
Thus, $\| y - w \| \leq 2 t$ by the triangle inequality. 
Assume that there exist $\alpha \in \CC$ and $v \in \CC^{n-1}$ such that the 
inequality 
\[
  \left\| G \begin{bmatrix} v \\ y \end{bmatrix} + 
  \alpha \begin{bmatrix} 0 \\ d \end{bmatrix} \right\| \leq t 
\]
holds true. Then on the set $\Eop(C)$, we have 
\[
\left\| G \begin{bmatrix} v \\ w \end{bmatrix} + 
  \alpha \begin{bmatrix} 0  \\ d \end{bmatrix} \right\| = 
\left\| G \left( \begin{bmatrix} v \\ w \end{bmatrix} 
 - \begin{bmatrix} v \\ y \end{bmatrix} \right) + 
 G \begin{bmatrix} v \\ y \end{bmatrix} + 
  \alpha \begin{bmatrix} 0  \\ d \end{bmatrix} \right\| 
 \leq 2 (C + |z|) t + t 
\]
By Lemma~\ref{entrop}, $| \Sigma_t | \leq 
(3 / t)^{2 | \cI |}$. Adjusting $t$ again in such a way 
that 
$\left(2C + 2|z| + 1)(2 + 4 \frac{C}{\sinf}\right) \frac{\ssup}{c_0} t 
\leq c_2$, we obtain that 
\begin{multline*} 
\PP \Bigl[ 
 \Bigl[ \inf_{\substack{v \in \CC^{n-1}, \\ y \in 
\cN_t(\SSS^{N-1}_{\cI}) \cap \SSS^{N-1}}}  
  \dist\left( G \begin{bmatrix} v \\ y \end{bmatrix},  
  \colspan\left(\begin{bmatrix} 0 \\ d \end{bmatrix} \right)  \right) 
   \leq t \Bigr]
  \cap \Eop(C) \Bigr] \\ 
\leq  \left( 3 / t\right)^{2 \theta_{\ref{incomp1}} n} 
 \left( \exp(-c_1 n) + C'' \exp(-(1-\theta_{\ref{incomp1}}) c_3 n) \right) .
\end{multline*} 
Finally, considering all the sets $\cI \subset [N]$ such that 
$|\cI| = \lfloor \theta_{\ref{incomp1}} N  \rfloor$, and using the bound  
${{m}\choose{k}} \leq (em/k)^k$ along with the union bound, we get that 
\begin{multline*} 
\PP \Bigl[ 
 \Bigl[ \inf_{\substack{v \in \CC^{n-1}, \\ w \in \comp(\theta_{\ref{incomp1}}, t)}}  
  \dist\left( G \begin{bmatrix} v \\ w \end{bmatrix},  
  \colspan\left( \begin{bmatrix} 0 \\ d \end{bmatrix} \right)  \right) 
   \leq t \Bigr]
  \cap \Eop(C) \Bigr] \\ 
\leq  \left( e / \theta_{\ref{incomp1}} \right)^{\theta_{\ref{incomp1}} N}  
\left( 3 / t\right)^{2 \theta_{\ref{incomp1}} n} 
 \left( \exp(-c_1 n) + C'' \exp(-(1-\theta_{\ref{incomp1}}) c_3 n) \right) .
\end{multline*} 
Choosing $\theta_{\ref{incomp1}}$ small enough, we get the result with 
$\rho_{\ref{incomp1}} = t$ and $c_{\ref{incomp1}}$ small enough. 
\end{proof} 

\begin{corollary} 
\label{Rd} 
For each deterministic vector $d \in \CC^{N} \setminus \{ 0 \}$, 
\[
\PP\left[ \left[ Rd / \| R d \| 
  \in \comp(\theta_{\ref{incomp1}},\rho_{\ref{incomp1}}) \right] 
  \cap \Eop(C) \right] \leq \exp(-c_{\ref{incomp1}} n) . 
\]
\end{corollary} 
\begin{proof}
Write 
\[
\begin{bmatrix} u \\ y \end{bmatrix} = 
  G^{-1} \begin{bmatrix} 0 \\ d \end{bmatrix} =
  \begin{bmatrix} Fd \\ Rd \end{bmatrix} 
\]
with $y \in \CC^{N}$, and let $\tilde y = y / \| y \|$, which can be shown to 
be defined w.p.~1 as in the proof of Lemma~\ref{lm-inv}. 
Considering the event 
$\cE_{\tilde y} = 
[ \tilde y \ \text{defined}, \, 
\tilde y \in \comp(\theta_{\ref{incomp1}},\rho_{\ref{incomp1}})]$, our purpose 
is to show that 
$\PP[ \cE_{\tilde y} \cap \Eop(C) ] \leq \exp(-c_{\ref{incomp1}} n)$. Since 
\[
G \begin{bmatrix} u / \| y \| \\ \tilde y \end{bmatrix} =  
 \| y \|^{-1} \begin{bmatrix} 0 \\ d \end{bmatrix} , 
\]
it holds that 
\[
\cE_{\tilde y} \subset 
 \Bigl[ \inf_{\substack{v \in \CC^{n-1}, \\ 
  w \in \comp(\theta_{\ref{incomp1}},\rho_{\ref{incomp1}})}}  
  \dist\left( G \begin{bmatrix} v \\ w \end{bmatrix},  
  \colspan\left(\begin{bmatrix} 0 \\ d \end{bmatrix} \right) \right)  
   \leq \rho_{\ref{incomp1}} \Bigr] \, , 
\]
and the result follows from Proposition~\ref{incomp1}. 
\end{proof} 

\subsection{Handling the denominator Den in \eqref{eq:num-den}} 

\begin{lemma}
\label{denom}
There exist positive constants $c_{\ref{denom}}$ and $C_{\ref{denom}}$ such 
that 
\[
\PP\left[ \left[ \| M \| \geq C_{\ref{denom}} \| R \| \right] 
  \cap \Eop(C) \right] \leq \exp(-c_{\ref{denom}} n) , 
\]
where $M = F$, $P$, or $E$. 
\end{lemma} 
\begin{proof}
We reuse here the notations of the singular value decomposition~\eqref{svdB} 
of $B$. For any matrix $M$ with $n-1$ rows, we also use the notations 
$M_Q = \Pi_Q M$ and $M_q = \Pi_q M$. We first prove the result for $M=F$. 

From Lemma~\ref{Xu}, we know that there exist $c_0, c > 0$ such that 
$\PP[ \| W q \| \leq c_0 ] \leq \exp(-c n)$. We shall show that on the event 
$[ \| W q \| \geq c_0 ] \cap \Eop(C)$, there exists some $C_1 > 0$, such that  
\[
\forall u\in \SSS^{N-1},  \| F u \| \leq C_1 (1 + \| R u \| ).  
\]
This will establish that 
\begin{equation}
\label{F<R} 
\PP\left[ \left[ \| F \|  \geq C_1 ( 1 + \| R \| ) \right] 
  \cap \Eop(C) \right] \leq \exp(-cn). 
\end{equation}  

%Writing $v = F u$ and $w = R u$, 
Recall that 
\[
\begin{bmatrix} v \\ w \end{bmatrix} = 
  G^{-1} \begin{bmatrix} 0 \\ u \end{bmatrix}= \begin{bmatrix} Fu\\Ru\end{bmatrix}=\begin{bmatrix} v \\ w \end{bmatrix} \ \ \text{say}, 
\]
or equivalently, 
\begin{subequations} 
\begin{align} 
B v + W^* w &= 0 \label{uvw1} \\
Wv + z w &= u . \label{uvw2}
\end{align} 
\end{subequations} 
Since $B v_q \perp B v_Q$, we get from Lemma~\ref{weyl} and~\eqref{uvw1} that 
\[
\sinf \| v_Q \| \leq \| B v_Q \| \leq \| W^* w \|. 
\]
Thus, $\| v_Q \| \leq (C/\sinf) \| w \|$ on $\Eop(C)$. 
Writing $v_q = \beta q$, Equation~\eqref{uvw2} can be rewritten as 
$\beta W q = u - z w - W v_Q$, which gives that 
\[
|\beta| \leq \frac{1}{c_0} + \frac{|z| + C^2/\sinf}{c_0} \| w \|  
\]
on $[ \| W q \| \geq c_0 ] \cap \Eop(C)$. Since $\| v \|^2 = |\beta|^2 + 
 \| v_Q \|^2$, there exists $C_1 > 0$ such that 
$\| v \| \leq C_1 ( 1 + \| w \|)$, and the inequality~\eqref{F<R} follows. 

Our next step is to show that there exists a constant $C_2$ such that 
$\Eop(C) \subset[ \| R \| \geq C_2]$. It is then easy to deduce 
from~\eqref{F<R} that 
$\PP\left[ \left[ \| F \| \geq C'\| R \| \right] \cap \Eop(C) \right] 
\leq \exp(-cn)$ with $C' = C_1(C_2^{-1} + 1)$. 
We shall assume that $\|R\| < C_2$ on $\Eop(C)$ and obtain a contradiction if 
$C_2$ is chosen small enough. From the equation $GG^{-1} = I_{N+n-1}$, we have 
\begin{subequations} 
\begin{align} 
B F + W^* R &= 0,  \label{BF} \\
W F + z R &= I . \label{WF} 
\end{align} 
\end{subequations} 
By Equation~\eqref{BF}, $\| BF \| \leq C C_2$ on $\Eop(C)$. Writing 
$BF = B F_Q + B F_q$ and observing from~\eqref{svdB} that 
$\colspan(B F_Q)$ and $\colspan(B F_q)$ are orthogonal, we obtain that 
$\| B F_Q \| \leq \| B F_Q + BF_q \| \leq CC_2$. Turning to~\eqref{svdB} again
and using Lemma~\ref{weyl}, we also have 
\[
\| B F_Q \|^2 = \| F^* Q \Sigma^2 Q^* F \| \geq \sinf^2 \| F^* Q Q^* F \| = 
 \sinf^2 \| F_Q \|^2, 
\]
thus, $\| F_Q \| \leq C C_2 / \sinf$. Now, rewriting Equation~\eqref{WF} as 
$W F_q - I = -z R - W F_Q$ and using the triangle inequality, we get that 
$\| W F_q - I \| \leq | z | \| R \| + \| W F_Q \| \leq 
  (| z | + C^2 / \sinf) C_2$. 
Since $W F_q$ is a rank-one matrix, the set of vectors $u \in \SSS^{N-1}$ such 
that $W F_q u = 0$ is not empty. For any such vectors, we have 
\[
(| z | + C^2 / \sinf) C_2 \geq \| W F_q - I \| \geq 
  \| (W F_q - I )u \| = 1, 
\]
which raises a contradiction if we choose 
$C_2 < (| z | + C^2 / \sinf)^{-1}$. The lemma is proven for $M = F$. 

The case $M = P$ can be shown similarly. To handle the case $M = E$, 
we first show an analogue of~\eqref{F<R} where $(F,R)$ is replaced with 
$(E,F)$, and then we combine the obtained inequality with~\eqref{F<R} to get 
that 
$\PP\left[ \left[ \| E \|  \geq C_1 ( 1 + \| R \| ) \right] 
  \cap \Eop(C) \right] \leq \exp(-cn)$ with possibly different constants. 
The rest of the proof is unchanged. 
\end{proof} 

The following lemma is very close to \cite[Prop.~8.2]{ver-14}, with the 
difference that the bound on the probability in Statement~\ref{xDHS} is a 
Berry-Esséen type bound. 
 
\begin{lemma}
\label{ctrl-den}
The following hold true: 
\begin{enumerate}
\item\label{gD>} There exist $c_{\ref{ctrl-den}}, C_{\ref{ctrl-den}} > 0$ such 
that  
\[
\PP[ [\| g_{01} G^{-1} \| \leq C_{\ref{ctrl-den}}] \cap \Eop(C) ] \leq 
  \exp(-c_{\ref{ctrl-den}} n). 
\]
\item\label{yM<} Let $y = [ y_0,\ldots, y_{N-1}]^\T \in \CC^N$ be a random 
vector with independent elements such that $\EE y_i = 0$ and 
$\EE|y_i|^2 = 1/n$ for all $i\in[N]$, and let $M \in \CC^{N\times N}$ be 
deterministic. Then for each $\eta > 0$, 
\[
\PP \left[ \| y^* M \| \leq 
 \frac{1}{\sqrt{\eta}} \frac{ \| M \|_\HS}{\sqrt{n}} \right] \geq 1 - \eta. 
\]
\item\label{xDHS} There exists $c > 0$ such that for each $\varepsilon \geq 0$, 
\[
\PP\left[ \left[ \| x^* R \| 
     \leq \varepsilon \frac{\| R \|_\HS}{\sqrt{n}} \right] 
  \cap \Eop(C) \right] \leq c \varepsilon + \frac{c}{\sqrt{n}} . 
\]
\end{enumerate} 
\end{lemma}
\begin{proof} 
To prove the first statement, we write 
$\| g_{01} \| = \| g_{01} G^{-1} G \| \leq  \| g_{01} G^{-1} \| \, \| G \|$. 
By Lemma~\ref{Xu}, there exist two constants $c, c_{\ref{ctrl-den}} >0$ such 
that $\| g_{01} \| \geq \| x \| \geq c$ with a probability
larger than $1 - \exp(-c_{\ref{ctrl-den}}n)$. Moreover, 
$\| G \| \leq (C + |z|\vee \ssup)$ on $\Eop(C)$, hence the result. 

We have 
\[
\EE \| y^* M \|^2 = \EE y^* MM^* y = \frac{\| M \|_\HS^2}{n} . 
\] 
Thus, $\PP [ \| y^* M \| \geq \| M \|_\HS / \sqrt{\eta n} ] \leq \eta$ by 
Markov's inequality. This proves Statement~\ref{yM<}. 

Turning to the third statement, we start by writing   
\begin{equation}
\label{eq-xR} 
\| x^* R \|^2 = \sum_{k\in [N]} | \ps{R^* x, e_{N,k}} |^2 
  = \sum_{k\in [N]} | \ps{x, R e_{N,k}} |^2  
  = \sum_{k\in [N]} \| R e_{N,k} \|^2 
    | \ps{x, \frac{R e_{N,k}}{\| R e_{N,k} \|} } |^2  .
\end{equation} 
Define $u_k = R e_{N,k} / \| R e_{N,k} \| = [ u_{0,k}, \ldots, u_{N-1,k} ]^\T$. 
The idea of the proof is the following. By Corollary~\ref{Rd}, $u_k$ is 
incompressible with high probability. Moreover, $x$ and $u_k$ are 
independent. Therefore, we can use the Berry-Esséen theorem 
(Proposition~\ref{prop:BE}) to control the behavior of the inner
products $\ps{x, u_k}$. We then use \cite[Lemma 8.3]{ver-14} to pass from these
inner products to the sum at the right hand side of~\eqref{eq-xR}. 
Indeed, this lemma shows that if $Z_0,\ldots, Z_{N-1}$ are arbitrary 
non-negative random variables and if $p_0,\ldots, p_{N-1}$ are non-negative 
numbers such that $\sum p_k = 1$, then 
$\PP[ \sum p_k Z_k \leq t ] \leq 2 \sum p_k \PP[Z_k \leq 2t ]$ for each 
$t \geq 0$. 

Specifically, define for each $k \in [N]$ the set of indices 
\[
\cI_k = \left\{ i \in [N] \, : \, 
  \frac{\rho_{\ref{incomp1}}}{\sqrt{N}} \leq | u_{i,k} | \leq 
  \frac{2}{\sqrt{\theta_{\ref{incomp1}} N}}  \right\}. 
\] 
Using the independence of $x$ and $u_k$, Lemma~\ref{lm:restrict} and 
Proposition~\ref{prop:BE}, we get after a small calculation that 
\[ 
\PP_x\left[ |\ps{x, u_k}| \leq \varepsilon \sqrt{2/n} \right] \leq 
\cL_x\Bigl( \sum_{i\in \cI_k} \bar x_{i,0} u_{i,k}, 
  \varepsilon \sqrt{2/n} \Bigr) \leq V_k \wedge 1,   
\] 
where 
\[
V_k = \frac{c \varepsilon\sqrt{2}}{\rho_{\ref{incomp1}} 
  \sqrt{|\cI_k| N^{-1}}} + 
 \frac{8 c \mom^{3/4}}{\theta_{\ref{incomp1}}^{3/2} \rho_{\ref{incomp1}}^3} 
 \frac{1}{\sqrt{| \cI_k |}} , 
\]
and $c >0$ is the constant that appears in the statement of 
Proposition~\ref{prop:BE}. Observing that 
$\sum_{k\in[N]} \| R e_{N,k} \|^2 = \| R \|_\HS^2$ and using 
\cite[Lemma 8.3]{ver-14}, we get that 
\[
\PP_x \Bigl[ \sum_{k\in [N]} \frac{\| R e_{N,k} \|^2}{\| R \|_\HS^2}  
    | \ps{x, u_k} |^2 \leq \frac{\varepsilon^2}{n} \Bigr] 
\leq 2 \sum_{k\in[N]} \frac{\| R e_{N,k} \|^2}{\| R \|_\HS^2} 
 (V_k \wedge 1) . 
\]
Defining the event 
$\cE_\incomp = \cap_{k\in[N]} 
 [ u_k \in \incomp(\theta_{\ref{incomp1}}, \rho_{\ref{incomp1}}) ]$, 
we know from Corollary~\ref{Rd} that 
$\PP[ \cE_\incomp^\cpl \cap \Eop(C) ] \leq N \exp(-c_{\ref{incomp1}} n)$. 
Moreover, $| \cI_k | \geq \theta_{\ref{incomp1}} N / 2$ on $\cE_\incomp$ for
each $k \in [N]$ by Lemma~\ref{spread}. Thus, by changing the value of the 
constant $c$ above we get that $V_k \leq c \varepsilon + c / \sqrt{n}$ on 
$\cE_\incomp$ for each $k \in [N]$. Putting things together, we conclude that
\begin{align*} 
\PP\Bigr[ \Bigr[ \| x^* R \| 
     \leq \varepsilon \frac{\| R \|_\HS}{\sqrt{n}} \Bigl] 
  \cap \Eop(C) \Bigl] 
&= \EE_W \Bigl[ \PP_x 
 \Bigl[ \sum_{k\in [N]} \frac{\| R e_{N,k} \|^2}{\| R \|_\HS^2}  
    | \ps{x, u_k} |^2 \leq \frac{\varepsilon^2}{n} \Bigr] 
    \1_{\Eop(C)} \Bigr] \\
&\leq 
 2 \EE_W \Bigl[ 
 \sum_{k\in [N]} \frac{\| R e_{N,k} \|^2}{\| R \|_\HS^2}  (V_k\wedge 1) 
 \1_{\cE_\incomp} \Bigr] 
 + 2 \EE_W [ \1_{\cE_\incomp^\cpl} \1_{\Eop(C)} ]  \\
&\leq 2 c \varepsilon + 2 c / \sqrt{n} + 2 N \exp(-c_{\ref{incomp1}} n),  
\end{align*} 
which leads to the required result after changing once again the value of $c$. 
\end{proof} 

Lemmas~\ref{denom} and~\ref{ctrl-den} lead to the following control on the
denominator: 
\begin{lemma}
\label{Den<HS}
There exist positive constants $c_{\ref{Den<HS}}$ and $C_{\ref{Den<HS}}$ such 
that for each $\eta > 0$, 
\[
\PP\left[ 
 \left[ \text{Den}^2 \geq C_{\ref{Den<HS}} ( 1 + \eta^{-1} ) \| R \|_\HS^2 
   \right] \cap \Eop(C) \right] \leq 2 \eta + \exp(- c_{\ref{Den<HS}} n ) . 
\]
\end{lemma}
\begin{proof}
Starting with the expression $\text{Den}^2 = 1 + \| g_{01} G^{-1} \|^2$, and 
using Lemma~\ref{ctrl-den}\textendash\ref{gD>}, we get that 
\begin{equation} 
\label{deng12} 
\PP\left[ \left[ 
 \text{Den}^2 \geq (C_{\ref{ctrl-den}}^{-2} + 1) \| g_{01} G^{-1} \|^2 
\right] \cap \Eop(C) \right] \leq \exp(- c_{\ref{ctrl-den}} n ) . 
\end{equation} 
Since 
$\| g_{01} G^{-1} \|^2 \leq 2 (\| b_{01} E \|^2 + \| b_{01} F \|^2 + 
  \| x^*  P \|^2 + \| x^* R \|^2)$, the event 
\[
\cE = \left[ \| g_{01} G^{-1} \|^2 \geq 2 (\| b_{01} E \|^2 + \| b_{01} F \|^2 
 + \| P \|_\HS^2 / (\eta n) + \| R \|_\HS^2 / (\eta n) ) \right]
\]
is included in the event
\[
\cE' = \left[ \| x^* P \|^2 \geq \| P \|_\HS^2 / (\eta n) \right] 
  \cup 
  \left[ \| x^* R \|^2 \geq \| R \|_\HS^2 / (\eta n) \right] . 
\]
Thus, $\PP[ \cE ] \leq \PP[ \cE' ] 
  = \PP_W \otimes \PP_x [\cE'] \leq 2 \eta$ by Lemma~\ref{ctrl-den}\textendash\ref{yM<}.
Furthermore, the event 
\[
\cE'' = \left[ \| g_{01} G^{-1} \|^2 \geq 
  4 \ssup^2 C_{\ref{denom}}^2 \| R \|^2 
  + 2 C_{\ref{denom}}^2 \| R \|^2 / \eta + 2 \| R \|_\HS^2 / (\eta n) ) \right]
\]
is included in the event 
\[
\cE \cup 
  \left[ \| E \| \geq C_{\ref{denom}} \| R \| \right] \cup 
  \left[ \| F \| \geq C_{\ref{denom}} \| R \| \right] \cup 
  \left[ \| P \| \geq C_{\ref{denom}} \| R \| \right] , 
\]
since $\| P \|_\HS / n \leq \| P \|$. Thus, 
\[
\PP\left[ \cE'' \cap \Eop(C) \right] \leq 
  2 \eta + 3 \exp( - c_{\ref{denom}} n ) 
\]
by Lemma~\ref{denom}. The proof is completed by combining this inequality 
with~\eqref{deng12} and using the inequality $\| R \| \leq \| R \|_\HS$. 
\end{proof} 

\subsection{Handling the numerator Num in~\eqref{eq:num-den}}\label{subsecnum}

\textit{This is the only section where we shall need Assumption \ref{Xcomp}}. 
 
We shall use the idea of decoupling that will allow us to replace the term $x^*
P b_{10} - b_{01} F x - x^* R x$ in the expression of this numerator with an
inner product whose concentration function is manageable by means of the
Berry-Esséen theorem. This decoupling idea that dates back to~\cite{got-79} has
been used many times in the literature. The following lemma is found
in~\cite{ver-14} (see also~\cite{sid-93}).  
\begin{lemma}
\label{sid} 
Let $Y$ and $Z$ be independent random vectors, and let $Z'$ be an independent
copy of $Z$. Let $\cE(Y,Z)$ be an event that depends on $Y$ and $Z$. Then
\[
\PP[ \cE(Y,Z) ]^2 \leq \PP[ \cE(Y,Z) \cap \cE(Y, Z') ] . 
\]
\end{lemma}

\begin{lemma}
\label{decouple}
Let $a \in \CC$, $u,v \in \CC^{N}$ and $M \in \CC^{N\times N}$ be 
deterministic. Let $\cI \subset [N]$. Then for each 
$t > 0$, 
\[
\PP\left[ \left| x^* M x + u^* x + x^* v + a \right| \leq t \right]^2 
\leq \EE_{x_{\cI^\cpl}, x'_{\cI^\cpl}} 
 \cL_{x_\cI}\left( 
 ( x_{\cI^\cpl} - x'_{\cI^\cpl} )^* M_{\cI^\cpl, \cI} x_\cI 
 + x_\cI^* M_{\cI, \cI^\cpl} ( x_{\cI^\cpl} - x'_{\cI^\cpl} ), 2t \right) , 
\]
where $x'$ is an independent copy of $x$ (here we assume that the right hand 
side is equal to one if $\cI = \emptyset$ or $[N]$). 
\end{lemma} 
\begin{proof}
Assume without loss of generality that $\cI = [ | \cI | ]$. Write 
\[
x = \begin{bmatrix} x_\cI \\ x_{\cI^\cpl} \end{bmatrix}, \quad 
 \text{and} \quad  
\tilde x = \begin{bmatrix} x_\cI \\ x'_{\cI^\cpl} \end{bmatrix}. 
\]
Using Lemma~\ref{sid} with $Y = x_\cI$, $Z = x_{\cI^\cpl}$, and 
$Z' = x'_{\cI^\cpl}$, we get 
\begin{align*}
& \PP\left[\left| x^* M x + u^* x + x^* v + a \right| \leq t \right]^2 \\
&\leq \PP_{x_\cI, x_{\cI^\cpl}, x'_{\cI^\cpl}} 
  \left[ | x^* M x + u^* x + x^* v + a | \leq t, 
 | \tilde x^* M \tilde x + u^* \tilde x + \tilde x^* v + a | \leq t \right] \\ 
&\leq 
\PP_{x_\cI, x_{\cI^\cpl}, x'_{\cI^\cpl}} 
   \left[ | x^* M x - \tilde x^* M \tilde x 
          + u^* (x - \tilde x) + (x - \tilde x)^* v | \leq 2 t \right] , 
\end{align*} 
where the second inequality is due to the triangle inequality. Developing,
we get that 
\begin{align*} 
&\PP_{x_\cI, x_{\cI^\cpl}, x'_{\cI^\cpl}} 
   \left[ | x^* M x - \tilde x^* M \tilde x 
          + u^* (x - \tilde x) + (x - \tilde x)^* v | \leq 2 t \right] \\ 
 &= \PP_{x_\cI, x_{\cI^\cpl}, x'_{\cI^\cpl}} \left[ 
 \left| ( x_{\cI^\cpl} - x'_{\cI^\cpl} )^* M_{\cI^\cpl, \cI} x_\cI 
 + x_\cI^* M_{\cI, \cI^\cpl} ( x_{\cI^\cpl} - x'_{\cI^\cpl} ) \right.\right. \\
& 
 \quad \quad \quad \quad \quad \quad \quad \quad 
 \quad \quad \quad \quad \quad \quad \quad \quad 
 \left.\left. 
 + u^*_{\cI^\cpl} (x_{\cI^\cpl} - x'_{\cI^\cpl}) + 
 (x_{\cI^\cpl} - x'_{\cI^\cpl})^* v_{\cI^\cpl}  \right| \leq 2t \right] \\
&\leq \EE_{x_{\cI^\cpl}, x'_{\cI^\cpl}} 
 \cL_{x_\cI}\left( 
 ( x_{\cI^\cpl} - x'_{\cI^\cpl} )^* M_{\cI^\cpl, \cI} x_\cI 
 + x_\cI^* M_{\cI, \cI^\cpl} ( x_{\cI^\cpl} - x'_{\cI^\cpl} ), 2t \right) . 
\end{align*} 
\end{proof}

We now have all the ingredients to prove Proposition~\ref{dh1}. 

\paragraph{Proof of Proposition~\ref{dh1}.} 
In the remainder, we write 
\[
\cE_{\text{Den}}(\eta) = \left[ 
  \text{Den} \leq C_\eta \| R \|_\HS \right] ,  
\]
where $C_\eta = C_{\ref{Den<HS}}^{1/2} (1 + \eta^{-1})^{1/2}$. 
Given $t > 0$, we have 
\begin{align*} 
& \PP\left[ [ \dist(h_0, H_{-0}) \leq t ] \cap \Eop(C) \right]^2 \\
&= \PP [ [ \text{Num} \leq t \text{Den}] \cap \Eop(C) ]^2   \\
&\leq 
 2 \PP\left[ [ \text{Num} \leq t \text{Den} ] \cap 
 \cE_{\text{Den}}(\eta) \cap \Eop(C) \right]^2 
  + 2 \PP[ \cE_{\text{Den}}(\eta)^\cpl \cap \Eop(C) ]^2 , 
\end{align*}
and 
\begin{align*} 
 \PP\left[ [ \text{Num} \leq t \text{Den} ] \cap 
 \cE_{\text{Den}}(\eta) \cap \Eop(C) \right]^2 
&\leq 
 \PP\left[ [ \text{Num} / \| R \|_\HS \leq t C_\eta ] 
   \cap \Eop(C) \right]^2 \\ 
&= 
 \EE_W \left[ \EE_x [ 
   \1_{[ \text{Num} / \| R \|_\HS \leq t C_\eta ]} ] 
   \1_{\Eop(C)} \right]^2 \\ 
&\leq 
 \EE_W \left[ (\EE_x \1_{[ \text{Num} / \| R \|_\HS \leq t C_\eta ]})^2 
   \1_{\Eop(C)} \right] . 
\end{align*} 

Given an arbitrary $\cI \subset [n]$, we denote as 
$u \in \CC^{|\cI|}$, $v \in \CC^{|\cI^\cpl|}$, and $w \in \CC^{|\cI^\cpl|}$ 
three independent vectors, independent of everything else, such that 
$u \stackrel{\cL}{=} x_{\cI}$ and $v,w \stackrel{\cL}{=} x_{\cI^\cpl}$.  
Recalling the expression of $\text{Num}$ in~\eqref{eq:num-den} and using 
Lemma~\ref{decouple}, we get that for each $s > 0$, 
\begin{align} 
\PP_x [ \text{Num} \leq s ]^2 &\leq 
  \EE_{v , w} \cL_{u} 
  \left( ( v - w )^* R_{\cI^\cpl, \cI} u   
 + u^* R_{\cI, \cI^\cpl} ( v - w ),  2 s \right) \nonumber \\
&= 
  \EE_{v , w} \cL_{u} 
  \left( ( v - w )^* P_{\cI^\cpl}^* R P_\cI u   
 + u^* P_\cI^* R P_{\cI^\cpl} ( v - w ),  2 s \right),  
\label{gotze} 
\end{align} 
where $P_{\cI}$ the $\CC^{|\cI|} \to \CC^N$ linear mapping 
such that if $\cI = \{ i_1 < \cdots < i_{|\cI|} \}$, then 
$P_{\cI} u = ( 0, \ldots, 0, u_{1}, 0, \ldots, 0, u_{|\cI|}, 0, \ldots )$,
where $u_j$ is at the position $i_j$. 

Let $\xi = (\xi_0, \ldots, \xi_{N-1})$ be a vector of $N$ i.i.d.~Bernoulli 
random variables valued in $\{0,1\}$ such that $\PP[\xi_0 = 1] = p$, where the
probability $p$ will be fixed below.  
This vector is assumed to be independent of everything else.  
Since~\eqref{gotze} is true for each $\cI \subset [N]$, we can randomize 
$\cI$ by setting $\cI = \{ i \in [n] \, : \, \xi_i = 1 \}$. Setting 
$s = \| R \|_\HS C_\eta t$, we obtain 
\begin{align} 
(\EE_x \1_{[\text{Num} / \| R \|_\HS \leq t C_\eta ]})^2 
 &\leq \EE_\xi \EE_{v,w} \cL_{u} 
  \left( \frac{( v - w )^* P_{\cI^\cpl}^* R}{\| R \|_\HS} P_\cI u   
 + u^* P_\cI^* \frac{R P_{\cI^\cpl} ( v - w )}{\| R \|_\HS},  
 2 C_\eta t \right)  \nonumber \\ 
&=  \EE_{\xi,x,x'} 
  \cL_{u} 
  \left( \frac{(x-x')^* \Pi_{\cI^\cpl} R}{\| R \|_\HS} P_\cI u 
 + u^* P_\cI^* \frac{R \Pi_{\cI^\cpl} ( x-x')}{\| R \|_\HS} , 
 2 C_\eta t \right) . \label{eqLu} 
\end{align} 
where $x'$ is a vector that has the same law as $x$ and that is independent
of all other random variables. 

Write 
\[
y = 
\frac{ R \Pi_{\cI^\cpl} (x -x')}{\| R \Pi_{\cI^\cpl} (x-x') \|} = 
\begin{bmatrix} y_0 \\ \vdots\\ y_{N-1} \end{bmatrix}, \quad 
 \text{and} \quad 
\tilde y^* = 
\frac{(x-x')^* \Pi_{\cI^\cpl} R}{\|(x-x')^* \Pi_{\cI^\cpl} R\|} = 
\begin{bmatrix} \overline{\tilde y}_0 & \cdots & \overline{\tilde y}_{N-1} 
 \end{bmatrix}, 
\]
and let 
\[
\alpha = 
\frac{\sqrt{n} \| R \Pi_{\cI^\cpl} (x-x')\|}{\sqrt{2(1-p)} \| R \|_\HS} 
\ \text{and} \ 
\tilde \alpha = 
\frac{\sqrt{n} \|(x-x')^* \Pi_{\cI^\cpl} R \|}{\sqrt{2(1-p)}\| R \|_\HS} . 
\]
For $i \in \cI$, let 
\[Z_i = \tilde\alpha \overline{\tilde y_i} [P_\cI u]_i + 
 \alpha \overline{[P_\cI u]_i} y_i.
\] 
Then the concentration function $\cL_u$ at the right hand side of 
\eqref{eqLu} can be rewritten as 
%\[
%\cL_u\left( \tilde\alpha \tilde y^* P_\cI u + \alpha u^* P_\cI^* y , 
  %\sqrt{2/(1-p)} C_\eta t \sqrt{n} \right) .
%\]
%Our purpose is to control 
\[
\cL_u\left( \tilde\alpha \tilde y^* P_\cI u + \alpha u^* P_\cI^* y , 
  \sqrt{2/(1-p)} C_\eta t \sqrt{n} \right) = 
 \cL_u\Bigl( \sum_{i\in\cI} Z_i, \sqrt{2/(1-p)} C_\eta t \sqrt{n} \Bigr) 
\]
We wish to control this by using the Berry-Esséen theorem (Proposition~\ref{prop:BE}). Recalling 
Proposition~\ref{incomp1}, define the set 
\[
\cJ = \left\{ i \in [N] \, : \, 
 \frac{\rho_{\ref{incomp1}}}{\sqrt{N}} \leq |y_i| 
   \leq \frac{2}{\sqrt{\theta_{\ref{incomp1}} N}} 
  \ \text{and} \ 
 |\tilde y_i| \leq \frac{2}{\sqrt{\theta_{\ref{incomp1}} N}}  \right\} .  
\] 
By the restriction lemma~\ref{lm:restrict}, we have 
\[
 \cL_u\Bigl( \sum_{i\in\cI} Z_i, \sqrt{2/(1-p)} C_\eta t \sqrt{n} \Bigr) 
\leq 
\cL_u \Bigl( \sum_{i\in \cI \cap \cJ} Z_i, 
           \sqrt{2/(1-p)} C_\eta t \sqrt{n} \Bigr) . 
\]

Informally, we expect $| \cI \cap \cJ |$ to be of order $\cO(n)$ with high 
probability, the $\EE_u | Z_i |^2$ to be lower bounded with high probability, 
and the $\EE_u | Z_i |^3$ to be upper bounded with high probability for 
$i \in \cI \cap \cJ$, in order to benefit from the effect of the Berry-Esséen 
theorem in a manner similar to Inequality~\eqref{eq:be-sqrt(n)}.

More rigorously, for each $i\in \cI$, we have 
\begin{multline*} 
\EE_u | Z_i |^2 = 
\EE_{x_{00}} | \tilde\alpha \overline{\tilde y_i} x_{00} + 
 \alpha y_i \overline{x_{00}} |^2 
= \EE|x_{00}|^2 
  \left( \tilde\alpha^2 |\tilde y_i |^2 + \alpha^2 |y_i|^2 \right) + 
  2 \alpha\tilde\alpha \Re\left( \EE x_{00}^2  \overline{\tilde y_i} 
  \overline{y_i} \right) \\ 
\geq n^{-1} \vartheta ( \tilde\alpha^2 |\tilde y_i |^2 + \alpha^2 |y_i|^2 ) 
\end{multline*} 
for all large enough $n$, where $\vartheta = \liminf_n n (1 - | \EE x_{00}^2 |)$ is positive 
\label{vrth} 
by Assumption~\ref{Xcomp}. Focusing on the set $\cI \cap \cJ$, we get that 
\begin{equation}
\label{eq-var-Zi} 
\sum_{i\in \cI \cap \cJ} \EE_u | Z_i |^2 \geq 
 n^{-1} \vartheta \sum_{i\in \cI \cap \cJ} \alpha^2 |y_i|^2 \geq 
 \alpha^2 \vartheta \rho_{\ref{incomp1}}^2 \frac{|\cI \cap \cJ|}{nN} . 
\end{equation} 
Moreover, 
\[
\sum_{i\in \cI \cap \cJ} \EE_u | Z_i |^3 \leq 
 32 \EE | x_{00} |^3 (\alpha^3 + \tilde\alpha^3) 
  \frac{|\cI \cap \cJ|}{\theta_{\ref{incomp1}}^{3/2} N^{3/2}} .
\]
Then, by the Berry-Esséen theorem, 
\begin{align*} 
& \cL_u\left( \sum_{i\in_cI} Z_i, \sqrt{2/(1-p)} C_\eta t \sqrt{n} \right) \\
&\leq 
\cL_u \Bigl( \sum_{i\in \cI \cap \cJ} Z_i, 
           \sqrt{2/(1-p)} C_\eta t \sqrt{n} \Bigr) \\
 &\leq 
   \Bigl(
   \sqrt{2/(1-p)} 
  c \frac{\sqrt{N}}{\alpha\rho_{\ref{incomp1}}\sqrt{\vartheta |\cI\cap\cJ|}} C_\eta n t  
   + \frac{32 c \mom^{3/4} (\alpha^3 + \tilde\alpha^3)}
    {\rho_{\ref{incomp1}}^3\vartheta^{3/2}\theta_{\ref{incomp1}}^{3/2} \alpha^3 }   
    \frac{1}{\sqrt{|\cI\cap\cJ|}} \Bigr) \wedge 1 \\
 &\eqdef V \wedge 1  
\end{align*} 
(here, we assume that $\cL_u ( \sum_{\cI \cap \cJ} \cdots ) = V = 1$  
if $\cI \cap \cJ = \emptyset$). 
The constant $c > 0$ in the term after the second inequality is the one that 
appears in the statement of Proposition~\ref{prop:BE}. In the remainder of the
proof, the value of this constant may change without mention.

At this stage of the calculation, we have 
\begin{equation}
\label{Pdist}
\PP[ [ \dist(h_0, H_{-0}) \leq t ] \cap \Eop(C)]^2  
\leq 
  2 \EE_{W, \xi, x, x'} [(V \wedge 1) \1_{\Eop(C)}] + 
  2 \PP[ \cE_{\text{Den}}(\eta)^\cpl \cap \Eop(C) ]^2 . 
\end{equation} 

Now, take $p = 1 - \theta_{\ref{incomp1}} / 8$, and consider the event 
\[
\cE_\xi = [ |\cI| > N(1 - \theta_{\ref{incomp1}}/4) ]  = 
\left[ \sum \xi_i > N(1 - \theta_{\ref{incomp1}}/4) \right]. 
\]
Since $\cE_\xi = [ |\cI| > N(p - \theta_{\ref{incomp1}} / 8) ]$, we get by Hoeffding's 
concentration inequality that 
\[
\PP[ \cE_\xi^\cpl ] \leq \exp(-N \theta_{\ref{incomp1}}^2 / 32) . 
\] 
Consider also the event 
\[
\cE_{\incomp} = [ y \in \incomp(\theta_{\ref{incomp1}},\rho_{\ref{incomp1}}) ] .
\] 
By Corollary~\ref{Rd}, there exists a constant $c > 0$ such that 
\[
\PP [ \cE_{\incomp}^\cpl \cap \Eop(C) ] \leq \exp(-cn). 
\]
On $\cE_{\incomp}$, we have that $|\cJ| \geq \theta_{\ref{incomp1}} N / 2$ by 
Lemma~\ref{spread}. Therefore, on $\cE_\xi \cap \cE_\incomp$, it holds that
\[
| \cI \cap \cJ | = N - | \cI^\cpl \cup \cJ^\cpl | 
 \geq N - | \cI^\cpl | - | \cJ^\cpl | 
 \geq N \theta_{\ref{incomp1}} / 4 . 
\]
 
It remains to control the terms $\alpha$ and $\tilde\alpha$ in the expression
of $V$. Given a small $\beta > 0$, consider the event 
\begin{multline*} 
\cE_\alpha(\beta) = 
 \Bigl[ \beta \| R \|_\HS /\sqrt{n} \leq 
  \frac{\|R \Pi_{\cI^\cpl} (x - x')\|}{\sqrt{2(1-p)}} \leq 
   \beta^{-1/2} \| R \|_\HS /\sqrt{n} \Bigr] \\
  \cap \ 
 \Bigl[ \frac{\| (x-x')^* \Pi_{\cI^\cpl} R \|}{\sqrt{2(1-p)}} \leq 
   \beta^{-1/2} \| R \|_\HS / \sqrt{n} \Bigr]  . 
\end{multline*} 
Note that $\alpha \in [\beta, \beta^{-1/2}]$ and 
$\tilde\alpha \leq \beta^{-1/2}$, thus 
$(\alpha^3 + \tilde\alpha^3)/\alpha^3 \leq 2\beta^{-9/2}$ on 
$\cE_\alpha(\beta)$. 
Applying Lemma~\ref{ctrl-den} after setting the vector $y$ in its statement to 
$[\xi_0(x_{00}-x'_0), \ldots, \xi_{N-1}(x_{N-1,0}-x'_{N-1})]^\T 
 / \sqrt{2(1-p)}$, we get that there exists a constant $c > 0$ for which 
\[
\PP[\cE_\alpha(\beta)^\cpl \cap \Eop(C)] \leq c \beta + \frac{c}{\sqrt{n}} .
\]  

Turning back to~\eqref{Pdist}, we can now conclude by writing 
\begin{align*} 
\PP[ [ \dist(h_0, H_{-0}) \leq t ] \cap \Eop(C)]^2  
 &\leq 
  2 \EE_{W,\xi, x, x'} [ V \1_{\cE_{\xi}} \1_{\cE_{\incomp}} 
  \1_{\cE_\alpha(\beta)} \1_{\Eop(C)} ] ] + 2 \PP[ \cE_{\xi}^\cpl ] \\
 &\phantom{=} + 2 \PP[ \cE_{\incomp}^\cpl \cap \Eop(C) ] 
+ 2 \PP[ \cE_\alpha(\beta)^\cpl \cap \Eop(C) ] \\  
 &\phantom{=} + 2 \PP[ \cE_{\text{Den}}(\eta)^\cpl \cap \Eop(C)]^2 \\
&\leq c \left( \frac{n}{\beta\sqrt{\eta}} t + 
  \frac{\beta^{-9/2}}{\sqrt{n}} + \beta + \eta + \frac{1}{\sqrt{n}} \right) 
 + \exp(-c'n) . 
\end{align*} 
If we take $\eta \propto n^{-1/2}$ and $\beta \propto n^{-1/11}$ (without 
further optimization of these exponents), then we get that 
\[
\PP[ [ \dist(h_0, H_{-0}) \leq t ] \cap \Eop(C)]^2  
\leq c (n^{59/44} t + n^{-1/11} ) + \exp(-c'n) , 
\]
which proves Proposition~\ref{dh1}. 

\subsection{Theorem~\ref{snA}: end of proof} 
\label{subsec:endthsnA} 
First note that for any $k \in [n]$,
Proposition~\ref{dh1} continues to hold when $\dist(h_0, H_{-0})$ 
is replaced by $\dist(h_k, H_{-k})$, by the same proof. 
When $n \leq k < N+n$, too, the proof continues to be valid once the roles of 
$A$ and $z$ are 
interchanged. Indeed, one can check that the argument is simpler and hence is omitted. 
Applying Lemma~\ref{lm-distM}, we obtain that 
\[
\PP\left[ \left[
\inf_{u\in \incomp(\theta_{\ref{prop-comp}},\rho_{\ref{prop-comp}})} 
  \| H u \| \leq t \right] \cap \Eop(C) \right] 
\leq c (n^{81/88} t^{1/2} + n^{-1/22} ) + \exp(-c'n) .  
\]
Using Proposition~\ref{prop-comp} along with the 
characterization~\eqref{comp-incomp} of the smallest singular value, we 
obtain Theorem~\ref{snA} with $\alpha = 81/88$ and $\beta = 1/22$. 

\begin{remark}
\label{rem:beta} 
The proof of Proposition~\ref{dh1} shows that the origin of the slow
decreasing term $n^{-\beta} = n^{-1/22}$ at the right hand side of the last
inequality is the $\cO(1/\sqrt{n})$ decay that is optimal while using the
Berry-Esséen theorem, as shown by Inequality~\eqref{eq:be-sqrt(n)}.  To obtain
a better decay rate of the concentration functions, one can use the so-called
\textit{Littlewood-Offord} theory instead.
This was the approach of \cite{rud-ver-advmath08, tao-vu-aop10,
tao-vu-anmath09,  ver-14} among others to solve small singular value problems.
\end{remark}

\begin{remark} 
Assumption~\ref{Xcomp} was needed in the proof of Proposition~\ref{dh1} 
to ensure that the variance at the left hand side of~\eqref{eq-var-Zi} is
bounded away from zero. 
\end{remark}

\section{Proof of Theorem~\ref{th:main}}  
\label{sec:prf-main} 

\subsection{A general approach: log potential} 
\label{subsec:approach} 

A well established technique for studying the spectral behavior of large random
non-Hermitian matrices is Girko's so-called \textit{hermitization technique}
\cite{gir-84}. This is intimately tied to the logarithmic potential of
their spectral measures. 

Recall that the \textit{logarithmic potential} of a 
probability measure $\mu$ on $\CC$ 
is the $\CC \to (-\infty,\infty]$ superharmonic function defined as
\[
U_\mu(z) = - \int_\CC \log |\lambda - z | \ \mu(d\lambda) \, \ \text{(whenever the integral is finite)}. 
\]  
The measure $\mu$ can be recovered from $U_\mu(\cdot)$ in the following way.
Let $\cD'(\CC)$ the space of Schwartz distributions on $\CC$
and let  
$\Delta = \partial_x^2 + \partial_y^2 = 4 \partial_z
\partial_{\bar z}$ for $z = x + \imath y \in \CC$
be the Laplace operator defined on $\cD'(\CC)$. 
Let   
\[C_{\text{c}}^\infty(\CC)= \{\varphi: \varphi \ \text{is a compactly supported real valued smooth function on} \ \CC\}. 
\]
Note that  $C_{\text{c}}^\infty(\CC) \subset \cD'(\CC)$.
Then \begin{equation}\label{mudelta}\mu = -(2\pi)^{-1} \Delta U_\mu 
\end{equation}
in the sense that 
\[
\int_\CC \varphi(z) \ \mu(dz) = - \frac{1}{2\pi} \int_\CC \Delta \varphi(z) 
 \, U_\mu(z) \ dz,  \ \forall \varphi \in C_{\text{c}}^\infty(\CC).  
\]
It is also known that the convergence of the logarithmic potentials for 
Lebesgue almost all $z\in\CC$ implies the weak convergence of the 
underlying measures under a tightness criterion (see, \textit{e.g.}, \cite{bor-cha-12}). 

Turning back to our matrix $Y = X J X^*$, the logarithmic potential of its 
spectral measure can be written as 
\begin{align*} 
U_{\mu_n}(z) &= - \frac 1N \sum \log | \lambda_i - z |  
= - \frac 1N \log | \det (Y - z) |  \\ 
  &= - \frac{1}{2N} \log \det (Y - z) (Y-z)^*  
   = - \int \log\lambda \ \nu_{n,z}(d\lambda),  
\end{align*}
where the probability measure $\nu_{n,z}$ is the \textit{singular value} 
distribution of $Y-z$, given as 
\[
\nu_{n,z} = \frac 1N \sum_{i=0}^{N-1} \delta_{s_i(Y - z)} . 
\]
The above observation is at the heart of the hermitization technique. It
transforms the eigenvalue problem into a problem of singular values. To study
the asymptotic behavior of $\mu_n$, we need to study the asymptotic behavior of
$U_{\mu_n}(z)$ for Lebesgue almost all $z\in\CC$. In the light of the above
relations, this approach may be formalized as follows: 

\begin{proposition}[Lemma 4.3 of \cite{bor-cha-12}] 
\label{prop-herm}
Let $(M_n)$ be a sequence of random matrices with complex entries. 
Let $\zeta_n$ be its 
spectral measure 
and let $\sigma_{n,z}$ be the
empirical singular value distribution of $M_n - z$. Assume that \\
\noindent (i) for almost every $z\in\CC$, there exists a probability measure $\bs\sigma_z$ such that
$\sigma_{n,z} \Rightarrow \bs\sigma_z$ in probability, \\
\noindent (ii) $\log$ is uniformly integrable in probability with respect to the
sequence $(\sigma_{n,z})$. 

\noindent Then, there exists a probability measure $\bs\zeta$
such that $\zeta_n \Rightarrow \bs\zeta$ in probability, and furthermore, 
\[
U_{\bs\zeta}(z) = - \int \log \lambda \ \bs\sigma_z(d\lambda) \quad 
 \CC-\text{a.e.}
\]
\end{proposition} 

Note that to successfully apply Proposition \ref{prop-herm} to $\mu_n$, we need to establish 
that:\vskip5pt  

\noindent Step 1: for almost all $z\in \CC$, $\nu_{n,z} \Rightarrow \nu_z$ (a deterministic
probability measure) in probability. \vskip5pt

\noindent 
Step 2:  
the function $\log$ is uniformly integrable with respect to the measure $\nu_{n,z}$ for almost all $z\in\CC$ 
in probability. That is, 
\begin{equation}\label{unintinprob}
\forall\varepsilon > 0, \quad
\lim_{T \to \infty}\limsup_{n\geq 1} 
 \PP \left[ \int_0^\infty |\log\lambda| \, \1_{|\log\lambda|\geq T} \ 
  \nu_{n,z}d(\lambda) > \varepsilon \right]= 0. 
 \end{equation}

By achieving these two steps, we conclude that there exists a probability 
measure $\bs\mu$ such that $\mu_n \Rightarrow \bs\mu$ in probability, and such 
that
$U_{\bs\mu}(z) = - \int \log|\lambda | \ \check{\bs\nu}_z(d\lambda)$ 
$\CC$-almost everywhere. It would then remain to identify the measure $\bs\mu$ to 
complete the proof of Theorem~\ref{th:main}. 

The proofs of the results devoted to the asymptotic behavior of $\nu_{n,z}$
(mainly, Step~1) are provided in Section~\ref{sec:lsd}.  Step~2 relies heavily
on Theorem~\ref{snA} above. The proofs of the results devoted to 
the identification of $\bs\mu$ are provided in Section~\ref{sec:prop-mu}.

\subsubsection{Step 1: Weak convergence of $\nu_{n,z}$} \label{subsub-nu_n} 
Going a bit further than what Proposition~\ref{prop-herm} requires on
$Y$, we shall show that for each $z\in \CC$, there exists a probability 
measure $\bs\nu_z$ such that $\nu_{n,z} \Rightarrow \bs\nu_z$ almost surely. 
As is usual in random matrix theory, this convergence will be 
established through the convergence of the associated Stieltjes transforms. 
For this, it will be convenient to consider the 
Hermitian matrix 
\[
\Sigma(z) = \begin{bmatrix}  & Y - z \\
Y^* - \bar z \end{bmatrix} % \in {\mathcal S}_{2N} . 
\]
whose spectral measure 
\[
\check\nu_{n,z} = \frac{1}{2N} \sum_{i=0}^{N-1} 
  \left( \delta_{s_i(Y - z)} + \delta_{-s_i(Y - z)} \right) 
\]
is the symmetrized version of $\nu_{n,z}$ ($\check\nu_{n,z}$ is
symmetric in the sense that $\check\nu_{n,z}(S) = \check\nu_{n,z}(-S)$ for
each Borel set $S\subset \RR$). It is enough to show that $\check\nu_{n,z}$ converges weakly a.s.~to a probability measure $\check{\bs\nu}_z$.

Given $\eta \in \CC_+ = \{ w\in\CC, \Im w > 0 \}$, let us write
\begin{align} 
Q(z,\eta) &= ( \Sigma(z) - \eta )^{-1} \nonumber \\
 &= 
\begin{bmatrix} Q_{00}(z,\eta) &  Q_{01}(z,\eta) \\
  Q_{10}(z,\eta) &  Q_{11}(z,\eta) \end{bmatrix} \nonumber \\
 &= \begin{bmatrix} 
\eta[ (Y-z)(Y-z)^* - \eta^2]^{-1} & (Y-z) [ (Y-z)^*(Y-z) - \eta^2]^{-1} \\
[ (Y-z)^*(Y-z) - \eta^2]^{-1} (Y-z)^* & 
\eta[ (Y-z)^*(Y-z) - \eta^2]^{-1} 
\end{bmatrix} , \label{Qeta} 
\end{align}
which is the \textit{resolvent} of $\Sigma(z)$ in the complex variable $\eta$.

The a.s.~convergence $\nu_{n,z} \Rightarrow \bs\nu_z$ is a consequence of the
following theorem. Its proof will be provided in Section~\ref{prf:narrow}.  By
this theorem, the first assumption in the statement of
Proposition~\ref{prop-herm} is satisfied when $M_n$ is replaced with $Y - z$.
Note that the Stieltjes transform of a symmetric probability
measure is purely imaginary with a positive imaginary part on the positive 
imaginary axis. 
\begin{theorem}
\label{narrowcvg} 
Let Assumption~\ref{ass-model} hold true. Then
\begin{equation}
\label{cvg-res} 
\frac{1}{2N} \tr Q(z,\eta) \toaslong \gamma^{-1} p(z,\eta) 
  \quad \text{and} \quad 
\frac{1}{N} \tr Q_{01}(z,\eta) \toaslong \gamma^{-1} d(z,\eta) , 
   \quad \eta\in\CC_+  
\end{equation} 
where for each $z\in \CC$,  $(p(z,\cdot), d(z,\cdot))$ 
is a pair of holomorphic functions on $\CC_+$ such that $\gamma^{-1}p(z,\cdot)$ is the 
Stieltjes transform of a symmetric probability measure, 
$|d(z, \eta)| \leq \gamma / \Im\eta$, and 
writing $p(z,\imath t) = \imath h(z,t)$ for $t > 0$,   the pair 
$(h(z,t), d(z, \imath t)) \in (0,\infty) \times \CC$ uniquely solves the equations 
\begin{subequations} 
\label{hdt} 
\begin{align} 
 - t h(z,t) + \bar z d(z,\imath t)  &= u(h(z,t),d(z,\imath t),t) -  \gamma,  
    \label{hdut} \\ 
  z h(z,t) + t d(z,\imath t) &= v(h(z,t),d(z,\imath t),t),  \label{hdvt}
\end{align} 
\end{subequations} 
where
\begin{equation}
\label{uvt} 
\begin{split} 
u(h,d) &= \frac{1}{2\pi} \int_0^{2\pi} 
  \frac{h^2 + |d|^2 + d \exp(\imath \theta)}
   {h^2 + | 1 + d \exp(\imath\theta)|^2} d\theta, 
 \quad \text{and} \\
  v(h,d) &= \frac{1}{2\pi} \int_0^{2\pi} 
   \frac{ h \exp(-\imath\theta)}
   {h^2 + | 1 + d \exp(\imath\theta)|^2} d\theta. 
\end{split}
\end{equation} 
\end{theorem} 
By \cite{GH2003}, the convergence of 
$(2N)^{-1} \tr Q(z,\eta)$ in~\eqref{cvg-res} implies that 
$\check\nu_{n,z} \Rightarrow \check{\bs\nu}_z$ (symmetric) a.s., 
with Stieltjes transform $\gamma^{-1} p(z,\cdot)$. 
The system of equations~\eqref{hdt} which provides the values of $p(z,\cdot)$
on the positive imaginary axis completely determines the measure 
$\check{\bs\nu}_z$. The function $d(z,\imath t)$ will be used below to 
identify the limit measure $\bs\mu$.

\subsubsection{Step 2: uniform integrability}\label{subsub-uniforminteg} 
It is equivalent to show the uniform integrability of $\log|\cdot|$ with
respect to $(\check\nu_{n,z})$. Note that  $\log$  is unbounded near both $0$
and $\infty$.  The following proposition will address uniform integrability
near zero. 

\begin{proposition}
\label{prop-ui}
Let $(M_n)$ be a sequence of random matrices such that 
$M_n \in \CC^{n\times n}$. Let 
$\bs M_n = \begin{bmatrix} & M_n \\ M_n^* \end{bmatrix}$, and assume that 
there exist three constants $\alpha, \beta, C > 0$ such that 
\begin{equation}
\label{wegner} 
\frac{-\imath}{n} \EE \tr \left( \bs M_n - \imath t \right)^{-1} 
  \leq C ( 1 + t^{-\alpha} n^{-\beta} ) . 
\end{equation} 
Assume in addition that there exist a sequence of events $(\cE_n)$ such that
$\PP[ \cE^{\text{c}}_n ] \to 0$, and two constants $\gamma, \tau > 0$ such 
that for large enough $n$, 
\begin{equation}
\label{psn} 
\PP\left[ [s_{n-1}( M_n ) \leq n^{-\gamma}] \cap \cE_n \right] 
  = \cO(n^{-\tau}). 
\end{equation} 
 Then, denoting as $\check\sigma_n$ the empirical 
singular value distribution of $\bs M_n$, 
\begin{equation}
\label{ui0} 
\forall\varepsilon > 0, \quad
\lim_{\delta \to 0}\limsup_{n\geq 1} \PP \left[ \Bigl| \int_{-\delta}^\delta \log|\lambda| \ 
  \check\sigma_{n}(d\lambda) \Bigr| > \varepsilon \right]=0. 
 \end{equation} 
\end{proposition} 

For a detailed proof, please refer 
to~\cite[Proposition~14]{gui-kri-zei-11} or
to~\cite[Section~6.2]{coo-hac-naj-ren-(arxiv)16}.  We  
just point out that starting from~\eqref{wegner} and
making some elementary Stieltjes transform calculations, one can show that
there exist constants $K, \rho > 0$ such that 
$\EE \check\sigma_n([-x,x]) \leq K(x \vee n^{-\rho})$.  This so-called 
\textit{local Wegner estimate} \cite{weg-81} on the
number of \textit{intermediate singular values}, used in conjunction with the
control provided by~\eqref{psn} on the smallest singular value, leads
to~\eqref{ui0}. 

The validity of condition (\ref{wegner}) in Proposition \ref{prop-ui} is guaranteed by the 
next proposition. It is proven in Section~\ref{prf:wegner}. 
The $n^{-1/2}$ rate can be improved but is adequate for
our purposes. 

\begin{proposition}
\label{weg-xjx}
Let Assumption~\ref{ass-model} hold true, and assume that $z\neq 0$. Then, 
there exist two constants 
$\alpha, C > 0$ such that 
\[
\frac{-\imath}{N} \EE \tr Q(z,\imath t) \leq C ( 1 + t^{-\alpha} n^{-1/2} ) . 
\]
\end{proposition}

Condition \eqref{psn} on the smallest singular value of $Y - z$ in 
Proposition~\ref{prop-ui} is a consequence of the following corollary to 
Theorem~\ref{snA}, whose proof is immediate. 
\begin{corollary}[Corollary to Theorem~\ref{snA}]  
\label{sn}
Let Assumptions~\ref{ass-model} and~\ref{Xcomp} hold true. Let $C$ be a 
positive constant. Then, there exist
$\alpha, \beta > 0$ such that for each $z \in \CC \setminus \{ 0 \}$, 
\begin{equation}
\label{sn(Y-z)} 
\PP\left[ s_{N-1}(Y - z) \leq t, \ \| X \| \leq C \right] 
 \leq c \left( n^\alpha t^{1/2} + n^{-\beta} \right) + \exp(-c'n) , 
\end{equation} 
where the constants $c, c' > 0$ depend on $C$, $z$, and $\mom$ only. 
\end{corollary} 

Invoking the boundedness of the fourth moment specified by 
Assumption~\ref{ass-model}, we know from~\cite{yin-bai-kri-88} that 
\begin{equation}
\label{yin-bai} 
\| X \| \toaslong 1 + \sqrt{\gamma} . 
\end{equation} 
Thus, the probability of the event $[ \| X \| \leq C ]$ 
converges to $1$ by setting $C = 2 + \sqrt{\gamma}$. 
By Proposition~\ref{weg-xjx} and Corollary~\ref{sn}, the assumptions of 
Proposition~\ref{prop-ui} are satisfied for $M_n = Y - z$ with $z \neq 0$.
Therefore, the uniform integrability of the $\log|\cdot |$ near zero specified 
by~\eqref{ui0} is true when $\check\sigma_n$ is replaced with 
$\check\nu_{n,z}$, and $z \neq 0$. 

\begin{remark} 
\label{rm:proba} 
The proof of Theorem~\ref{snA} showed that we can take $\beta = 1/22$. Recall
Remark~\ref{rem:beta} in Section~\ref{subsec:endthsnA} above for more comments
on this point.  The consequent slow rate at the right hand side
of~\eqref{sn(Y-z)} is the primary reason that we conclude the uniform
integrability of the $\log|\cdot|$ near zero only in probability. The
convergence in probability stated in  Theorem~\ref{th:main} is due to this.
\end{remark} 

To be able to apply Proposition~\ref{prop-herm}, it only remains to 
establish the uniform integrability of $\log|\cdot |$ near infinity, namely
\[
\forall\varepsilon > 0, \quad
\lim_{T\to \infty}\limsup_{n\geq 1} \PP \left[ \Bigl| \int_{|\lambda|\geq T} \log|\lambda| \ 
  \check\nu_{n,z}(d\lambda) \Bigr| > \varepsilon \right]=0. 
 \] 
But this result follows immediately from the identity  
$\1_{[\| X\| \leq C]} 
  \int_{|\lambda|\geq T} \log|\lambda| \, \check\nu_{n,z}(d\lambda) 
 = 0$, valid for $T > ( 1 + \sqrt{\gamma} )^2 + |z| + 1$. 

\subsection{Identification of $\bs\mu$} 
\label{subsec-ident-mu} 

At this point, we know that there exists a probability measure $\bs\mu$ such
that $\mu_n \Rightarrow \bs\mu$ in probability, and such that 
\[
U_{\bs\mu}(z) = - \int \log|\lambda | \ \check{\bs\nu}_z(d\lambda) 
 \quad \CC - \text{a.e.}
\]

We now aim  
to identify $\bs\mu$ and establish its properties that are specified in 
Theorem~\ref{th:main}. To that end, we rely on equation 
\eqref{mudelta}. 
We use an idea that dates back to~\cite{fei-zee-97} and 
that has been frequently used in the literature devoted to large non-Hermitian 
matrices. Define on $\CC \times (0,\infty)$ the 
\textit{regularized versions} of 
$U_{\mu_n}(z)$ and $U_{\bs\mu}(z)$ respectively: 
\begin{align*} 
\cU_n(z,t) &= 
- \frac{1}{2N} \log\det ( (Y-z)^* (Y-z) + t^2 ), 
  \quad \text{and} \\ 
{\bs\cU}(z,t) &= 
  - \frac 12 \int \log(\lambda^2 + t^2) \ \check{\bs\nu}_{z}(d\lambda)\,  . 
\end{align*} 
In parallel, let us get back to the resolvent $Q(z,\eta)$ defined 
in~\eqref{Qeta}. By Jacobi's formula,
\[
\partial_{\bar{z}} \cU_n(z,t) = \frac{1}{2N} \tr (Y-z) 
        ( (Y-z)^* (Y-z) + t^2 )^{-1}  = 
 \frac{1}{2N} \tr Q_{01}(z,\imath t) . 
\]
Letting $n\to\infty$ we know from Theorem~\ref{narrowcvg} that 
$\partial_{\bar{z}} \cU_n(z,t) \to (2\gamma)^{-1} d(z, \imath t)$ a.s. 
At the same time, $\cU_n(z,t) \to {\bs\cU}(z,t)$ a.s.~since 
$\nu_{n,z} \Rightarrow \bs\nu_z$. We can therefore assert that 
$\partial_{\bar{z}} {\bs\cU}(z,t) = (2\gamma)^{-1} d(z, \imath t)$ in 
$\cD'(\CC)$, and then extract the properties of $\bs\mu$ from the equation 
\[
\bs\mu = - \frac{1}{2\pi} \Delta U_{\bs\mu} = 
- \frac{2}{\pi} \lim_{t\to 0} 
    \partial_z \partial_{\bar{z}} {\bs\cU}(z,t) 
 = - \frac{1}{\gamma\pi} \lim_{t\to 0} \partial_z d(z, \imath t) 
\quad \text{in} \ \cD'(\CC). 
\]
This line of thought leads to the following proposition. Its proof is 
provided in Section~\ref{prf:d->U'}.

\begin{proposition} 
\label{d->U'} 
As $t\to 0$, the function 
$(2\gamma)^{-1} d(\cdot,\imath t)$ converges  
to $\partial_{\bar{z}} U_{\bs\mu}(\cdot)$ in $\cD'(\CC)$. 
\end{proposition}

The following lemma specifies the
properties of the function $g$ defined in (\ref{functiong}), 
that we shall need. Its
proof is straight-forward and is omitted. 

\begin{lemma}
\label{prop-g} 
Consider the function $g$ on the interval $[0\vee (\gamma-1), \gamma]$. It is 
analytical and increasing on $(0\vee (\gamma-1), \gamma)$. 
Moreover, $g(0 \vee (\gamma-1)) = 0 \vee (\gamma-1)^3 / \gamma$, and 
$g(\gamma) = \gamma(\gamma+1)$. 
\end{lemma} 

By this lemma, $g$ has an inverse $g^{-1}$ on 
$[ 0 \vee (\gamma-1)^3 / \gamma, \gamma(\gamma+1) ]$ that takes this interval 
to $[ 0 \vee (\gamma - 1) , \gamma ]$. On 
$( 0 \vee (\gamma-1)^3 / \gamma, \gamma(\gamma+1) )$, the function 
$g^{-1}$ is analytical and increasing. 

By showing that $d(z, \imath t)$ converges as $t\to 0$ point-wise for each 
$z\neq 0$ and by identifying the limit function $b(z)$, we get the following 
proposition whose proof is given in Section~\ref{prf:d->b}. 

\begin{proposition}
\label{d->b} 
Let $b(z)$ be the function defined on $\CC\setminus \{ 0 \}$ as follows: \\
If $\gamma \leq 1$, then
\[
b(z) = \left\{\begin{array}{cl} 
  - g^{-1}(|z|^2) / \bar z & \ \text{if} \ \ 0 < |z| \leq \sqrt{\gamma(\gamma+1)}, \\ \\
 - \gamma / \bar z &\ \text{if} \ \ |z| \geq \sqrt{\gamma(\gamma+1)}. 
 \end{array}\right.
\]
If $\gamma > 1$, then 
\[
b(z) = \left\{\begin{array}{cl} 
 - (\gamma-1) / \bar z & \ \text{if} \ \ 0 < |z| \leq  (\gamma-1)^{3/2} / \sqrt{\gamma}, \\ \\
 - g^{-1}(|z|^2) / \bar z &\ \text{if} \ \ (\gamma-1)^{3/2} / \sqrt{\gamma} \leq |z| \leq  \sqrt{\gamma(\gamma+1)},\\ \\
b(z) = - \gamma / \bar z & \ \text{if}\ \ |z| \geq \sqrt{\gamma(\gamma+1)}. 
 \end{array}\right.
\]
Then $\partial_{\bar{z}} U_{\bs\mu}(z) = (2\gamma)^{-1} b(z)$ in $\cD'(\CC)$. 
\end{proposition}

By Lemma~\ref{prop-g}, $b(z)=b(u + \imath v)$ defined in the 
statement of Proposition~\ref{d->b} is continuously differentiable as a 
function of $u,v$ on the open set 
\[S = \{ z \in \CC \, : \, z \neq 0, \, |z|^2 \neq (\gamma - 1)^3/\gamma, \, 
 |z|^2 \neq \gamma(\gamma+1) \}.\] Therefore, 
$\Delta U_{\bs\mu} = 4 \partial_z \partial_{\bar z} U_{\bs\mu}$ coincides with 
$2\gamma^{-1} \partial_z b$ in $\cD'(S)$, where $\partial_z b$ is the pointwise 
derivative of $b$ w.r.t.~$z$. Specifically, for each test function
$\varphi \in C_{\text{c}}^\infty(S)$, we have 
\[
\int_\CC \varphi \, d\bs\mu = 
- \frac{1}{2\pi} \int_\CC \varphi(z) \Delta U_{\bs\mu}(z) \ dz = 
- \frac{1}{\gamma\pi} \int_\CC \varphi(z) \partial_z b(z) \ dz = 
\int_\CC \varphi(z) f(z) \ dz,  
\]
where, by Proposition~\ref{d->b}, the density $f(z)$ of $\bs\mu$ on $S$ is given by 
\begin{equation}
f(z) = \left\{ \begin{array}{cl} 
\displaystyle{\frac{1}{\gamma\pi} \partial_z \frac{g^{-1}(|z|^2)}{\bar z} = 
  \frac{1}{\gamma\pi} \partial_{|z|^2} g^{-1}(|z|^2)} & \text{if} \ 
  \ 0 \vee ((\gamma - 1)^3/\gamma) < |z|^2 < \gamma(\gamma+1), \\ \\
  0 & \ \ \ \text{elsewhere}
	\label{densityf} 
\end{array} 
\right. 
\end{equation}
Hence the density $f$ depends on $z$ through $|z|$ only,
and thus $\bs\mu$ is rotationally invariant on $S$.  

Now we consider $\mu$ on the boundary $\partial S$. 
We deal separately with the cases $\gamma\leq 1$ and 
$\gamma > 1$. 

First suppose $\gamma\leq 1$. 
Let $0 < s < r < \sqrt{\gamma(\gamma+1)}$. Changing to polar co-ordinates, we get 
\begin{align*} 
\bs\mu(\{ z : |z| \in [s,r]\}) &= 
 \frac{1}{\gamma\pi} \int_{\{ z : |z| \in [s,r]\}} 
         \partial_{|z|^2} g^{-1}(|z|^2) \ dz  \\
 &= \frac{1}{\gamma} \frac{1}{2\pi} \int_0^{2\pi} d\theta 
   \int_s^r 2\rho \, \partial_{\rho^2} g^{-1}(\rho^2) \ d\rho \\  
 &= \gamma^{-1} g^{-1}(r^2) - \gamma^{-1} g^{-1}(s^2).\label{mumeasure} 
\end{align*} 
But since $\gamma^{-1} g^{-1}(0) = 0$ and 
$\gamma^{-1} g^{-1}(\gamma(\gamma+1)) = 1$, we get that 
\[\bs\mu(\{ 0 \}) = 
\bs\mu(\{z : |z| = \sqrt{\gamma(\gamma+1)}\}) = 0.
\]
establishing the formula in Theorem~\ref{th:main} for $\gamma\leq 1$. 

Now suppose $\gamma > 1$.
Put 
 $a = (\gamma-1)^{3/2} / \sqrt{\gamma}$. 

If we set 
$0 < s < r < a$, we obtain from (\ref{densityf})
that $\bs\mu(\{ z : |z| \in [s,r]\}) = 0$. 

If 
$a < s < r < \sqrt{\gamma(\gamma+1)}$, then 
$\bs\mu(\{ z : |z| \in [s,r]\}) =  
 \gamma^{-1} g^{-1}(r^2) - \gamma^{-1} g^{-1}(s^2)$ 
by the same derivation as for $\gamma\leq 1$. 

Now we claim that  
$\bs\mu(\{ z : |z| = a \}) = 0$. To show this, let $\phi : [-1,1] \to[0,1]$ be a smooth 
function such that $\phi(0) = 1$ and $\phi(-1)=\phi(1)=0$. Given $\varepsilon > 0$, define the 
$\CC \to [0,1]$ function 
\[
\psi_\varepsilon(z) = \phi\left(\frac{|z| - a}{\varepsilon}\right) 
\]
which is supported on the ring $\{ z : a - \varepsilon \leq |z| 
  \leq a + \varepsilon \}$. It is then enough to show that $\int \psi_\varepsilon  \, d\bs\mu \to 0$ as $\varepsilon \to 0$. 
Indeed, by an integration by parts, we get that 
\[
\int \psi_\varepsilon  \, d\bs\mu = 
 - \frac{1}{\gamma\pi} \int \psi_\varepsilon(z) \, \partial_z b(z) \ dz 
 = \frac{1}{\gamma\pi} \int \partial_z \psi_\varepsilon(z) \, b(z) \ dz 
= \frac{1}{2\varepsilon\gamma\pi} \int \frac{1}{|z|} 
  \phi'\left(\frac{|z| - a}{\varepsilon}\right) c(|z|) \ dz , 
\]
where the function $c(\rho) = \bar z b(z)$ for $\rho = |z|$ is a real bounded 
function near $\rho = a$ that satisfies $c(a) = 1 - \gamma$ by 
Proposition~\ref{d->b}. Making a cartesian to polar variable change, we get 
that 
\begin{align*} 
\int \psi_\varepsilon  \, d\bs\mu &= \frac{1}{2\varepsilon\gamma\pi} 
 \int_0^{2\pi} d\theta \int 
  \phi'\left(\frac{\rho - a}{\varepsilon}\right) c(\rho) \ d\rho 
 = \frac{1}{\gamma} \int_{-1}^1 \phi'(u) c(\varepsilon u + a) \ du \\
&\xrightarrow[\varepsilon\to 0]{} 
\frac{1-\gamma}{\gamma} \left( \phi(1) - \phi(-1)\right) = 0 
\end{align*} 
by the dominated convergence theorem. 

Since 
$\gamma^{-1} g^{-1}(a^2) = 1 - \gamma^{-1}$, we can infer now that 
$\bs\mu( \{ z : s \leq |z| \leq r \}) = \gamma^{-1} g^{-1}(r^2) - 
  (1-\gamma^{-1})$ for each $s \in (0,a)$ and each 
$r \in [a, \sqrt{\gamma(\gamma+1)})$. Letting $s\downarrow 0$ and 
$r\uparrow \sqrt{\gamma(\gamma+1)}$, and recalling that 
$g^{-1}(\gamma(\gamma+1)) = \gamma$, we get that 
\[
\bs\mu( \{ z : |z| < \sqrt{\gamma(\gamma+1)} \}) = 
  1 - (1-\gamma^{-1}) + \bs\mu(\{ 0 \}) .
\]
Similarly to $\bs\mu(\{ z : |z| = a \}) = 0$, we can show that 
$\bs\mu(\{ z : |z| = \sqrt{\gamma(\gamma+1)} \}) = 0$. We therefore get that 
$\mu(\{ 0 \}) = 1 - \gamma^{-1}$, and hence 
the formula in
Theorem~\ref{th:main} is verified also for $\gamma > 1$. 

This completes the proof of Theorem~\ref{th:main}.

\section{Limit singular value distribution}
\label{sec:lsd} 

Given $(z,\eta) \in\CC \times \CC_+$, $\alpha\in\RR$, and a sequence 
$(a_n(z,\eta))_n$ of complex numbers, the notation $a_n = \cO_\eta(n^\alpha)$ 
(or $a_n = \cO_t(n^\alpha)$ when $\eta = \imath t$) 
will refer in this section to the existence of a constant $C > 0$ and two
non-negative integers $k$ and $\ell$ such that 
\[ 
|a_n(z,\eta) | \leq \frac{C |\eta|^k}{(\Im\eta)^\ell} n^\alpha .
\]
The constants $C$, $k$, and $\ell$ may depend on $z$ but not on $\eta$ or  
$n$. If $a_n(z,\eta)$ is a matrix, then the notations 
$a_n = \cO_\eta(n^\alpha)$ and $a_n = \cO_t(n^\alpha)$, are to be understood in a uniform entry-wise sense. 

\subsection{Proof of Theorem~\ref{narrowcvg}}  
\label{prf:narrow} 

We start by showing that for each $z\in\CC$, the bulk  behavior of the
singular values of $Y - z$ is completely specified by
Assumption~\ref{ass-model} and does not depend on the particular distribution
of the elements of $X$. 

Our first result is a standard concentration result which helps to replace the
traces by their expectations. The proof uses well-known methods, see
\cite{bai-sil-book}, and we omit it. 

\begin{proposition} 
\label{prop:conc}
Under Assumption~\ref{ass-model}, for each $(z,\eta) \in \CC \times \CC_+$, 
\[
\frac 1n \begin{bmatrix} \tr Q_{00}(z,\eta) &  \tr Q_{01}(z,\eta) \\
  \tr Q_{10}(z,\eta) &  \tr Q_{11}(z,\eta) \end{bmatrix} 
 - 
\frac 1n \begin{bmatrix} \tr \EE Q_{00}(z,\eta) &  \tr \EE Q_{01}(z,\eta) \\
  \tr \EE Q_{10}(z,\eta) &  \tr \EE Q_{11}(z,\eta) \end{bmatrix} 
 \toaslong 0. 
\]
\end{proposition} 

The above expectations are easier to compute when the entries are Gaussian. 
The next result establishes that we can assume this without any loss. 
Let $x^\cN = (U + \imath V) / \sqrt{2n}$, where $U$ and $V$ are real 
independent standard Gaussian random variables. 
Define 
$X^\cN = \begin{bmatrix} x^\cN_{ij} \end{bmatrix}_{i,j=0}^{N-1,n-1}$,
where the $x^\cN_{ij}$ are independent copies of $x^\cN$. 
Clearly, these entries satisfy 
Assumption~\ref{ass-model}.
Let $Q^\cN_{ij}(z,\eta)$ be the analogues of the $Q_{ij}(z,\eta)$, obtained by replacing the matrix $X$ with $X^\cN$. The proof of 
the following proposition proceeds along standard lines and uses the 
boundedness of the fourth moment in Assumption \ref{ass-model} and the fact that the first two moments of the two sets of variables agree.
We omit the details.  
\begin{proposition} 
\label{prop:gauss}
Under Assumption~\ref{ass-model}, for each $(z,\eta) \in \CC \times \CC_+$, 
\[
\frac 1n \begin{bmatrix} \tr \EE Q_{00}(z,\eta) &  \tr \EE Q_{01}(z,\eta) \\
  \tr \EE Q_{10}(z,\eta) &  \tr \EE Q_{11}(z,\eta) \end{bmatrix} 
 - 
\frac 1n \begin{bmatrix} \tr \EE Q_{00}^\cN(z,\eta) &  
        \tr \EE Q_{01}^\cN(z,\eta) \\
  \tr \EE Q_{10}^\cN(z,\eta) &  \tr \EE Q_{11}^\cN(z,\eta) \end{bmatrix} 
 = \cO_\eta(n^{-1/2}).  
\]
\end{proposition} 
\textit{Hence, in the rest of this subsection
we assume that the elements of
$X$ are distributed as $x^\cN$}. 
This enables us to study 
$n^{-1} \tr \EE Q_{ij}(z,\eta)$ with the help of two 
Gaussian tools that are frequently used in random matrix theory. The first 
is the Integration by Parts (IP) 
formula \cite{gli-jaf-(livre87)}, \cite{kho-pas-93}, and the second is the 
Poincaré-Nash (PN) inequality \cite{cha-bos-04}, \cite{pas-05}, which is also
a particular case of the Brascamp-Lieb inequality. 
A detailed account of the use of these tools in random matrix theory can
be found in the treatise~\cite{pas-livre}. 

Let $w = [ w_0, \ldots, w_{n-1} ]^\T$ be a complex Gaussian
random vector with  
$\EE w = 0$, $\EE w w^\T = 0$, and $\EE [ w w^* ] = \Xi$. Let 
$\varphi = \varphi(w_0, \ldots, w_{n-1}, \bar w_0, \ldots, \bar w_{n-1} )$ 
be a $C^1$ complex function which is polynomially bounded together with
its derivatives. Then, the \textit{IP formula} reads as 
\begin{equation}\label{ipformula}
\EE w_k \varphi( w ) = 
\sum_{\ell = 0}^{n-1} \left[ \Xi \right]_{k \ell} 
\EE \left[ \frac{\partial \varphi(w)}{\partial \bar w_{\ell}} \right].
\end{equation}

\noindent Furthermore, writing 
 \[\nabla_w\varphi = [ \partial\varphi / \partial w_0, \ldots, 
 \partial\varphi / \partial w_{n-1} ]^\T \ \ \text{and}\ \ 
 \nabla_{\bar{w}}\varphi = [ \partial\varphi / 
 \partial \bar w_0, \ldots, \partial\varphi / \partial \bar w_{n-1} ]^\T,\] 
the \textit{PN inequality} is 
\begin{equation}\label{pninequality}
 \var\left( {\varphi}( w) \right) \leq 
 \EE \left[ \nabla_w \varphi( w )^T \ \Xi \ 
 \overline{\nabla_w \varphi( w )} 
 \right] 
 +
 \EE \left[ \left( \nabla_{\bar{w}}  \varphi( w ) \right)^* 
 \ \Xi \ \nabla_{\bar{w}} \varphi(w) 
 \right].
\end{equation}

We shall apply (\ref{ipformula}) to the case  $w \equiv X$ and 
$\varphi \equiv u^* Q v$ where $Q=Q(z, \eta)$ is the resolvent given by Eq.~\eqref{Qeta} 
(seen as a function of $X$), and $u$ and $v$ are deterministic vectors in 
$\CC^{2N}$. 
If we disregard $(z,\eta)$ and write $Q(z,\eta) = Q^X$ to emphasize the
dependence of the resolvent on $X$, then, given
a matrix $\Delta \in \CC^{N\times n}$, the resolvent identity implies that
\[
Q^{X+\Delta} - Q^X = - Q^{X+\Delta} 
\begin{bmatrix} & (X+\Delta) J (X+\Delta)^* - X JX^* \\
(X+\Delta) J^{-1} (X+\Delta)^* - X J^{-1} X^* \end{bmatrix} Q^X . 
\]
Using this equation, we can obtain the expression of 
$\partial u^* Q v / \partial \bar x_{ij}$, where $i\in [N]$ and $j \in [n]$. 
Taking $\Delta = e_{N,i} e_{n,j}^*$ we get that
\begin{align*} 
\frac{\partial u^* Q v}{\partial \bar x_{ij}} 
 &= - u^* Q \begin{bmatrix}  & X J e_{n,j} e^*_{N,i} \\ 
  X J^{-1} e_{n,j} e_{N,i}^* \end{bmatrix} Q v . 
\end{align*} 
In particular, by taking $u = e_{2N,k}$ and $v = e_{2N,\ell}$ 
for $k,\ell \in [N]$, we obtain from these equations that 
\begin{equation}\label{firstpartial}
\frac{\partial [ Q_{00} ]_{k,\ell}}{\partial \bar x_{ij}} = 
 - [ Q_{01} X J^{-1} ]_{kj} [Q_{00}]_{i \ell} 
 - [ Q_{00} X J ]_{kj} [Q_{10} ]_{i\ell}, 
\end{equation}
and by taking $u = e_{2N,k}$ and $v = e_{2N, N+\ell}$ for $k,\ell \in [N]$, 
we get 
\begin{equation}\label{secondpartial}
\frac{\partial [ Q_{01} ]_{k,\ell}}{\partial \bar x_{ij}} = 
 - [ Q_{00} X J ]_{kj} [Q_{11}]_{i \ell} 
 - [ Q_{01} X J^{-1} ]_{kj} [Q_{01} ]_{i\ell}. 
\end{equation}
Given $M \in \CC^{n\times n}$, we shall also use the trivial relations 
\[
[ MJ^k ]_{\cdot, j} = [M]_{\cdot, j+k} \quad \text{and} \quad
[ J^k M ]_{i,\cdot} = [M]_{i-k,\cdot} ,  
\]
where both the sum $j+k$ and the difference $i-k$ are taken modulo-$n$. 

We can now start our calculations. Recalling that $x_j$ refers to the 
$j^{\text{th}}$ column of $X$ for $j \in [n]$, our first task is to study 
quadratic forms of the type $x_k^* Q_{00} x_\ell$ and $x_k^* Q_{01} x_\ell$. 
Define the matrices 
\[
A_{00} = \EE \begin{bmatrix} x_k^* Q_{00} x_\ell \end{bmatrix}_{k,\ell=0}^{n-1}
\quad \text{and} \quad 
A_{01} 
  = \EE \begin{bmatrix} x_k^* Q_{01} x_\ell \end{bmatrix}_{k,\ell=0}^{n-1} .
\]
It is obvious that $X \stackrel{{\mathcal L}}{=} X J^m$ for each $m\in \ZZ$.
Thus, given a measurable function $f : \CC^{N\times N} \to \CC^{N\times N}$ 
and the integers $k,\ell,m \in [n]$, it holds that 
\begin{align*} 
x_{k+m}^* f(XJX^*) x_{\ell+m} &= e_{n,k+m}^* X^* f(XJX^*) X e_{n,\ell+m} \\
&= e_{n,k}^* J^{-m} X^* f(X J^m J J^{-m} X^*) X J_m e_{n,\ell}   \\ 
&\stackrel{{\mathcal L}}{=} x_{k}^* f(XJX^*) x_{\ell}, 
\end{align*} 
where the index summations are taken modulo-$n$. As a consequence, \textit{the matrices
$A_{00}$ and $A_{01}$ are 
circulant matrices}, a fact very useful to us. 

Starting with $A_{00}$, we have by the IP formula (\ref{ipformula}), 
\begin{align*}
\EE x_k^* Q_{00} x_\ell &= \sum_{i,j=0}^{N-1} 
 \EE [(\bar x_{ik} [ Q_{00}]_{ij}) x_{j\ell}] 
= \frac 1n \sum_{i,j} \EE \left[ \frac{\partial ( \bar x_{ik} [ Q_{00}]_{ij})}
 {\partial \bar x_{j\ell}} \right] \\
 &= \frac 1n \sum_{i,j} \1_{i=j} \1_{k=\ell} \EE [ Q_{00}]_{ij}  \\
 &\phantom{=} 
  - \frac 1n \sum_{i,j} 
   \EE [ Q_{01} X J^{-1} ]_{i\ell} \bar x_{ik} [Q_{00}]_{jj} 
  - \frac 1n \sum_{i,j} 
   \EE [ Q_{00} X J ]_{i\ell} \bar x_{ik} [Q_{10} ]_{jj}  \ \text{(using \ref{firstpartial})}\\
&= \1_{k=\ell} \EE \tr Q_{00} / n 
 - \EE \left[ [ X^* Q_{01} X J^{-1} ]_{k\ell} \tr Q_{00} / n \right] 
 - \EE \left[ [ X^* Q_{00} X J ]_{k\ell} \tr Q_{10} / n \right] . 
\end{align*} 
We also have 
\begin{align*}
\EE x_k^* Q_{01} x_\ell &= \sum_{i,j=0}^{N-1}  
 \EE [(\bar x_{ik} [ Q_{01}]_{ij}) x_{j\ell}] 
= \frac 1n \sum_{i,j} \EE \left[ \frac{\partial ( \bar x_{ik} [ Q_{01}]_{ij})}
 {\partial \bar x_{j\ell}} \right] \ \text{(using \ref{secondpartial})}\\
%\text{[\wh{skip next two lines?}]}\\
 %&= \frac 1n \sum_{i,j} \1_{i=j} \1_{k=\ell} \EE [ Q_{01}]_{ij}  \\
 %&\phantom{=} 
  %- \frac 1n \sum_{i,j} 
  %\EE [ Q_{00} X J ]_{i\ell} \bar x_{ik} [Q_{11}]_{jj} 
  %- \frac 1n \sum_{i,j} 
  %\EE [ Q_{01} X J^{-1} ]_{i\ell} \bar x_{ik} [Q_{01} ]_{jj}  \\ 
%% 
 &= \1_{k=\ell} \EE[ \tr Q_{01} / n ] 
 - \EE\left[ [ X^* Q_{00} X J ]_{k\ell} \tr Q_{11} / n \right] 
 - \EE\left[ [ X^* Q_{01} X J^{-1} ]_{k\ell} \tr Q_{01} / n \right] . 
\end{align*} 

In right side of the above two expressions 
we have terms of the type
$\EE\left[ \left[ \cdots \right]_{k\ell} \tr Q_{ij} / n \right]$. We can use 
the PN inequality (\ref{pninequality}) to decouple $\left[ \cdots \right]_{k\ell}$ from 
$\tr Q_{ij} / n$. Specifically, we have the following lemma, which is proven in Appendix~\ref{prf-lm:np}.

\begin{lemma}
\label{lm:np} 
For each $i,j \in \{0,1\}$ and each $k,\ell \in [n]$, 
\[
\var \left( \tr Q_{ij} /n \right) = \cO_\eta(n^{-2}) \quad \text{and} \quad
\var \left( x_k^* Q_{ij} x_\ell \right) = \cO_\eta(n^{-1}) . 
\]
\end{lemma} 

Let us write $q_{ij}=q_{ij}(z,\eta) = n^{-1} 
 \EE \tr Q_{ij}(z,\eta)$ for $i,j \in \{0,1\}$.  
Using the lemma, and applying the Cauchy-Schwartz inequality, 
it is easy to see that 
\begin{align*} 
\EE x_k^* Q_{00} x_\ell &= 
 \1_{k=\ell} q_{00} 
 - \EE \left[ [ X^* Q_{01} X J^{-1} ]_{k\ell} \right] q_{00} 
 - \EE \left[ [ X^* Q_{00} X J ]_{k\ell} \right] q_{10} 
  + {\mathcal O}_\eta(n^{-3/2}), \ \text{and} \\
\EE x_k^* Q_{01} x_\ell &= 
  \1_{k=\ell} q_{01} 
 - \EE\left[ [ X^* Q_{00} X J ]_{k\ell} \right] q_{11} 
 - \EE\left[ [ X^* Q_{01} X J^{-1} ]_{k\ell} \right] q_{01} 
  + {\mathcal O}_\eta(n^{-3/2}).  
\end{align*} 

Since $Y$ is a square matrix, we see from~\eqref{Qeta} 
that $q_{00} = q_{11}$. Thus, the equations above can be written in a matrix 
form as 
\begin{align} 
A_{00} (I_n +  q_{10} J ) + q_{00} A_{01} J^{-1} &= q_{00} I_n + 
           {\mathcal O}_\eta(n^{-3/2}), \ \text{and} \label{A00} \\ 
q_{00} A_{00} J + A_{01} ( I_n + q_{01} J^{-1}) &= q_{01} I_n + 
          {\mathcal O}_\eta(n^{-3/2})  .  \label{A01} 
\end{align}
Let us give these equations a more symmetric form. Developing 
$\eqref{A01} \times q_{00} J^{-1}- \eqref{A00} \times (I + q_{01} J^{-1})$, 
we get that 
\begin{equation}
\label{eq:A00} 
A_{00} \left[ q_{00}^2 - (I_n +  q_{10} J ) (I_n + q_{01} J^{-1}) \right] 
= - q_{00} + {\mathcal O}_\eta(n^{-3/2}) .  
\end{equation} 
Similarly, taking $\eqref{A00} \times q_{00} J - \eqref{A01} \times 
(I + q_{10} J)$, we get 
\begin{equation}
\label{eq:A01} 
A_{01} \left[ q_{00}^2 - (I_n +  q_{10} J ) (I_n + q_{01} J^{-1}) \right]
= q_{00}^2 J - q_{01} (I_n + q_{10} J) + {\mathcal O}_\eta(n^{-3/2}) .  
\end{equation} 

Now, by using the obvious identity $Q (\Sigma - \eta) = I_{2N}$ we obtain 
\begin{align*} 
 - \eta Q_{00} - \bar z Q_{01} + Q_{01} X J^{-1} X^* &= I_N, \\ 
 - z Q_{00} - \eta Q_{01} + Q_{00} X JX^* &= 0, 
\end{align*} 
(the similar equations involving the terms $Q_{10}$ and $Q_{11}$ will not be
used). Taking the traces of the expectations, we get 
\begin{align} 
 - \eta q_{00} - \bar z q_{01} + n^{-1} \tr A_{01} J^{-1} &= \gamma_n, 
 \label{QS1} \\ 
 - z q_{00} - \eta q_{01} + n^{-1} \tr A_{00} J &= 0,  \label{QS2} 
\end{align} 
where $\gamma_n = N/n$. 

Recalling that $q^{(n)}_{00}(z,\eta) = n^{-1}\EE \tr Q^{(n)}(z,\eta)$, the 
function $\gamma_n^{-1} q^{(n)}_{00}(z,\cdot)$ is the Stieltjes transform of 
the probability measure $\EE\check \nu_{n,z}$. Hence,  
$| \gamma_n^{-1} q^{(n)}_{00}(z,\eta) | \leq 1 / \Im\eta$. So,  
$\{ q^{(n)}_{00}(z,\cdot) \}_{n\in\NN}$ is a \textit{normal family} of holomorphic functions on $\CC_+$. Similarly, 
$q^{(n)}_{01}(z,\cdot) = n^{-1} \EE \tr Q^{(n)}_{01}(z,\cdot)$ and 
$q^{(n)}_{10}(z,\eta) = n^{-1} \EE \tr Q^{(n)}_{10}(z,\eta)$ are holomorphic
functions in $\eta\in\CC_+$ whose absolute values are bounded by 
$\sup_n \gamma_n / \Im\eta$. 

Using the normal family theorem, let us extract from the sequence $(n)$ a 
subsequence (still denoted as $(n)$) such that $q_{00}^{(n)}(z,\cdot)$, 
$q_{01}^{(n)}(z,\cdot)$, and $q_{10}^{(n)}(z,\cdot)$ converge to holomorphic 
functions in the sense of uniform convergence on the compact subsets of $\CC_+$. 
Denote these functions respectively as $p(z,\cdot)$, $d(z,\cdot)$ and $\tilde d(z,\cdot)$. We shall show that they  uniquely solve a system of
equations on the line segment $\imath [C, \infty)$ of the positive imaginary 
axis, where $C$ is some positive
constant. This will show that $p(z,\cdot)$ is uniquely defined on $\CC_+$, and
that $q_{00}(z,\cdot) \to_n p(z,\cdot)$ and $q_{01}(z,\cdot) \to_n d(z,\cdot)$ 
on $\CC_+$. We then show 
that $t \Im p(z, \imath t) \to \gamma$ as $t\to\infty$. This will lead to the 
fact that $\gamma^{-1} p(z,\cdot)$ is the Stieltjes transform of a 
symmetric probability measure $\check{\bs\nu}_{z}$. 

Assume that $\eta = \imath t$ where $t > 0$. Then, since the measure $\EE\check \nu_{n,z}$
is symmetric, $q_{00}(z, \imath t) = \imath s(z, t)$ with $s(z,t) > 0$. Moreover, 
we notice from the expressions of $Q_{01}$ and $Q_{10}$ in~\eqref{Qeta} that
$q_{10}(z, \imath t) = \bar q_{01}(z, \imath t)$. 

Recall that $A_{00}$ and $A_{01}$ are circulant matrices. 
Define the so-called Fourier matrix 
\[
{\sf F}_n = 
n^{-1/2} \left[ \exp(2\imath \pi k\ell/n) \right]_{k,\ell=0}^{n-1}. 
\]
Then the circulant matrix $J$ can be written as 
\[
J = {\sf F}_n \diag( \exp(-2\imath\pi k / n) )_{k=0}^{n-1} {\sf F}_n^* , 
\]
Notice that the matrices $A_{00}$, $A_{01}$ and $J$ commute, since they are 
circulant. 

Now Equation~\eqref{eq:A00} can be rewritten as 
$A_{00} P = \imath s + E$ where $E = \cO_t(n^{-3/2})$ is a circulant matrix, 
and 
\begin{equation}
\label{def:P} 
P = s^2  + (I_n +  \bar q_{01} J ) (I_n + q_{01} J^*) 
 = {\sf F}_n 
 \diag\left( s^2 + \left| 1 + q_{01} \exp(2\imath \ell / n) 
        \right|^2 \right)_{\ell=0}^{n-1} {\sf F}_n^* . 
\end{equation} 
If $t \geq 2 \sup_n\gamma_n$, then $|q_{01}| \leq 1/2$, and thus, 
the positive definite matrix $P$ satisfies $P \geq (1/4) I_n$ in the 
semi-definite positive ordering. In view of Equation~\eqref{QS2}, we need an 
expression for $n^{-1} \tr A_{00} J$. We can write 
\begin{equation}
\label{A00J} 
\frac{\tr A_{00} J}{n} 
= \frac{\imath s \tr P^{-1} J}{n} + \frac{\tr P^{-1} J E}{n}  
 = \frac{\imath s}{n} \sum_{\ell=0}^{n-1} 
  \frac{\exp(-2\imath\pi \ell/n)}
   {s^2 + |1 + q_{01} \exp(2\imath\pi\ell/n) |^2} 
 + \frac{\tr P^{-1} J E}{n}.  
\end{equation} 
Given two square matrices $M_1$ and $M_2$ of the same size, it is well known 
that $|\tr M_1 M_2 | \leq (\tr M_1 M_1^*)^{1/2} (\tr M_2 M_2^*)^{1/2}$. Thus, 
since $E = \cO_t(n^{-3/2})$, we get that
\[
\frac{|\tr P^{-1} J E|}{n} \leq \frac 1n \sqrt{\tr P^{-2}} 
 \sqrt{\tr E E^*} \leq \frac 1n 2n^{1/2} \cO_t(n^{-1/2}) = \cO_t(n^{-1}). 
\]
By a similar derivation, and in view of Equation~\eqref{QS1}, we also get from 
Equation~\eqref{eq:A01} that 
\begin{equation}
\label{A01J} 
\frac{\tr A_{01} J^{-1}}{n} 
 = \frac 1n \sum_{\ell=0}^{n-1} 
  \frac{s^2 + |q_{01}|^2 + q_{01} \exp(2\imath\pi \ell/n)}
   {s^2 + |1 + q_{01} \exp(2\imath\pi\ell/n) |^2} 
   + {\mathcal O}_t(n^{-1}) . 
\end{equation} 

Now, taking $n$ to infinity along the subsequence $(n)$ in
Equations~\eqref{QS1}, \eqref{QS2}, \eqref{A00J}, and \eqref{A01J}, writing  
$p(\imath t) = \imath h(z,t)$ where $h(z,t) \geq 0$, and noting that 
$\tilde d(z,\imath t) = \bar d(z,\imath t)$, the pair 
$(h(z,t), d(z,\imath t))$ satisfies the system of Equations~\eqref{hdt} of the
statement of Theorem~\ref{narrowcvg} for $t \geq 2 \sup_n\gamma_n$. 

Let us consider the system of equations in $(h,d) \in (0,\infty) \times \CC$ 
\begin{subequations} 
\label{hd} 
\begin{align} 
 - t h + \bar z d  &= u(h,d) -  \gamma,  \label{hdu} \\ 
  z h + t d &= v(h,d),  \label{hdv}
\end{align} 
\end{subequations} 
where $u(h,d)$ and $v(h,d)$ are given by Equations~\eqref{uvt}. 
Writing 
\[
{\bs I}(a, u) = \frac{1}{2\pi} \int_0^{2\pi} 
  \frac{1}{a^2 + |1 + u \exp(\imath\theta) |^2} d\theta 
\quad \text{and} \quad 
{\bs J}(a,u) = 
\frac{1}{2\pi} \int_0^{2\pi} 
  \frac{\exp(\imath \theta)}
   {a^2 + |1 + u \exp(\imath\theta) |^2} d\theta , 
\]
the system~\eqref{hd} can be rewritten as 
\begin{subequations}
\label{hd'} 
\begin{align} 
-t h + \bar z d &= (h^2 + |d|^2) {\bs I}(h, d) + d {\bs J}(h,d) - \gamma,  
 \label{hdu'} \\
z h + t d &= h \overline{\bs J}(h,d) . \label{hdv2}
\end{align} 
\end{subequations}  
By using the residue theorem (derivations omitted), we know that the integrals 
are given by the expressions 
\begin{equation}
\label{IJ} 
{\bs I}(a,u) = 
 \frac{1}{\sqrt{(a^2 + |u|^2 + 1)^2 - 4 |u|^2}} , 
\quad \text{and} \quad 
{\bs J}(a,u) = 
  \frac{1}{2u} \left( 1 - 
 \frac{a^2 + |u|^2 +1}{\sqrt{(a^2 + |u|^2 + 1)^2 - 4 |u|^2}} \right) 
\end{equation} 
for each $a \in \RR$ and $u\in\CC$ such that $a \neq 0$ or $|u| \neq 1$. 

\begin{lemma}
There exists $C > 0$ (depending on $z$ and $\gamma$) such that for each $t\in [C, \infty)$, the 
system~\eqref{hd} has a unique solution 
$(h, d)$ such that $h \in (0, \gamma/t)$ and $|d| < \gamma/t$.
\end{lemma} 
\begin{proof}
Equation~\eqref{hdv2} can be written equivalently as $\bar z d = d {\bs J}(h,d)
- t |d|^2 / h$. Note that $d {\bs J}(h,d)$ and ${\bs I}(h,d)$ are both real and
depend on $d$ through $|d|$ only. Thus, $\bar z d \in \RR$, and if we
write $z = \rho\exp(\imath\theta)$, then $(h,d)$ is a solution for this $z$ if
and only if $(h, d')$ is a solution for $z = \rho$, where 
$d' = d \exp(-\imath\theta) \in \RR$. Consequently, we can assume without loss of  
generality that $z$ and $d$ belong to $\RR$ in the system~\eqref{hd}. 

This system can be written equivalently as $[h,d]^\T = f([h,d]^\T)$, where
\[
f\left(\begin{bmatrix} h \\ d \end{bmatrix}\right) = 
\frac{1}{t^2 + z^2} \begin{bmatrix} -t & z \\ z & t \end{bmatrix} 
 \begin{bmatrix} u(h,d) - \gamma \\ v(h,d) \end{bmatrix} .
\]
We shall show if $C$ is large, $f$ is a Banach contraction on the space 
$S = [0, 2\gamma/t] \times [-\gamma/t, \gamma/t]$. 

If $C$ is large enough, we get from the integral expressions of $u(h,d)$ 
and $v(h,d)$ that
\[
|u(h,d)|  \leq \frac{2\gamma}{t} , \quad \text{and} \quad 
| v(h,d)| 
%\leq \frac{2 \gamma t^{-1}}{(1 - \gamma t^{-1})^2} 
\leq 
  \frac{3\gamma}{t}. 
\] 
Hence, writing $f([h,d]^\T) = [ f_1, f_2 ]^\T$, (and recalling that $z$ is real), we get that 
$\gamma t \geq t | u(h,d,t) | + | z v(h,d,t)|$, thus $f_1 \geq 0$, and 
moreover,
\[
f_1 \leq \frac{2 \gamma + t\gamma + 3 |z| \gamma t^{-1}}{t^2 + z^2} 
 \leq \frac{2\gamma}{t} 
 \quad \text{and} \quad 
 | f_2| \leq  \frac{2 \gamma |z| t^{-1} + |z|\gamma + 3 \gamma}{t^2 + z^2} 
 \leq \frac \gamma t 
\]
for large enough $C$. Thus, $f([h,d]^\T) \in S$ when $[h,d]^\T \in S$. 

We now consider the Jacobian matrix $\jac(f)$ of $f$. After some easy 
derivations that we omit, we obtain that when $C$ is large, there exist a 
constant $C' > 0$ such that 
\[
\left| \frac{\partial u}{\partial h} \right| \leq \frac{C'}{t}, \ 
\left| \frac{\partial u}{\partial d} \right| \leq C', \ 
\left| \frac{\partial v}{\partial h} \right| \leq C', \ \text{and} \ 
\left| \frac{\partial v}{\partial d} \right| \leq \frac{C'}{t} \ \ \text{(on} \ S).
\] 
 Since 
\[
\jac(f) = 
\frac{1}{t^2 + z^2} \begin{bmatrix} -t & z \\ z & t \end{bmatrix} 
 \begin{bmatrix} 
\partial u / \partial h & \partial u / \partial d \\  
\partial v / \partial h & \partial v / \partial d 
\end{bmatrix} , 
\]
we get that $\| \jac(f) \| \leq 1/2$ on $S$ for large enough $C$, and the 
result follows from Banach's fixed point theorem. 
\end{proof} 

\begin{lemma}
$t h(z,t) \to \gamma$ as $t\to\infty$. 
\end{lemma}
\begin{proof} 
The functions $h(z,t)$ and $d(t)$ satisfy Equation~\eqref{hdut}, and furthermore,
$0\leq h(z,t), |d(t)| \leq \gamma/t$. From the expressions~\eqref{IJ}, it is 
clear that $(h^2+d^2){\bs I}(h(z,t), d(t))$ and $d(t) {\bs J}(h(z,t), d(t))$ converge to
zero as $t\to\infty$. The result is then obtained from Equation~\eqref{hdu'}. 
\end{proof}

We now need to prove that 
$(h(z,\imath t), d(z, \imath t))$ satisfy the system of Equations~\eqref{hdt}
for each $t > 0$. By the convergence $q_{00}(z,\cdot) \to p(z,\cdot)$, we get
that $\EE\check\nu_{n,z} \Rightarrow \check{\bs\nu}_z$. In particular, 
$\EE\check\nu_{n,z}$ is tight. Let $a > 0$ be such that 
$\inf_n \EE\check\nu_{n,z}([-a,a]) \geq 1/2$. Recalling that 
$\EE\check\nu_{n,z}$ is symmetric, we have 
\begin{align*} 
s(z,\imath t) &= -\imath q_{00}(z, \imath t) = 
-\imath \gamma_n \int \frac{1}{\lambda - \imath t} 
   \EE\check{\nu}_{n,z}(d\lambda) \\ 
 &= \frac{-\imath\gamma_n}{2}  
\int \Bigl(\frac{1}{\lambda - \imath t} + \frac{1}{-\lambda - \imath t} \Bigr)
 \EE \check{\nu}_{n,z}(d\lambda) 
=  \int \frac{\gamma_n t}{\lambda^2 + t^2} \EE \check{\nu}_{n,z}(d\lambda) \\
 &\geq \frac{\gamma_n}{2} \frac{t}{a^2 + t^2} . 
\end{align*} 
Therefore, for each $t > 0$, the matrix $P$ defined in \eqref{def:P} 
satisfies $P \geq \gamma t / (4(a^2 + t^2)) I$ in the semidefinite ordering for
all large enough $n$. By repeating the argument that follows Equation~\eqref{def:P}, 
we obtain that $(p(z,\imath t), d(z,\imath t))$ solve the system~\eqref{hdt}. 

To complete the proof of Theorem~\ref{narrowcvg}, it remains to show that
$\check\nu_{n,z} \Rightarrow \bs\nu_z$ almost surely. This is obtained at once
by combining Propositions~\ref{prop:conc} and~\ref{prop:gauss} with the 
convergence of $q_{00}(z,\cdot)$ to $p(z,\cdot)$.

\subsection{Proof of Proposition~\ref{weg-xjx}}  
\label{prf:wegner} 

We first assume that $X \stackrel{\cL}{=} X^\cN$, where
$X^\cN$ was defined before the statement of Proposition~\ref{prop:gauss}. 
Fixing $z\neq 0$, and writing $q_{00}(z, \imath t) = \imath s$, we first show that there exist constants $\alpha, C > 0$ such that 
\begin{equation}
\label{spoly}
s \in (0, C ( 1 + n^{-1} t^{-\alpha})] \quad \text{for} \quad t \in (0, 1]. 
\end{equation}   
From Equations~\eqref{QS1}, \eqref{QS2}, \eqref{A00J}, and \eqref{A01J}, we 
get that 
\begin{subequations}
\label{sd01} 
\begin{align}
-t s + \bar z q_{01} &= u_n(t) - \gamma_n + {\mathcal O}_t(n^{-1}), 
 \label{sd01a} \\
z s + t q_{01} &= v_n + {\mathcal O}_t(n^{-1}) , \label{sd01b} 
\end{align} 
\end{subequations} 
where 
\[
u_n = \frac 1n \sum_{\ell=0}^{n-1} 
  \frac{s^2 + |q_{01}|^2 + q_{01} \exp(2\imath\pi \ell/n)}
   {s^2 + |1 + q_{01} \exp(2\imath\pi\ell/n) |^2}, 
\quad \text{and} \quad 
v_n = \frac{1}{n} \sum_{\ell=0}^{n-1} 
  \frac{s \exp(-2\imath\pi \ell/n)}
   {s^2 + |1 + q_{01} \exp(2\imath\pi\ell/n) |^2} . 
\]
We now show that $|q_{01}| \leq C_1 +  {\mathcal O}_t(n^{-1})$ for some constant $C_1 > 0$. 
It is enough to focus on the case $|q_{01}| \geq 2$. Using
\[
|u_n | \leq \frac{s^2 + |q_{01}|^2 + |q_{01}|}{s^2 + ( |q_{01}| - 1)^2}, 
\]
it is easy to see that 
$|u_n| \leq 5$. Since 
$s \leq \sup_n \gamma_n/t$, we get from Equation~\eqref{sd01a} that
$|q_{01}| \leq C_1 +  {\mathcal O}_t(n^{-1})$ for some constant $C_1 > 0$.
 
Now using the obvious inequality $|v_n | \leq 1 / s$ along with 
Equation~\eqref{sd01b}, we get that 
$|z| s \leq 1/s + C_1 + {\mathcal O}_t(n^{-1})$. Thus, $s^2 \leq C_2(1 + s + s t^{-\alpha} n^{-1})$ for some $\alpha, C_2 > 0$. 
If $s \geq 2C_2$, we have that 
$0.5 s^2 \leq s^2 - C_2s \leq C_2(1 + s t^{-\alpha} n^{-1}) 
\leq s (1/2 + C_2 t^{-\alpha} n^{-1})$, thus, 
$s \leq 1 + 2 C_2 t^{-\alpha} n^{-1}$, and~\eqref{spoly} is established. 

Removing the Gaussian assumption on the elements of $X$, 
Proposition~\ref{weg-xjx} is obtained by combining
Proposition~\ref{prop:gauss} with~\eqref{spoly}. 

\section{Properties of $\bs\mu$} 
\label{sec:prop-mu} 

\subsection{Proof of Proposition~\ref{d->U'}}  
\label{prf:d->U'} 

We first show that $\bs\cU(\cdot,t)$ is continuous, and that there is a 
probability one set on which 
\begin{equation}
\label{cvgU} 
\cU_n(\cdot,t) \ \xrightarrow[n\to\infty]{\cD'(\CC)} \ \bs\cU(\cdot,t) \ \ \text{for each}\ \ t > 0.
\end{equation} 
 Fix $t > 0$. 
From the almost sure weak convergence of $\check\nu_{n,z}$ to 
$\check{\bs\nu}_z$ and~\eqref{yin-bai}, we get that  
$\cU_n(z,t) \toaslong \bs\cU(z,t)$ for each $z \in \CC$. 
Furthermore, by the Hoffman-Wielandt theorem (see \cite{HorJoh90}), given $z,z' \in \CC$, we have 
$\max_i | s_i(Y - z) - s_i(Y-z') | \leq | z - z' |$. Thus, $|\cU_n(z,t) - \cU_n(z',t) | \leq 
| z - z' | / (2 t^2)$, and taking $n$ to infinity, we get that 
$\bs\cU(\cdot,t)$ is continuous. 

Let $\cK$ be a compact set of $\CC$. For $z\in \cK$, we have  
\begin{equation}
\label{Ubnd} 
|\cU_n(z,t)| \leq \frac{1}{2N} \sum \log(1 + s_i(Y-z)^2 / t^2 ) 
 + | \log t | 
 \leq \frac{(\| X \|^2 + \max_{z\in\cK} |z|)^2}{2 t^2} + |\log t| .  
\end{equation} 
On the underlying probability space $\Omega$, let $\cE_0$ be the probability 
one event where the convergence~\eqref{yin-bai} takes place. 
For each $z\in \cK$, define the event 
\[
\cE(z) = [ \cU_n(z,t) \to_n \bs\cU(z,t) ]. 
\]
Since $\cU_n(\cdot,t)$ is measurable on the product space $\Omega\times \CC$ 
and $\bs\cU(\cdot,t)$ is continuous, the function $\1_{\cE(z)}(\omega,z)$ is 
measurable on $\Omega\times \CC$. Moreover, for each $z\in\CC$, 
\[
\int_\Omega (1 - \1_{\cE(z)}(\omega,z)) \ \PP(d\omega) = 0 .
\]
By the Fubini-Tonelli theorem, there exists a probability one set 
$S\subset\Omega$ such that 
\[
\int_\cK (1 - \1_{\cE(z)}(\omega,z)) \ dz = 0 \ \ \text{for each}\ \ \omega \in S.
\]
 Let $\varphi \in C_{\text{c}}^\infty(\CC)$ be supported by $\cK$, and let 
$\omega \in S \cap \cE_0$. We have  
\begin{align*} 
\Bigl| \int \varphi(z) (\cU_n(z,t) - \bs\cU(z,t)) \ dz \Bigr| &\leq 
\| \varphi \|_\infty \int_\cK  \left| \cU_n(z,t) - \bs\cU(z,t) \right| 
  \1_{\cE(z)}(\omega,z) \ dz  \\
&\phantom{=} 
 + \| \varphi \|_\infty \int_\cK  \left| \cU_n(z,t) - \bs\cU(z,t) \right| 
  (1 - \1_{\cE(z)}(\omega,z)) \ dz . 
\end{align*} 
By~\eqref{Ubnd} and the continuity of $\bs\cU(\cdot,t)$, the term 
$\left| \cU_n(z,t) - \bs\cU(z,t) \right|$ is bounded. Thus, the first
term at the right hand side converges to zero by the dominated convergence, 
while the second is zero for large enough $n$. This proves~\eqref{cvgU}. 

Equation~\eqref{cvgU} implies that, 
$\partial_{\bar{z}} \cU_n(\cdot,t) \xrightarrow[n\to\infty]{\cD'(\CC)} 
 \partial_{\bar{z}} \bs\cU(\cdot,t)$ almost surely. On the other hand, we know 
from Jacobi's formula that the pointwise derivative of $\cU_n(z,t)$ with 
respect to $\bar z$ is $(2N)^{-1} \tr Q_{01}(z,\imath t)$. Moreover, this 
derivative coincides with the distributional derivative 
$\partial_{\bar z} \cU_n(z,t)$. 
By Theorem~\ref{narrowcvg}, $(2N)^{-1} \tr Q_{01}(z,\imath t) \toaslong 
(2\gamma)^{-1} d(z, \imath t)$ for each $z\in \CC$. By an argument similar to 
the one used in the proof of~\eqref{cvgU}, we can show that this convergence 
holds almost surely in $\cD'(\CC)$. Thus, 
\begin{equation}
\label{U'=d}
\partial_{\bar{z}} \bs\cU(z,t) \ = \ (2\gamma)^{-1} d(z,\imath t) 
\ \text{in} \ \cD'(\CC)\ \ \text{for each} \ \ t>0.  
\end{equation}

We now show that 
\begin{equation}
\label{conv:U->Umu} 
\bs\cU(\cdot,t) \ \xrightarrow{\cD'(\CC)} \ U_{\bs\mu} \ \ \text{as}\ \ t\downarrow 0.
\end{equation}
It is clear from the expressions of $\bs\cU(z,t)$ and $U_{\bs\mu}(z)$ that 
$\bs\cU(z,t) \uparrow U_{\bs\mu}(z)$ as $t\downarrow 0$.  Therefore, 
$0 \leq \bs\cU(z,t) - \bs\cU(z,t_0) \leq \cU_{\bs\mu}(z) - \bs\cU(z,t_0)$
for $0 < t \leq t_0$, and since $\bs\cU(\cdot,t_0)$ is continuous hence 
locally integrable, we get~\eqref{conv:U->Umu} by the monotone convergence 
theorem.    

Proposition~\ref{d->U'} follows from~\eqref{U'=d} and~\eqref{conv:U->Umu}.

\subsection{Proof of Proposition~\ref{d->b}} 
\label{prf:d->b} 

The following preliminary lemma is needed. 

\begin{lemma}
\label{hdbnd} 
For each $z\neq 0$, the function $h(z,t)$ is bounded for $t\in (0,\infty)$. 
Moreover, $|d(z,\imath t)| \leq C / |z|$, where $C$ is a positive constant. 
\end{lemma} 
\begin{proof}
Assume without loss that $|d(t)| \geq 2$. Using
\[
|u| \leq \frac{h^2 + |d|^2 + |d|}{h^2 + ( |d| - 1)^2}, 
\]
it is easily seen that $|u(t)| \leq 5$.
Observing that 
$h(z,t) \leq 1/t$ by the general properties of the Stieltjes transforms, we 
obtain from~\eqref{hdu} that $|d(t)| \leq C / |z|$ for some $C > 0$. 

Using the inequality $|v(t)| \leq 1 / h(z,t)$ along with~\eqref{hdv}, we get 
that
$h(z,t)^2 \leq |z|^{-1} ( t h(z,t) |d(t)| + 1)$ which shows that $h(z,t)$ is 
bounded when $z\neq 0$. 
\end{proof} 

We now enter the proof of Proposition~\ref{d->b}. Since $\check{\bs\nu}_z$ 
is symmetric, 
\[
h(z,\imath t) = -\imath p(z, \imath t) = 
-\imath \gamma \int \frac{1}{\lambda - \imath t} \check{\bs\nu}_z(d\lambda) 
= 
\frac{-\imath\gamma}{2}  
\int \Bigl(\frac{1}{\lambda - \imath t} + \frac{1}{-\lambda - \imath t} \Bigr)
\check{\bs\nu}_z(d\lambda) 
=  \int \frac{\gamma t}{\lambda^2 + t^2} \check{\bs\nu}_z(d\lambda) , 
\]
thus, on $t\in(0,1]$, the function 
\[
\frac{h(z,t)}{t} \geq 
  \int \frac{\gamma}{\lambda^2 + 1} \check{\bs\nu}_z(d\lambda) 
\]
is lower-bounded by a positive constant. 

In the proof, we shall use the fact that $(h(z,t), d(z, \imath t))$ satisfies 
the system of equations~\eqref{hd'}. We rewrite Equation~\eqref{hdu'}  
as  $\gamma = (h^2 + |d|^2) {\bs I}(h, d) + d {\bs J}(h,d) - \bar z d + t h$, 
and Equation~\eqref{hdv2} as 
$\bar z h + t \bar d = h {\bs J}(h,d)$, or equivalently, as 
$\bar z d = d {\bs J}(h,d) - t |d|^2 / h$.  
Since $h(z,t) > 0$ for $t > 0$, we can use the expressions~\eqref{IJ} of the
integrals ${\bs I}(h, d)$ and ${\bs J}(h,d)$ to obtain 
\begin{equation}
\label{gzt}
\begin{split}
\gamma &= \frac{h^2 + |d|^2}{\sqrt{\Delta(h,d)}} 
  + \frac{t}{h}\left( h^2 + |d|^2 \right) , \\
2 \bar z d &= 1 - \frac{h^2 + |d|^2 + 1}{\sqrt{\Delta(h,d)}} - 
 2 | d|^2 \frac{t}{h}  
 = 1 - \gamma - \frac{\gamma}{h^2 + |d|^2} + 
  \frac{t}{h}\left( h^2 + 1 - |d|^2 \right), 
\end{split} 
\end{equation} 
where $\Delta(h,d) = ( h^2 + |d|^2 + 1)^2 - 4 | d |^2$. 

We now let $t\to 0$. Here, each sequence $t_k \to 0$ satisfies 
one of two cases : either $t_k / h(z,t_k) \to 0$, or $t_k / h(z,t_k) \to \alpha$ 
where $\alpha$ is a positive number. Indeed, we have just shown that 
$t_k / h(z,t_k) \to\infty$ is excluded. 

\paragraph*{Case $t_k / h(z,t_k) \to 0$.} 
Using Lemma~\ref{hdbnd}, and taking a further subsequence that we still denote 
as $(k)$, we can assume that $d(t_k) \to b \in \CC$ and $h(z,t_k) \to r \geq 0$. 
The pair $(r,b)$ satisfies the equations
\begin{align} 
\gamma^2 \Delta(r, b) &= (r^2 + |b|^2)^2,  \quad \text{and} \label{eq-r} \\ 
2 \bar z b &= 1 - \gamma - \frac{\gamma}{r^2 + |b|^2} \label{2bz} . 
\end{align} 
By Equation~\eqref{2bz}, the number $y = - \bar z b$ is real and satisfies 
\begin{equation}
\label{ry} 
r^2 + |b|^2 = r^2 + \frac{y^2}{|z|^2} = \frac{\gamma}{1 - \gamma + 2y}. 
\end{equation}
Moreover, we have  
$\Delta(r,b) = ( (\gamma / (1-\gamma + 2y) + 1)^2 - 4 y^2 / | z |^2$. 
Replacing in~\eqref{eq-r}, we get 
\[
\left( \frac{\gamma}{1 - \gamma + 2y} + 1 \right)^2 - 4\frac{y^2}{|z|^2} 
= \frac{1}{(1-\gamma+2y)^2} .
\]
Reducing to the same denominator, we get after some simple manipulations that
\[
|z|^2 = g(y), 
\]
where $g$ is the function given in the statement of Theorem~\ref{th:main}. 
Let us delineate the domain of variation of $y$. Equation $|z|^2 = g(y) = 
(1-\gamma+2y)^2 y / (y+1)$ shows that $y(y+1) > 0$, thus $y < -1$ or $y > 0$. 
By Equation~\eqref{ry},  
\[
\frac{\gamma}{1 - \gamma + 2y} \geq \frac{y^2}{|z|^2} = 
\frac{y(y+1)}{(1-\gamma+2y)^2} . 
\] 
We therefore get that $2 y + 1 - \gamma > 0$ and furthermore, by 
rearranging the terms of the inequality above, that  
$y^2 + (1-2\gamma) y + \gamma(\gamma-1) \leq 0$. The last 
inequality implies that $\gamma-1\leq y\leq \gamma$. In conclusion, we get that 
$y \in [ 0 \vee (\gamma-1), \gamma ] \setminus \{ 0 \}$. 

\paragraph*{The case $t_k / h(z,t_k) \to \alpha > 0$.} 
Here we get of course that $h(z,t_k) \to 0$. Taking a subsequence if 
necessary, we shall assume that $d(t_k) \to b$. Getting back to the
system~\eqref{gzt} and taking $t_k$ to zero, we get that 
\begin{align*}
\gamma \left| 1 - |b|^2 \right| &= |b|^2 
      + \alpha |b|^2 \left| 1 - |b|^2 \right| , \\
2\bar z b |b|^2 &= (1-\gamma) |b|^2 - \gamma + \alpha |b|^2 ( 1 - |b|^2 ) .
\end{align*}
The first equation implies that $|b| \not\in \{ 0, 1 \}$, and that
\[
\alpha = \frac{\gamma}{|b|^2} - \frac{1}{| 1 - |b|^2 |} .
\]
Replacing $\alpha$ by its value in the second equation, we get after a 
simple calculation that 
\[
2 \bar z b = 1 - 2\gamma - \frac{1-|b|^2}{| 1 - |b|^2|} .
\] 
Here we need to consider two cases: either $|b| < 1$ or $|b| > 1$. If 
$| b | < 1$, we get from the last equation that $b = - \gamma / \bar z$
(thus, $|z| \geq \gamma$). 
Plugging in the expression of $\alpha$, we get that 
\[ 
\alpha = |z|^2 \left( \frac{1}{\gamma}- \frac{1}{|z|^2 - \gamma^2}\right). 
\]
Since $\alpha > 0$, this implies that $|z| > \sqrt{\gamma(\gamma+1)}$. 

If $|b| > 1$, we obtain that $b = (1 - \gamma) / \bar z$, thus, 
$|z| < |1-\gamma|$ and  
\[
\alpha = |z|^2 \left( 
 \frac{\gamma}{(1-\gamma)^2} - \frac{1}{(1-\gamma)^2 - |z|^2} \right). 
\]
Using again that $\alpha > 0$, we get after a small calculation that 
$\gamma > 1$ and $|z|^2 \leq (\gamma-1)^3 / \gamma$. 

Let us summarize our conclusions for clarity. 
\begin{itemize} 
\item If $t_k / h(z, t_k) \to 0$, let $b$ be an arbitrary accumulation point 
 of $d(z, t_k)$, and let $y = - \bar z b$. 
 \begin{itemize} 
 \item If $\gamma \leq 1$, then $y \in (0, \gamma]$, and 
   $|z|^2 = g(y) \in (0, \gamma(\gamma+1)]$. 
 \item If $\gamma > 1$, then $y \in [\gamma-1, \gamma]$, and 
     $|z|^2 = g(y) \in [(\gamma-1)^3/\gamma, \gamma(\gamma+1)]$. 
 \end{itemize} 

\item If $t_k / h(z, t_k)$ converges to a positive constant, let $b$ be an 
  arbitrary accumulation point of $d(z, t_k)$. 
 \begin{itemize} 
 \item If $\gamma \leq 1$, then $|z|^2 > \gamma(\gamma+1)$, and 
   $b = -\gamma/ \bar z$. 
 \item If $\gamma > 1$, thein either $|z|^2 > \gamma(\gamma+1)$ in which case
  $b = -\gamma/ \bar z$, or $|z|^2 < (\gamma-1)^3/\gamma$, in which case 
  $b = (1-\gamma) / \bar z$. 
 \end{itemize} 
\end{itemize} 

These statements show that given $z \neq 0$, the accumulation points $b$ reduce
to a genuine limit. Moreover, the behavior of this limit $b(z)$ is as
described in the statement of Proposition~\ref{d->b}.   

From the point-wise convergence $d(z, \imath t) \to_{t\to 0} b(z)$ for 
$z\neq 0$ and Lemma~\ref{hdbnd}, we get that 
$d(\cdot, \imath t) \to_{t\to 0} b(\cdot)$ in $\cD'(\CC)$. Thus, 
$(2\gamma)^{-1} b(z) = \partial_{\bar{z}} U_{\bs\mu}(z)$ in $\cD'(\CC)$ by 
Proposition~\ref{d->U'}. 

\paragraph{Acknowledgements.} 
The visit of AB to France has been funded by the Indo-French Center for
Applicable Mathematics project \textit{High Dimensional Random Matrix Models
with Applications}. This work has been partially supported by the French 
ANR grant HIDITSA (ANR-17-CE40-0003). The authors would like to thank
Monika Bhattacharjee, Nicholas Cook, and David Renfrew for fruitful 
discussions. 

\appendix

\section{Supplementary proofs} 
\label{anx-prfs}

\subsection{Proof of Proposition~\ref{proj1}} 
\label{prf-proj1}

Given $a > 1$, we have by Markov's inequality that 
$\PP[ |Z_0| \geq a ] \leq 1 / a$. Let $b \in (0, 1 - 1 / a)$. By Hoeffding's 
concentration inequality, we have
\begin{align*}
\PP\left[ \sum_{i=0}^{n-1} \1_{|Z_i|\leq a} \leq nb \right] &= 
\PP\left[ \sum_{i=0}^{n-1} - \1_{|Z_i|\leq a} + n \PP[ |Z_0| \leq a ] 
   \geq n ( \PP[ |Z_0| \leq a ] - b) \right]  \\
&\leq \exp( - 2 n ( \PP[ |Z_0| \leq a ] - b)^2 ) . 
\end{align*} 
Given $\cJ \subset [n]$, let $\mathcal E_\cJ$ be the event 
\[
{\mathcal E}_\cJ = \left[ j \in \cJ \Leftrightarrow  |Z_j|\leq a \right] . 
\]
Then we just showed that 
\[
\PP\Bigl[ \bigcup_{\cJ\subset [n] : |\cJ| > nb} {\mathcal E}_\cJ 
  \Bigr] \geq 1 - \exp( - 2 n ( \PP[ |Z_0| \leq a ] - b)^2 ) .
\]
Let $\cJ \subset [n]$ be such that $|\cJ| > nb$. Assume without loss of generality that $\cJ = [|\cJ|]$. To obtain the result, it is enough to prove that 
\begin{equation}
\label{tal1} 
\PP [ \dist( Z, V) \leq c_1 \sqrt{n} \, | \, {\mathcal E}_\cJ ] 
    \leq \exp(-c_2 n ) , 
\end{equation}
where $c_1, c_2 > 0$ depend on $\kappa$ and $C_\kappa$ only.

Define $Z^a = [ Z_0^a, \ldots, Z_{n-1}^a ]^\T$, where    
the $Z_i^a$ are independent copies of a random variable whose law is the 
distribution of $Z_0$ conditionally on the event $[|Z_0| \leq a]$. 
Recalling that $\Pi_\cJ$ is the orthogonal projection on the subspace of the 
vectors that are supported by $\cJ$, we note that 
$\dist(Z, V) \geq \dist(\Pi_\cJ(Z), \Pi_\cJ(V))$. Then, the 
inequality~\eqref{tal1} will be established if we show that 
\[
\PP\left[ \dist( \Pi_\cJ( Z^a ), \Pi_\cJ(V) ) \leq c_1 \sqrt{n} \right] 
 \leq \exp(-c_2 n ) . 
\]
Write $Z^a = \mathring{Z}^a + \EE Z^a$, and define the subspace 
$W = \colspan (V, \EE Z^a)$. Since $\dist( \Pi_\cJ( Z^a ), \Pi_\cJ(V) ) 
 \geq \dist( \Pi_\cJ( Z^a ), \Pi_\cJ(W) ) 
  = \dist( \Pi_\cJ( \mathring{Z}^a ), \Pi_\cJ(W) )$, the claim can be reduced 
to 
\begin{equation}
\label{tal2} 
\PP\left[ \dist( \Pi_\cJ( \mathring{Z}^a), \Pi_\cJ(W) ) \leq c_1 \sqrt{n} \right]
  \leq \exp(-c_2 n ) .  
\end{equation} 
Consider the disc $D = \{ z \in \CC \, : \, |z| \leq a \}$, and define the 
convex and $1$-Lipschitz function 
\[
f : D^{|\cJ|} \to \RR_+, \quad x \mapsto \dist(x, \Pi_\cJ(W)) .
\]
If we denote as $\mu_a$ the probability law of an element of $Z^a$, 
then~\eqref{tal2} can be re-expressed as 
\begin{equation}
\label{f<=c1} 
\mu_a^{\otimes |\cJ|} (f \leq c_1 \sqrt{n}) \leq \exp(-c_2 n ).
\end{equation} 
We can now make use of Talagrand's concentration inequality, which shows that 
\begin{equation}
\label{talag} 
\mu_a^{\otimes |\cJ|} | f - Mf | \geq a t ) \leq 
          4 \exp\left( - t^2 / 16  \right) , 
\end{equation} 
where $Mf$ is a median of $f$ under $\mu_a^{\otimes |\cJ|}$. This inequality
shows that there exists a constant $C > 0$ such that 
\begin{equation}
\label{EfMf} 
| \EE f - Mf | \leq Ca \quad \text{and} \quad 
  \EE f \geq \sqrt{\EE f^2} - Ca .  
\end{equation} 
Writing $\mathring{Z}^a = (\mathring{Z}^a_1,\ldots, \mathring{Z}^a_n)$,
we now have
\begin{align*} 
\EE f^2 &= \EE \left[ \dist( \Pi_\cJ( \mathring{Z}^a), \Pi_\cJ(W) )^2 \right] 
 = \sum_{j=1}^{|\cJ|} \EE [ (\mathring{Z}_i^a)^2 ] 
                          \left[ \Pi_{\Pi_\cJ(W)^\perp} \right]_{jj}  . 
\end{align*} 
Observe that $\EE [ Z_0^2 \, | \, | Z_0 | \leq a ] = ( 1 - 
\EE [ Z_0^2 \1_{| Z_0 | > a} ] ) / \PP[ |Z_0| \leq a]$. 
From the assumption on the $(2+\kappa)$-th moment, 
$\EE [ Z_0^2 \1_{| Z_0 | > a} ] \leq C_\kappa / a^\kappa$, and hence, 
$\EE [ Z_0^2 \, | \, | Z_0 | \leq a ] \geq 1 - C_\kappa / a^\kappa$.  
We can similarly show that $(\EE [ Z_0 \, | \, | Z_0 | \leq a ])^2 \leq 
2 C_\kappa / a^\kappa$. Thus, 
$\EE [ (\mathring{Z}_0^a)^2 ] \geq 1 - 3 C_\kappa / a^{\kappa}$. Moreover, 
since $\dim(\Pi_\cJ W) \leq \dim(W) \leq \dim(V) + 1$, we have 
\begin{align*}
\sum_{j=1}^{|\cJ|} \left[ \Pi_{\Pi_\cJ(W)^\perp} \right]_{jj}  &= 
\sum_{j=1}^{n} \left[ \Pi_{\Pi_\cJ(W)^\perp} \right]_{jj}  - 
\sum_{j\in \cJ^{\text{c}}} \left[ \Pi_{\Pi_\cJ(W)^\perp} \right]_{jj}  \\
 &\geq  n - \dim(V) - 1 - (n - |\cJ|) \\
 &> nb - \dim(V) - 1 . 
\end{align*} 
Assuming that $\dim(V) < nb$, and using \eqref{EfMf}, this leads to: 
\[
M f \geq \sqrt{nb - \dim(V)} / 2    
\]
for $a$ above a value that depends on $\kappa$ and $C_\kappa$ only. 
Put $t = \sqrt{nb - \dim(V)} / (4a)$. Observing that in this case, 
$f \leq at \Rightarrow | f - Mf | \geq at$, we can apply~\eqref{talag} to 
obtain 
\begin{align*} 
\mu_a^{\otimes |\cJ|} (f \leq \sqrt{nb - \dim(V)} / 4) 
 &\leq 
\mu_a^{\otimes |\cJ|} (| f -Mf | \geq \sqrt{nb - \dim(V)} / 4) \\ 
 &\leq 4 \exp\left( - (nb - \dim(V)) / (2^8 a^2) \right), 
\end{align*} 
and \eqref{f<=c1} follows after bounding $\dim(V)$ and adjusting $c_1$ and $c_2$
in a straightforward manner. This concludes the proof of Proposition~\ref{proj1}.

\subsection{Proof of Lemma~\ref{lm:np}} 
\label{prf-lm:np} 

We start by showing that $\var(n^{-1} \tr Q_{00}) = \cO_\eta(n^{-2})$, the
proof for the other $n^{-1} \tr Q_{ij}$ being similar. 
Applying the NP inequality to the function $\varphi(X) = n^{-1} \tr Q_{00}$, 
we get that 
\begin{equation}
\label{eq:npphi} 
\var \varphi \leq \frac 1n \sum_{i,j=0}^{N-1,n-1} 
 \EE \left| \frac{\partial \varphi}{\partial \bar x_{ij}} \right|^2 
 + 
 \frac 1n \sum_{i,j=0}^{N-1,n-1} 
 \EE \left| \frac{\partial \varphi}{\partial x_{ij}} \right|^2. 
\end{equation} 
We focus on the first term at the right hand side of this inequality, the other
term being treated similarly. Using Equation~\eqref{firstpartial},
\begin{align*}
 \frac{\partial \varphi}{\partial \bar x_{ij}}  &= 
\frac 1n \sum_{k=0}^{N-1} \frac{\partial [Q_{00}]_{kk}}{\partial \bar x_{ij}} 
 = 
- \frac 1n \sum_{k=0}^{N-1} [ Q_{01} X J^{-1} ]_{kj} [Q_{00}]_{ik} + 
  [ Q_{00} X J ]_{kj} [Q_{10} ]_{ik} \\
 &= -\frac 1n \left( [ Q_{00} Q_{01} X J^{-1} ]_{ij} + 
                     [ Q_{10} Q_{00} X J ]_{ij} \right). 
\end{align*} 
Hence, using the inequality $\tr M P \leq \| M \| \tr P$ when the matrix  
$P$ is Hermitian and non-negative, we get  
\begin{align*} 
\frac 1n \sum_{i,j}
 \EE \left| \frac{\partial \varphi}{\partial \bar x_{ij}} \right|^2 
 &\leq \frac{2}{n^3} 
 \EE \tr X^* Q_{01}^* Q_{00}^* Q_{00} Q_{01} X + 
 \EE \tr X^* Q_{00}^* Q_{10}^* Q_{10} Q_{00} X \\ 
 &\leq \frac{4}{(\Im\eta)^4} \frac{1}{n^3} \EE \tr XX^* 
 = \frac{4N}{(\Im\eta)^4n} \times \frac{1}{n^2}.  
\end{align*} 
This shows that $\var(n^{-1} \tr Q_{00}) = \cO_\eta(n^{-2})$. 

We now show that 
$\var \left( x_k^* Q_{00} x_\ell \right) = \cO_\eta(n^{-1})$ 
(proof is similar for other $Q_{ij}$). 
Writing this time $\varphi(X) = x_k^* Q_{00} x_\ell$, we also use the 
inequality~\eqref{eq:npphi} to bound $\var \varphi$. Here we have 
\begin{align*}
 \frac{\partial \varphi}{\partial \bar x_{ij}}  &= 
 \sum_{m,p=0}^{N-1} \frac{\partial \bar x_{mk} [Q_{00}]_{mp} x_{p\ell}}
   {\partial \bar x_{ij}} \\ 
 &= \1_{k=j} [ Q_{00} X]_{i\ell} 
   - [ X^* Q_{01} X J^{-1} ]_{kj} [ Q_{00} X]_{i\ell}  
   - [ X^* Q_{00} X J ]_{kj} [ Q_{10} X]_{i\ell} ,  
\end{align*} 
thus, 
\begin{align} 
\frac 1n \sum_{i,j} 
 \EE\left| \frac{\partial \varphi}{\partial \bar x_{ij}} \right|^2 
 &\leq \frac 3n \left( 
 \EE \| Q_{00} x_\ell \|^2 + 
 \EE \left( x_k^* Q_{01} X X^* Q_{01}^* x_k \, 
                     x_\ell^* Q_{00} Q_{00}^* x_\ell \right) \right. 
    \nonumber \\ 
 & 
 \ \ \ \ \ \ \ \ \ \ \ \ \ \ \ \ 
 \ \ \ \ \ \ \ \ \ \ \ \ \ \ \ \ 
 \left. + 
 \EE \left( x_k^* Q_{00} X X^* Q_{00}^* x_k \, 
                     x_\ell^* Q_{10} Q_{10}^* x_\ell \right) \right) . 
\label{eq:varphiij} 
\end{align} 
We have $\EE \| Q_{00} x_\ell \|^2 \leq (\Im\eta)^{-2} \EE \| x_\ell \|^2 
  = (\Im\eta)^{-2} N / n$. Moreover, 
\[
\EE (x_k^* Q_{01} X X^* Q_{01}^* x_k)^2 \leq 
  \EE\left[ \| Q_{01} X \|^4 \| x_k \|^4 \right] 
 \leq (\Im\eta)^{-4} \left( \EE \| X \|^8 \right)^{1/2} 
   \left( \EE \| x_k \|^8 \right)^{1/2} . 
\]
Using, \textit{e.g.}, \cite[Prop.~2.3]{rud-ver-cpam09}, we get that there
exists a constant $C > 0$ such that $\EE \| X \|^8 \leq C$ (we note here that
\cite[Prop~2.3]{rud-ver-cpam09} can be applied to the Gaussian real case.
Extension to the complex Gaussian case is easy). Thus, 
$\EE (x_k^* Q_{01} X X^* Q_{01}^* x_k)^2 = \cO_\eta(1)$. It is clear that
$\EE (x_\ell^* Q_{00} Q_{00}^* x_\ell)^2 = \cO_\eta(1)$, and hence, 
$\EE \left( x_k^* Q_{01} X X^* Q_{01}^* x_k \, 
                     x_\ell^* Q_{00} Q_{00}^* x_\ell \right) = \cO_\eta(1)$ by 
the Cauchy-Schwarz inequality. 
The third term at the right hand side of Inequality~\eqref{eq:varphiij} 
can be dealt with similarly, which shows that 
$n^{-1} \sum_{i,j} 
 \EE\left| \partial \varphi / \partial \bar x_{ij} \right|^2 
 = \cO_\eta(n^{-1})$. 
The term involving $\partial\varphi / \partial x_{ij}$ at the right hand
side of~\eqref{eq:npphi} can be bounded in a similar manner, leading to 
the bound $\var \left( x_k^* Q_{00} x_\ell \right) = \cO_\eta(n^{-1})$.  
This concludes the proof.

% \bibliographystyle{plain} 
% \bibliography{math}

\def\cprime{$'$} \def\cdprime{$''$} \def\cprime{$'$} \def\cprime{$'$}
  \def\cprime{$'$} \def\cprime{$'$}

\end{document}